\documentclass[
    10pt, %
    a4paper, %
    oneside, %
    ]{amsart}

\usepackage{amssymb,amsmath,amsthm, amsfonts}
\usepackage{lmodern}
\usepackage[a4paper, lmargin=1in, rmargin=1in, bmargin=0.9in, tmargin=0.9in]{geometry}
\usepackage{xcolor}
\usepackage[hyphens, spaces, obeyspaces]{url}
\usepackage{hyperref}
\usepackage{cleveref}
\usepackage{enumitem}
\usepackage{tikz}
\usepackage{tikz-cd}
\usepackage{quiver}
\usepackage[most]{tcolorbox}
\usepackage{thmtools}
\usepackage{thm-restate}
\usepackage{ragged2e}
\usepackage{stickstootext}
\usepackage[T1]{fontenc}

\hypersetup{
pdftitle={Cohomologically smooth morphisms for étale F_p-sheaves in characteristic p},
pdfsubject={Mathematics, Algebraic Geometry, Six functor formalisms},
pdfauthor={Felix Lotter},
pdfkeywords={Algebraic Geometry, Six functor formalisms, cohomologically etale morphisms, cohomologically smooth morphisms, etale sheaves, smooth objects, arXiv, Preprint}
}

\definecolor{linksc}{RGB}{150,50,100}
\definecolor{citec}{RGB}{0,150,120}
\definecolor{urlc}{RGB}{0,150,120}
\hypersetup{
    colorlinks=true,
    linkcolor={linksc},
    citecolor={citec},
    urlcolor={urlc}
}

\newcommand{\mb}[1]{\mathbb #1}
\newcommand{\mc}[1]{\mathcal #1}

\tikzset{
  pic/.is if=tikz@node@is@pic,
  dot corner/.pic={
    \path[pic actions,draw=none,shade=none,-]
      (-6\pgflinewidth,6\pgflinewidth) circle[radius=1.5\pgflinewidth];
    \path[pic actions,fill=none,shade=none,-,line cap=round,line join=round]
      (-7.5\pgflinewidth,0pt) -| (0pt,7.5\pgflinewidth);},
  dot corner'/.code={\arrow[to path={node also[dot corner={#1}](\tikztostart)}]},
  dot corner/.cd, .is choice,
  north west/.style={label={[draw,fill,pic]north west:dot corner}},
  south west/.style={label={[draw,fill,pic,yscale=-1]south west:dot corner}},
  north east/.style={label={[draw,fill,pic,xscale=-1]north east:dot corner}},
  south east/.style={label={[draw,fill,pic,xscale=-1,yscale=-1]south east:dot corner}}}
\tikzcdset{dot corner'/.style={to path={node also[dot corner={#1}](\tikztostart)}}}

\newcommand{\colim}{\varinjlim}
\newcommand{\ilim}{\varprojlim}

\DeclareFontFamily{U}{mathb}{}
\DeclareFontShape{U}{mathb}{m}{n}{
  <-5.5> mathb5
  <5.5-6.5> mathb6
  <6.5-7.5> mathb7
  <7.5-8.5> mathb8
  <8.5-9.5> mathb9
  <9.5-11.5> mathb10
  <11.5-> mathb12
}{}
\DeclareSymbolFont{mathb}{U}{mathb}{m}{n}
\DeclareMathSymbol{\solid}{\mathbin}{mathb}{"0D}%

\newcommand{\blank}{\rule{0.3cm}{0.15mm}}

\theoremstyle{plain}
\newtheorem{theorem}{Theorem}[section]
\newtheorem*{theorem*}{Theorem}
\newtheorem{lemma}[theorem]{Lemma}
\newtheorem{proposition}[theorem]{Proposition}
\newtheorem*{proposition*}{Proposition}
\newtheorem{corollary}[theorem]{Corollary}

\theoremstyle{definition}
\newtheorem{definition}[theorem]{Definition}
\newtheorem*{mquestion}{Main question}
\newtheorem{example}[theorem]{Example}
\newtheorem{notation}[theorem]{Notation}

\newtheorem{remark}[theorem]{Remark}

\theoremstyle{remark}
\newtheorem*{claim}{Claim}
\newtheorem*{subtquestion}{Question}
\newtheorem{pquestion}[theorem]{Question}

\DeclareMathOperator{\Sch}{Sch}
\DeclareMathOperator{\tensor}{\otimes}
\DeclareMathOperator{\Spec}{Spec}
\DeclareMathOperator{\rSpec}{\underline{Spec}}
\DeclareMathOperator{\Hom}{Hom}

\DeclareMathOperator{\iHom}{\underline{Hom}}
\DeclareMathOperator{\im}{im}

\renewcommand{\phi}{\varphi}
\newcommand{\set}[1]{\left\{#1\right\}}

\newcommand{\et}{\mathrm{\acute et}}

\title[Coh. smooth morphisms for {\'e}tale $\mb F_p$-sheaves in characteristic $p$]{Cohomologically smooth morphisms for \\ {\'e}tale $\mb F_p$-sheaves in characteristic $p$}
\author{Felix Lotter}
\date{}

\begin{document}
\begin{abstract}
    We scrutinise the notions of cohomologically smooth morphisms and smooth objects for the six functor formalism of étale $\mb F_p$-sheaves on schemes in characteristic $p$. We show that only cohomologically étale morphisms are cohomologically smooth in this setting. This is complemented by a characterisation of cohomologically étale morphisms in arbitrary characteristic. In fact, we prove that such a morphism is already étale up to universal homeomorphism.
\end{abstract}
\maketitle

{ \hypersetup{hidelinks} \setcounter{tocdepth}{1} \tableofcontents }

\section{Introduction}
\noindent The main goal of this paper is to answer the following question:
\vspace{0.7em}\\
    \textit{Which morphisms admit Poincaré duality for étale $\mb F_p$-sheaves on schemes in characteristic $p$?}
\vspace{0.7em}\\
Let us clarify what we mean by this.
\subsection{The 6-functor formalism}
For a scheme $X$, one can consider the category $\mathrm{Sh}(X, \Lambda)$ of étale sheaves of $\Lambda$-modules on $X$, where $\Lambda$ is some torsion ring. This category is equipped with a tensor product $\tensor$ and an internal $\Hom$ object $\iHom$; for a morphism $f: Y \to X$ there is an induced pullback functor $f^*: \mathrm{Sh}(X, \Lambda) \to \mathrm{Sh}(Y, \Lambda)$ which admits a right adjoint $f_*: \mathrm{Sh}(Y, \Lambda) \to \mathrm{Sh}(X, \Lambda)$. Deriving all of these functors yields four of the six operations: The tensor product $\tensor^L$, the internal $\Hom$ $R\iHom$, the pullback functor $f^*: \mc D(X,\Lambda) \to \mc D(Y,\Lambda)$ and its right adjoint $Rf_*: \mc D(Y,\Lambda) \to \mc D(X,\Lambda)$. The remaining two are special in that they do not arise from deriving a functor on the abelian level and are not even defined for all morphisms: If $f$ is a separated, finite type morphism of qcqs schemes, one defines a functor $f_!: \mc D(Y,\Lambda) \to \mc D(X,\Lambda)$ and again it admits a right adjoint, $f^!: \mc D(X,\Lambda) \to \mc D(Y,\Lambda)$, called the exceptional pullback. These $6$ functors are subject to certain relations, e.g. the projection formula and base change formula. Roughly, all of this data is what constitutes an abstract 6-functor formalism; but for a long time this term was an unprecise heuristic. However, a concise definition was found recently by L. Mann in \cite{Mann}, vastly streamlining the previous work of Liu-Zheng (\cite{LiuZheng}) and Gaitsgory-Rozenblyum (\cite{gaitsgory}).
\subsection{Cohomologically smooth morphisms} Using Mann's definition of an abstract six functor formalism, one can give meaning to the question if a morphism admits Poincaré duality by employing the notion of \textit{cohomologically smooth} morphisms. Cohomological smoothness was first defined in the context of étale cohomology of diamonds by P. Scholze in \cite{Scholze2} to capture the cohomological properties that are expected from "smooth" morphisms. This notion makes sense in any abstract 6-functor formalism. Roughly, a cohomologically smooth morphism is a morphism $f$ for which the functor $f^!$ right adjoint to $f_!$ is given by twisted pullback (in a universal way). If one considers the étale $6$-functor formalism of $\mb F_\ell$-sheaves on the category of separated, finite type schemes over a field $k = \bar k$ with $\ell \in k^\times$, all smooth morphisms are cohomologically smooth and this amounts indeed to Poincaré duality in étale cohomology (cf. \Cref{xmpl:coh smooth example}).\\
\\
Let us now make our question more precise:

\begin{mquestion}
    What are the cohomologically smooth morphisms for the 6-functor formalism of étale $\mb F_p$-sheaves on finite type, separated schemes over a field of characteristic $p$?
\end{mquestion}
In this setting, it is clear that Poincaré duality can not be expected from smooth morphisms in a naive way. Consider for example the projective line $\mb P^1_k$ over an algebraically closed field $k$ of characteristic $p$. Then one computes $R\Gamma(\mb P^1_k, \mb F_p) \cong \mb F_p[0]$ via the Artin-Schreier sequence, but Poincaré duality would predict that $H^0(\mb P^1_k, \mb F_p)^\lor \cong  H^2(\mb P^1_k, \mb F_p)$. In fact, we will see that the observation that $\mb P^1_k$ has no higher $\mb F_p$-cohomology even implies that $\mb P^1_k \to \Spec k$ can not be cohomologically smooth (\Cref{p1n-not-smooth}). Using the strong stability properties of cohomologically smooth morphisms, we will deduce from this that the cohomologically smooth morphisms must be quasi-finite. It turns out that this is enough to give a full geometric characterisation of them, answering our main question:

\begin{theorem*}[\ref{thm:char}]\label{introduction-1}
    Let $k$ be a field of characteristic $p$ and let $\mc D_p$ be the 6-functor formalism of étale $\mb F_p$-sheaves\footnote{See \Cref{thm:et 6 functors} for the precise definition.} on the category $\mc C$ of finite type, separated schemes over $k$. Let $f:Y \to X$ be a morphism in $\mc C$. Then the following are equivalent:
    \begin{enumerate}[label=\roman*)]
        \item $f$ is $\mc D_p$-cohomologically smooth.
        \item $f$ is $\mc D_p$-cohomologically étale.
        \item $f^{\mathrm{perf}}$ is étale.
    \end{enumerate}
    where $f^\mathrm{perf}$ denotes the perfection of $f$ (see \Cref{prop:perf and awn}).
\end{theorem*}

\subsection{Cohomologically étale morphisms}

In the above theorem the notion of a \textit{cohomologically étale} morphism appeared, see \Cref{def:coh et prop}. Here, a cohomologically étale morphism is just a morphism $f$ for which $f_!$ is "universally" left adjoint to $f^*$; in other words, it is a cohomologically smooth morphism for which the twisted pullback $f^!$ is just the pullback $f^*$ (\Cref{lmm:coh etale vs coh smooth}).\par
We mentioned that it is enough to prove that cohomologically smooth morphisms are quasi-finite to deduce the characterisation above. This is due to the following more general description of cohomologically étale morphisms:

\begin{theorem*}[\ref{cor:coh sm quasifinite coh etale}]
    Let $\Lambda$ be a finite ring and $\mc D_\Lambda$ the 6-functor formalism of étale $\Lambda$-sheaves on the category $\Sch$ of qcqs schemes. Let $f:Y \to X$ be a morphism of finite type in $\Sch$. Then the following are equivalent:
    \begin{enumerate}[label=\roman*)]
        \item $f$ is $\mc D_\Lambda$-cohomologically smooth and quasi-finite.
        \item $f$ is $\mc D_\Lambda$-cohomologically étale.
        \item $f^\textrm{awn}$ is étale.
    \end{enumerate}
    where $f^\mathrm{awn}$ denotes the absolute weak normalisation of $f$ (see \Cref{lmm:awn-adjoint}).
\end{theorem*}
The notion of absolutely weakly normal schemes was introduced by D. Rydh in \cite{rydh}. In positive characteristic it agrees with the notion of perfect schemes (i.e.\ schemes on which the Frobenius endomorphism is an isomorphism) and the absolute weak normalisation is just the perfection (cf. \Cref{prop:perf and awn}).\par
While the equivalence of $i)$ and $ii)$ in the above theorem follows rather formally, we will deduce the equivalence of $ii)$ and $iii)$ using the descent results proven in \cite{arc} by B. Bhatt and A. Mathew. More precisely, we will first observe that cohomologically étale morphisms are universally open and \textit{topologically unramified} (i.e. their diagonal is topologically open, see \Cref{def:top unram}) and then prove that each universally open, topologically unramified morphism in $\mc C$ is étale after absolute weak normalisation (\Cref{thm:qf open unram awn etale}), using descent for étale morphisms along universally open covers on the absolutely weakly normal site while doing induction on the cardinality of geometric fibers.

\subsection{Smooth objects}
Finally, we address a generalisation of our main question:
\vspace{0.7em}\\
\textit{If $f:Y \to X$ is a morphism in $\mc C$, for which objects $A, B \in \mc D(Y)$ is $f_!(A \tensor \blank)$ "universally" left adjoint to $f^* \blank \tensor B$ in the étale $6$-functor formalism of $\mb F_p$-sheaves?}
\vspace{0.7em}\\
That is, we want to understand the $f$-smooth objects (see \Cref{def:sm obj}). In the absolute case, i.e. for $f:X \to \Spec k$, we will characterise them as the scyscraper sheaves on $X$ which are dualisable on their support (\Cref{prop:sm object char absolute}). We can deduce the following statement about the relative case from this:

\begin{theorem*}[\ref{cor:supp of sm coh et}]
    Assume $k = \bar{k}$ and let $f:Y \to X$ be a morphism in $\mc C$ and $A$ an $f$-smooth object. Let $Z$ be the support of $A$. Then $A|_Z$ is $f|_Z$-smooth and $Z \to X$ is quasi-finite.
\end{theorem*}

It is a nice property of quasi-finite morphisms in our setting that they are cohomologically étale over a stratification of the target (\Cref{cor:coh et on stratum}). In particular, the support of smooth objects is cohomologically étale over a stratification of the target and in particular they are dualisable on the corresponding stratification of the support (\Cref{prop:some sm observation}). For formal reasons, we also know that sheaves whose support is cohomologically étale over the target and which are dualisable on their support are smooth. However, these are not all smooth objects which can be seen by just taking direct sums of smooth objects (\Cref{xmpl:sm-sums}) and we were not able to find a full characterisation of smooth objects in the relative case.
\begin{subtquestion}[\ref{question:rel smooth}]
    Let $f: Y \to X$ be a morphism in $\mc C$. Are $f$-smooth objects generated by smooth objects of the form in \Cref{prop:a class of sm} under direct sums and cones?
\end{subtquestion}

\subsection{Relation to the Riemann-Hilbert correspondence} 
The formalism of $\mb F_p$-sheaves in characteristic $p$ is related to coherent cohomology in a certain sense by the $\mathrm{mod} \ p$ Riemann-Hilbert correspondence proven by B. Bhatt and J. Lurie in \cite{lurie-riehilb}. Also see \cite{mathew-riehilb} for an alternative exposition. They show that $\mb F_p$-sheaves in characteristic $p$ correspond to \textit{algebraic Frobenius modules} via an equivalence of categories which satisfies certain compatibilities with pullback, pushfoward and $!$-functors. However, coherent cohomology in the classic sense cannot be upgraded to a 6-functor formalism in a reasonable way. This problem can be avoided by using the theory of condensed mathematics developed by P. Scholze and D. Clausen, which might open up a possibility to turn the $\mathrm{mod} \ p$ Riemann-Hilbert correspondence into a correspondence of 6-functor formalisms. Cohomological smoothness would transfer along such a correspondence which adds additional motivation to \Cref{thm:char}.

\subsection{Overview}
This paper is organised into four sections.\par
In section 2 we recall the notion of 6-functor formalisms in the sense of \cite{Mann} and the formal groundwork that we will need later on. Following \cite{Scholze1}, we construct the étale $6$-functor formalism in a generality which allows us to pass to the absolute weak normalisation (which we also briefly recall from \cite{rydh} in this context), using the results of P. Hamacher in \cite{Hamacher}. We then discuss the notions of smooth objects and cohomologically smooth and cohomologically étale morphisms and prove important criteria and stability properties. Here we essentially follow the arguments given in \cite{Mann2}, \cite{Scholze1} and \cite{Zav}. This is also intended as a small overview over the different (but mostly equivalent) formal notions and their properties in these three sources, and the relations between them. While some of the statements do not seem to be recorded in the literature yet, there is hardly a claim of originality in this section. \par
Section 3 is concerned with cohomologically étale morphisms in the étale 6-functor formalism and in particular our geometric characterisation of them via the absolute weak normalisation. After observing some important properties of absolutely weakly normal schemes, we first give a description of the \textit{cohomologically unramified} morphisms (i.e. morphisms whose diagonal is cohomologically étale, cf. \Cref{def:coh unram}); namely, we prove that they are exactly the topologically unramified ones. We then show that topologically unramified morphisms are quasi-finite (\Cref{cor:coh-et implies quasi-finite}) and that a quasi-finite cohomologically smooth morphism is open (\Cref{prop:quasi-finite coh sm open}). All in all, we see that cohomologically étale morphisms are universally open and topologically unramified as they are both cohomologically smooth and cohomologically unramified. We then use the descent results for perfect schemes from \cite{arc} to prove that a morphism with these properties is étale after absolute weak normalisation (\Cref{thm:qf open unram awn etale}). To finish the proof of \Cref{cor:coh sm quasifinite coh etale}, we prove that a quasi-finite morphism is cohomologically étale if and only if it is cohomologically smooth. We end the section with a digression on quasi-finite morphisms in positive characteristic. Namely, we show that they are cohomologically étale over a stratification of the target (\Cref{cor:coh et on stratum}) which will have an interesting consequence for the discussion of smooth objects (\Cref{prop:some sm observation}).\par
Using the results from section 3, we then answer our main question from above in section 4: For this, we first show that $\mb A_k^n \to \Spec k$ is not $\mc D_p$-cohomologically smooth for any $n > 0$ by an explicit calculation and then use this fact to characterise the cohomologically smooth morphisms to a field. By stability of cohomologically smooth morphisms under base change, this shows that $\mc D_p$-cohomologically smooth morphisms are quasi-finite (\Cref{cor:coh-smooth implies quasi-finite}).\par
Finally, we employ a similar strategy for the description of smooth objects in section 5: We first deal with the absolute case, i.e. the objects which are smooth for morphisms to the base field. By an explicit calculation, we prove that $\mb F_p$ is not smooth for $\mb A_k^1 \to \Spec k$ (\Cref{prop:1 not P1 sm}) and can deduce the characterisation of smooth objects from this (\Cref{prop:sm object char absolute}), mostly by formal arguments. By stability of smooth objects under base change, this readily implies the statement about the relative case (\Cref{cor:sm obj quasi-finite}).
\vspace{1em}

\subsection{Acknowledgements}
    This paper grew out of my master's thesis, so first and foremost I want to thank Johannes Anschütz who suggested the main question and was a thoughtful and caring advisor throughout. This paper owes a huge intellectual debt to his ideas. I also want to thank Peter Scholze for his insightful lectures on six functor formalisms in the winter semester 22/23 which introduced me to this highly fascinating topic, and I am very grateful to Lucas Mann for kindly answering some intricate questions about abstract six functor formalisms. A special thanks goes to my friend and fellow student Tim Kuppel for countless illuminating discussions.
\newpage

\section{Abstract 6-functor formalisms}

\subsection{The definition}
We recall the definition of an abstract 6-functor formalism. This requires some preparation. We adopt most of the terminology and notation from \cite[A.5]{Mann}.

\begin{definition}[{\cite[A.5.1]{Mann}}]
A geometric setup is a pair $\mathfrak G := (\mc C, E)$, where $\mc C$ is an $\infty$-category with all finite limits and $E$ is a collection of homotopy classes of edges in $\mc C$ with the following two properties:
\begin{enumerate}[label=\roman*)]
\item $E$ contains all isomorphisms and is stable under composition.
\item $E$ is stable under pullback along arbitrary morphisms in $\mc C$.
\end{enumerate}
If $E$ satisfies also
\begin{enumerate}[label=\roman*)]
    \setcounter{enumi}{2}
    \item For $f \in E$ and $g$ any morphism, $f \circ g \in E$ implies $g \in E$.
\end{enumerate}
then we call $(\mc C, E)$ a \textit{good} geometric setup.
We write $\mc C(\mathfrak G) := \mc C$ and $E(\mathfrak G) := E$.
\end{definition}

\begin{remark}
The condition that $\mc C$ has finite limits is not part of L. Mann's definition of a geometric setup. However, this assumption simplifies the discussion of general 6-functor formalisms and is satisfied in the cases relevant for us in this paper. In general, for Mann's definition of a 6-functor formalism, one only needs to require existence of pullbacks for morphisms in $E$. The notion of good geometric setups is non-standard.
\end{remark}

\begin{example}
    The category $\mc C$ of qcqs schemes (or a slice thereof) with $E$ the class of all separated, finite type morphisms is a geometric setup. We will later extend the class $E$ slightly, cf. \Cref{cor:et geom setup}.
\end{example}

Given a geometric setup $(\mc C, E)$, one can construct the symmetric monoidal \textit{$\infty$-category $\mathrm{Corr}(\mc C)^{\tensor}_{E,all}$ of correspondences associated to $(\mc C, E)$}. We refer to \cite[A.5]{Mann} for the definition and record only some important facts about this category:
\begin{itemize}
    \item The underlying $\infty$-category $\mathrm{Corr}(\mc C)_{E,all}$ of $\mathrm{Corr}(\mc C)^{\tensor}_{E,all}$ has as objects the objects of $\mc C$.
    \item Morphisms $Y \to X$ in $\mathrm{Corr}(\mc C)_{E,all}$ are given by diagrams
    \[\begin{tikzcd}
        & Y' \\
        Y && X
        \arrow[from=1-2, to=2-1]
        \arrow["{\in E}", from=1-2, to=2-3]
    \end{tikzcd}\]
    where $Y' \to X$ is a morphism in $E$. For two morphisms $Z \leftarrow Z' \to Y$ and $Y \leftarrow Y' \to X$ there is a 2-simplex
    \[\begin{tikzcd}[/tikz/cells={/tikz/nodes={shape=asymmetrical
        rectangle,text width=1.5cm,text height=1ex,text depth=0.15ex,align=center}}]
        && {Z'\times_Y Y'} \\
        & {Z'} && {Y'} \\
        {Z} && Y && X
        \arrow["{\in E}", from=2-4, to=3-5]
        \arrow[from=2-2, to=3-1]
        \arrow["{\in E}", from=2-2, to=3-3]
        \arrow[from=2-4, to=3-3]
        \arrow[from=1-3, to=2-2]
        \arrow["{\in E}", from=1-3, to=2-4]
    \end{tikzcd}\]
    witnessing that $Z \leftarrow Z' \times_Y Y' \to X$ is a composition of the two morphisms.
    \item The symmetric monoidal structure $\tensor$ on $\mathrm{Corr}(\mc C)_{E,all}$ is on objects just given by the cartesian product $X \tensor Y = X \times Y$.\footnote{For simplicity, we describe monoidal structures and (lax) symmetric monoidal functors here in a naive way as well and refer for the correct definition (as a coCartesian fibration/$\infty$-operad) to \cite[A.5.4]{Mann}.} Given two morphisms $X \leftarrow X' \to S$ and $Z \leftarrow Z' \to Y$, the induced morphism $Z \tensor X \to Y \tensor S$ is given by $Z \times X \leftarrow Z' \times X' \to Y \times S$. %
\end{itemize}

In particular, we have the following three kinds of morphisms in $\mathrm{Corr}(\mc C)_{E,all}$:
\begin{enumerate}[leftmargin=2cm,label=(\arabic*)]
\item For a map $f:Y \to X$ in $\mc C$ a map $X \to Y$ given by the diagram $X \leftarrow Y = Y$.
\item For a map $f:Y \to X$ in $E$ a map $Y \to X$ given by the diagram $Y = Y \to X$.
\item For $X\in \mc C$, a morphism $X \tensor X \to X$ induced by the diagram $X \times X \overset{\Delta}{\leftarrow} X \to X$.
\end{enumerate}

\begin{remark}
Note that in particular there is an embedding of symmetric monoidal $\infty$-categories
$$\mc C^\mathrm{op} \hookrightarrow \mathrm{Corr}(\mc C)_{E,all}^{\tensor}$$
given by the contravariant functor which sends an object to itself and a morphism $Y \to X$ to the corresponding morphism of type (1) in the above list. Here, $\mc C^\text{op}$ is equipped with the cocartesian monoidal structure.
\end{remark}

Let us now recall the definition of a 3-functor formalism.

\begin{definition}[{\cite[Definition A.5.6]{Mann}}]\label{def:3 functor formalism}
Let $(\mc C,E)$ be a geometric setup. A 3-functor formalism is a lax symmetric monoidal functor
$$\mc D: \mathrm{Corr}(\mc C)_{E,all}^{\tensor} \to \mathrm{Cat}_\infty$$
where $\mathrm{Cat}_\infty$ is equipped with the cartesian symmetric monoidal structure.
\end{definition}

The three functors are then induced by the three maps in $\mathrm{Corr}(\mc C)_{E,all}^{\tensor}$ mentioned above:

\begin{enumerate}[leftmargin=2cm,label=(\arabic*)]
    \item The map $X \to Y$ given by the diagram $X \overset{f}\leftarrow Y = Y$ induces the \textit{pullback} functor $f^*: \mc D(X) \to D(Y)$.
    \item The map $Y \to X$ given by the diagram $Y = Y \overset{f}\to X$ induces the \textit{lower shriek} functor $f_!: \mc D(Y) \to \mc D(X)$.
    \item The morphism $X \tensor X \to X$ induces the \textit{tensor product} $\tensor: \mc D(X) \times \mc D(X) \to \mc D(X \times X) \to \mc D(X)$ where the first map is induced by the lax monoidality of $\mc D$ and the second map is given by $\Delta^*$. This equips $\mc D(X)$ with a symmetric monoidal structure $\tensor$.
\end{enumerate}
\begin{notation}
    $f^*, f_!$ and $\tensor$ are sometimes also denoted $\mc D^*(f)$, $\mc D_!(f)$ and $\tensor^{\mc D}$ to indicate the 6-functor formalism that induces them.
\end{notation}

\Cref{def:3 functor formalism} encodes both base change and projection formula into the functoriality of $\mc D$:
\begin{proposition}\label{prop:BC and PF}
    Let $\mc D$ be a 3-functor formalism on a geometric setup $(\mc C, E)$. Then:
    \begin{enumerate}[label=\roman*)]
        \item (Base Change) For every cartesian diagram
            \[\begin{tikzcd}
                Y' \rar["f'"] \dar["g'"] & X' \dar["g"] \\
                Y \rar["f"] & X
            \end{tikzcd}\]
            with $g \in E$ there is a natural isomorphism of functors $\mc D(X') \to \mc D(Y)$
            $$f^* g_! \cong g'_!  {f'}^* $$
        \item (Projection Formula) For every $f: Y \to X$ in $E$, there is a natural isomorphism of functors $\mc D(Y) \times \mc D(X) \to \mc D(X)$
        $$f_!(\blank \tensor f^* \blank) \cong f_! \blank \tensor \blank$$
    \end{enumerate}
\end{proposition}
\begin{proof}
    We only show $i)$, the argument for $ii)$ is similar (cf. \cite[A.5.8]{Mann}). By definition, $f^*$ is induced by the $1$-simplex $a:=(X \overset{f}\leftarrow Y = Y)$ and $g_!$ is induced by the $1$-simplex $b:=(X' = X' \overset{g}\to X)$ in $\mathrm{Corr}(\mc C)_{E,all}^{\tensor}$. By the definition of $\mathrm{Corr}(\mc C)_{E,all}^{\tensor}$ there is a 2-simplex witnessing that the morphism $c = (X' \leftarrow Y' \to Y)$ is a composition $a \circ b$ (as the diagram is cartesian, cf. the description of $\mathrm{Corr}(\mc C)_{E,all}^{\tensor}$ above); but another 2-simplex realises $c$ as a composition of $X' \leftarrow Y' = Y'$ and $Y' = Y' \to Y$, which correspond under $\mc D$ to $g'^*$ and $f'_!$.
\end{proof}

\begin{corollary}[Künneth Formula]\label{cor:Kuenneth}
    Let $\mc D$ be a 3-functor formalism on a geometric setup $(\mc C, E)$. For every cartesian diagram
            \[\begin{tikzcd}
                Z \times_X Y \rar["f'"] \dar["g'"] \drar{h} & Z \dar["g"] \\
                Y \rar["f"] & X
            \end{tikzcd}\]
        with $f,g \in E$ (and thus $h \in E$) and for all $A \in \mc D(Y)$ and $B \in \mc D(Z)$ there is a natural isomorphism
        $$h_!(g'^* A \tensor f'^* B) \cong f_! A \tensor g_! B$$
\end{corollary}
\begin{proof}
    We use base change and projection formula (\Cref{prop:BC and PF}) to calculate
    $$h_!(g'^* A \tensor f'^* B) \cong f_! g'_! (g'^* A \tensor f'^* B) \cong f_! (A \tensor g'_! f'^* B) \cong f_! (A \tensor f^* g_! B) \cong f_! A \tensor g_! B$$
\end{proof}

To get a $6$-functor formalism, it remains to require the existence of adjoints:
\begin{definition}[{\cite[Definition A.5.7]{Mann}}]\label{def:6 functor formalism}
    Let $(\mc C,E)$ be a geometric setup. A \textit{6-functor formalism} is a 3-functor formalism 
    $$\mc D: \mathrm{Corr}(\mc C)_{E,all}^{\tensor} \to \mathrm{Cat}_\infty$$
    such that for all $X \in \mc C$ the symmetric monoidal $\infty$-category $\mc D(X)$ is closed and all $f^*: \mc D(X) \to \mc D(Y)$ and $f_!: \mc D(Y) \to \mc D(X)$ admit right adjoints. We write $\iHom$ for the internal Hom, $f_*$ for the right adjoint of $f^*$ and $f^!$ for the right adjoint of $f_!$.
\end{definition}

Naturally, we expect statements dual to \Cref{prop:BC and PF} for the adjoints $f_*, f^!$ and $\iHom$.

\begin{proposition}\label{PF-adj-isos}
    Let $\mc D$ be a 6-functor formalism on a geometric setup $(\mc C, E)$. Then:
    \begin{enumerate}[label=\roman*)]
        \item (Adjoint Base Change) For every cartesian diagram
        \[\begin{tikzcd}
            Y' \rar["f'"] \dar["g'"] & X' \dar["g"] \\
            Y \rar["f"] & X
        \end{tikzcd}\]
        with $g \in E$ there is a natural isomorphism of functors $\mc D(Y) \to \mc D(X')$
        $$g^! f_* \cong f'_* g'^! $$
        \item Let $f: Y \to X$ be a morphism in $E$.
        \begin{itemize}
            \item (Interior Pullback) There is a natural isomorphism of functors $\mc D(X) \times \mc D(X) \to \mc D(Y)$ 
            $$f^!\iHom(\blank,\blank) = \iHom(f^*\blank,f^!\blank)$$
            \item (Verdier Duality) There is a natural isomorphism of functors $\mc D(Y) \times \mc D(X) \to \mc D(X)$ 
            $$f_*\iHom(\blank,f^!\blank) \cong \iHom(f_! \blank,\blank)$$
        \end{itemize}
    \end{enumerate}
\end{proposition}
\begin{proof}
    Using the adjunctions, all assertions are equivalent to the corresponding statements in \Cref{prop:BC and PF} by the Yoneda lemma.
\end{proof}

\begin{remark}\label{rem:6-functor slices}
    Let $\mc D$ be a 3- (resp. 6-)functor formalism on a geometric setup $\mathfrak G = (\mc C, E)$ and $X \in \mc C$. Then $\mc D$ induces a 3- (resp. 6-)functor formalism on the slice geometric setup $\mathfrak G_{/X} := (\mc C_{/X}, E_{/X})$ (where $E_{/X}$ are the edges in $\mc C_{X}$ which lie in $E$). Indeed, this is just $\mc D$ precomposed with the lax symmetric monoidal functor $\mathrm{Corr}(\mc C_{/X})_{E_{/X},all}^{\tensor} \to \mathrm{Corr}(\mc C)_{E,all}^{\tensor}$ coming from the natural lax symmetric monoidal functor $(\mc C_{/X}^\mathrm{op})^\amalg \to (\mc C^\mathrm{op})^\amalg$.
\end{remark}

\subsection{Construction of 6-functor formalisms}

In practice, for a morphism $f: Y \to X$, one often wants to define $f_!$ as $p_! j_!$ where $f = p \circ j$ is a decomposition of $f$ into a morphism $p$ for which $p^*$ admits a right adjoint $p_!$ and a morphism $j$ for which $j^*$ admits a left adjoint $j_!$ (e.g. compactifications in étale cohomology). This motivates the following strategy to construct 3- or 6-functor formalisms: First construct a functor $\mc D_0: \mc C^{op} \to \mathrm{Cat}_\infty$, inducing the functors $f^*$. Choose subclasses $I$ and $P$ of morphisms in $E$ for which $f^*$ admits a left resp. right adjoint. Under the right conditions on $I$, $P$ and $E$, one can then define the functor $f_!$ for $f$ in $I$ or $P$ as the respective adjoint and extend this to all morphisms in $E$, yielding an extension of $\mc D_0$ to the desired 3- or 6-functor formalism $\mc D: \mathrm{Corr}(\mc C)_{E,all}^{\tensor} \to \mathrm{Cat}_\infty$. This was made precise by L. Mann in \cite[Appendix A]{Mann}.

\begin{definition}{\cite[Definition A.5.9]{Mann}}\label{def:suitable decomposition}
    Let $(\mc C, E)$ be a geometric setup. A suitable decomposition of $E$ is a pair $I,P \subseteq E$ of subclasses with the following properties:
    \begin{enumerate}[label = \roman*)]
        \item Every $f\in E$ is of the form $f = p \circ j$ for some $j \in I$ and some $p \in P$.
        \item Every morphism $f \in I \cap P$ is $n$-truncated for some $n \geq -2$ (which may depend on $f$).\footnote{As $\mc C$ has finite limits, here the following inductive definition is available (cf. \cite[Lemma 5.5.6.15]{HTT}): isomorphisms are $-2$-truncated and a morphism is called $n$-truncated if its diagonal is $n-1$-truncated. In particular, if $\mc C$ is (the nerve of) a 1-category, then this is automatic.}
        \item $I$ and $P$ contain all identity morphisms and are stable under pullback along arbitrary morphisms in $\mc C$.
        \item Given $f: Y \to X$ and $g: Z \to Y$ in $\mc C$ such that $f \in I$ (resp. $f \in P$), then $g \in I$ (resp. $g \in P$) if and only if $f \circ g \in I$ (resp. $f \circ g \in P$).
    \end{enumerate}
\end{definition}

\begin{example}\label{xmpl:small-E-setup}
    Let $(\mc C, E)$ be the geometric setup of qcqs schemes with all separated, finite type morphisms as $E$. Let $I$ be the subclass of all open immersions and $P$ the subclass of all proper morphisms. Then by Nagata's theorem on compactifications (cf. \cite[\texttt{0F41}]{Stacks}), every $f \in E$ can be factored as $f = p \circ j$ where $j \in I$ and $p \in P$. Every morphism $f \in \mc C$ is $0$-truncated and $I$ and $P$ are stable under pullback. Moreover, $iv)$ clearly holds for both $I$ and $P$. So $I$ and $P$ are a suitable decomposition of $E$.
\end{example}

We make a small definition to gather all the data in one object.

\begin{definition}
    We call a quadruple $\mathfrak{G} = (\mc C, E, I, P)$ a \textit{construction setup} if $(\mc C, E)$ is a geometric setup and $I,P \subseteq E$ is a suitable decomposition of $E$. We will sometimes write e.g. $I(\mathfrak G)$ to denote the set $I$ associated to $\mathfrak{G}$ and similarly for $\mc C, E$ and $P$.
\end{definition}

\begin{remark}\label{rem:construction setup}
    This is non-standard terminology. Note that if $\mathfrak{G}=(\mc C, E, I, P)$ is a construction setup, then $(\mc C, E)$ is a good geometric setup: Indeed, by the usual argument (cf. e.g. the proof of \Cref{prop:G_k suitable decomp}) this is equivalent to the question if the diagonal $\Delta_g$ lies in $E$ if $g$ does. We can write $g=p \circ j$ where $p \in P$ and $j \in I$. Then $\Delta_{p\circ j}$ is the composition of $\Delta_j$ with a base change of $\Delta_p$. But since $I$ and $P$ form a suitable decomposition, $\Delta_j \in I$ and $\Delta_p \in P$ and thus $\Delta_g = \Delta_{p\circ j} \in E$ as $E$ is stable under pullback and composition.
\end{remark}

\begin{proposition}[{\cite[Proposition A.5.10]{Mann}}\label{prop:construct 6 functors}]
    Let $\mathfrak G=(\mc C, E, I, P)$ be a construction setup. Let $\mc D_0: (\mc C^\mathrm{op})^\amalg \to \mathrm{Cat}_\infty$ be a lax symmetric monoidal functor, inducing for every $f:Y \to X$ the pullback functor $f^*: \mc D(X) \to \mc D(Y)$. Assume that the following two conditions are satisfied:
    \begin{enumerate}[label = \roman*)]
        \item For every $j:U \to X$ in $I$, $j^*$ admits a left adjoint $j_!$ satisfying base change and projection formula.\footnote{Because of the adjointness of $j_!$ and $j^*$ this is really a condition, not a datum: for both base change and projection formula there are natural maps between the respective functors. For example, here, we have a natural map $j_!(\blank \tensor j^* \blank) \to j_! \blank \tensor \blank$ adjoint to the map $\blank \tensor j^* \blank \to j^*(j_! \blank \tensor \blank) = j^* j_! \blank \tensor j^* \blank$ induced by the unit. The construction for morphisms $p \in P$ is similar, with the difference that the transformations go in the other direction.}
        \item For every $p:Y \to X$ in $P$, $p^*$ admits a right adjoint $p_!$ satisfying base change and projection formula.
    \end{enumerate}
    Moreover, assume the following holds:
    \begin{enumerate}[label = \roman*)]
        \setcounter{enumi}{2}
        \item For every cartesian diagram
        \[\begin{tikzcd}
            U' \rar["j'"] \dar["p'"] & X' \dar["p"] \\
            U \rar["j"] & X
        \end{tikzcd}\]
        with the horizontal morphisms in $I$ and vertical morphisms in $P$, the natural map $j_! p'_! \to p_! j'_!$ is an isomorphism.\footnote{This is the map adjoint to $p'_! \to p'_! j'^* j'_! \xleftarrow{\sim} j^* p_! j'_!$ where the first map is the unit of the $j_!$-$j^*$-adjunction and the second comes from the natural morphism $j^* p_! \to p'_! j'^*$ which is an isomorphism by $ii)$.}
    \end{enumerate}
    Then $\mc D_0$ can be extended to a $3$-functor formalism $\mc D$ such that $f_!:=\mc D_!(f)$ agrees with the respective adjoint of $f^*$ for morphisms $f$ in $I$ or $P$.\\
    If moreover $\mc D(X)$ is a closed $\infty$-category for all $X \in \mc C$, $f^*$ has a right adjoint for all $f$ in $\mc C$ and $f_!$ itself admits a right adjoint $f^!$ for $f \in P$, then $\mc D$ is already a $6$-functor formalism.
\end{proposition}
\begin{proof}
See \cite[Proposition A.5.10]{Mann}.
\end{proof}

We call a 6-functor formalism that arises from the construction in \Cref{prop:construct 6 functors} a \textit{6-functor formalism on the construction setup $\mathfrak G$}.

\subsection{The 6 functors for étale sheaves}

Now we want to define the étale 6-functor formalism on the category $\Sch$ of qcqs schemes. As in \Cref{xmpl:small-E-setup} we could choose all separated, finite type morphisms as the class $E$ of our geometric setup. However, we need a slight extension of this setting which enables us to pass to the \textit{absolute weak normalisation} in our formalism.

\begin{definition}[{\cite[Definition B.1]{rydh}}]
    A scheme $X$ is called \textit{absolutely weakly normal (a.w.n.)} if $X$ is reduced and every separated universal homeomorphism $\pi: X' \to X$ is an isomorphism if $X'$ is reduced.
\end{definition}

\begin{lemma}\label{lmm:awn-adjoint}
    The inclusion $\mathrm{Awn} \to \Sch$ of the full subcategory of absolutely weakly normal schemes admits a right adjoint $(\blank)^\mathrm{awn}$, called the absolute weak normalisation. It is given by sending $X$ to the initial object $X^\mathrm{awn}$ of the category $\mc U_X$ of universal homeomorphisms $X' \to X$ and a morphism $Y \to X$ to the morphism $Y^\mathrm{awn} \to Y \times_X X^\mathrm{awn} \to X^\mathrm{awn}$ (note that $Y \times_X X^\mathrm{awn} \to Y$ is a universal homeomorphism).
\end{lemma}
\begin{proof}
    By \cite[\texttt{0EUS}]{Stacks}, $X^\mathrm{awn}$ exists indeed for all $X$ and $X \mapsto X^\mathrm{awn}$ yields a functor $\Sch \to \mathrm{Awn}$. In fact, it is proven there that the initial object is given by $(\ilim_{X' \in \mc U'_X} X' \to X)$ (where the limit is taken in $\Sch$) with $\mc U'_X$ the category of universal homeomorphisms $X' \to X$ of finite type. Thus, for an a.w.n. scheme $Y$ it follows that $\Hom(Y, X^\mathrm{awn}) \cong \ilim_{X'\to X \in \mc U_X} \Hom(Y, X')$ and it remains to show that $$\ilim_{X'\to X \in \mc U_X} \Hom(Y, X') \to \Hom(Y,X)$$ is an isomorphism. Note that for a morphism $f: Y \to X$, the base change $g': Y' \to Y$ of any universal homeomorphism $g: X' \to X$ is again a universal homeomorphism and thus there is a unique section $Y \cong Y^\mathrm{awn} \to Y'$ of $g'$ since $Y$ is a.w.n.. In particular, there is a unique $Y \to X'$ such that $Y \to X' \to X$ is $f$. This proves both surjectivity and injectivity.
\end{proof}

We note that in positive characteristic the absolute weak normalisation has a particularly nice description:

\begin{proposition}\label{prop:perf and awn}
  A scheme $X$ in characteristic $p>0$ is a.w.n. if and only if the Frobenius endomorphism $\phi: X \to X$ is an isomorphism, i.e. $X$ is perfect. In particular, the absolute weak normalisation agrees with the perfection $(\blank)^\mathrm{perf}$, i.e. the functor mapping $X \mapsto \ilim_\phi X$.
\end{proposition}
\begin{proof}
    The second statement follows from the first by \Cref{lmm:awn-adjoint}, noting that the perfection is right adjoint to the inclusion of perfect schemes. Let us prove the first statement. The "only if" is clear from the definition, as the Frobenius is a universal homeomorphism. For the converse, let $X$ be perfect. Then one can use that for a universal homeomorphism $Y \to X$ the induced map $\ilim_\phi Y = Y^\mathrm{perf} \to X^\mathrm{perf}$ is an isomorphism (cf. \cite[Lemma 3.8]{bhatt-perfect}); in particular, $X^\mathrm{awn} \cong (X^\mathrm{awn})^\mathrm{perf} \cong X^\mathrm{perf} \cong X$, where the first isomorphism comes from the proven direction, the second one from the fact that $X^\mathrm{awn} \to X$ is a universal homeomorphism and the last one is implied by the assumption. Thus $X$ is a.w.n..
\end{proof}

We want the class $E$ of our geometric setup to contain all separated morphisms of finite type and their absolute weak normalisations; on the other hand, we want to apply \Cref{prop:construct 6 functors}, so compactifications as in \Cref{xmpl:small-E-setup} should still be available. As absolute weak normalisations of finite type morphisms are in general not at all of finite type, we need to weaken the assumptions in Nagata's theorem. The following result by P. Hamacher solves this problem.

\begin{definition}{\cite[Definition 1.4]{Hamacher}}
A morphism $f:Y \to X$ is called \textit{of finite expansion} if $X$ can be covered by opens $U_i$ such that there are $V_j$ covering $Y$ with the following property: each $f|_{V_j}$ factors as $V_j \to \mb A^n_{U_i} \to U_i \to X$ for some $i$ and some $n$ such that $V_j \to \mb A^n_{U_i}$ is integral.
\end{definition}

\begin{theorem}\label{thm:fin expansion comp}
Let $f:Y \to X$ be a separated morphism of finite expansion between qcqs schemes. Then $f = p \circ j$ where $j: Y \to \bar{Y}$ is an open immersion and $p: \bar{Y} \to X$ is separated and universally closed.
\end{theorem}
\begin{proof}
See \cite[Theorem 1.17]{Hamacher}.
\end{proof}

\begin{example}
    Clearly, all integral morphisms and all finite type morphisms are of finite expansion. Using the stability under composition proven below, this already shows that absolute weak normalisations of finite type morphisms are of finite expansion, using that for a morphism $f: Y\to X$ $f^\mathrm{awn}$ factors as $Y^\mathrm{awn} \cong (Y \times_{X} X^\mathrm{awn})^{\mathrm{awn}} \to Y \times_{X} X^\mathrm{awn} \to X^\mathrm{awn}$ (note that $(\blank)^\mathrm{awn}$ commutes with all limits by \Cref{lmm:awn-adjoint}) where the first map is integral (as it is a universal homeomorphism, cf. \cite[\texttt{04DF}]{Stacks}) and the second one is a base change of $f$.
\end{example}

\begin{proposition}
    Let $S$ be the set of morphisms in $\Sch$ which are of finite expansion. Then:
    \begin{enumerate}[label=\roman*)]
        \item $S$ is stable under composition.
        \item $S$ is stable under pullback along arbitrary morphisms in $\Sch$.
    \end{enumerate}
\end{proposition}
\begin{proof}
$ii)$ is clear by stability of both parts of the defining factorisations under base change. $i)$ reduces to the following observation: if we have factorisations $Z \to \mb A^m_Y \to Y$, $Y \to \mb A^n_X \to X$ where the first morphism is integral, then $Z \to \mb A^m_Y \cong \mb A^m_X \times_X Y \to \mb A^m_X \times_X \mb A^n_X \cong \mb A^{n+m}_X$ is integral as well, again by stability of integral morphisms under base change.
\end{proof}

\begin{corollary}\label{cor:et geom setup}
    Let $E$ be the set of separated morphisms of finite expansion. Then $(\Sch,E)$ is a geometric setup.
\end{corollary}
\begin{proof}
This follows from the previous proposition as separated morphisms are stable under composition and pullback.
\end{proof}

Now \Cref{thm:fin expansion comp} dictates the choice of $I,P \subseteq E$.

\begin{proposition}\label{prop:G_k suitable decomp}
    Let $E$ be the set of separated morphisms of finite expansion. Let $I$ be the class of open immersions and $P$ the class of universally closed maps in $E$. Then $I$ and $P$ form a suitable decomposition of $E$, inducing a construction setup which we will denote by $\mathfrak G$.
\end{proposition}
\begin{proof}
We need to check $i) - iv)$ in \Cref{def:suitable decomposition}. $i)$ is \Cref{thm:fin expansion comp}. $ii)$ is clear as all morphisms in $\mc C(\mathfrak G')$ are $0$-truncated. $iii)$ is obvious as well for both $I$ and $P$. It remains to show $iv)$. The "only if" part is clear. The "if" part follows as usual from considering the commutative diagram
\[\begin{tikzcd}
  Z & {Y \times_X Z} & Z \\
  & Y & X
  \arrow["\Gamma", from=1-1, to=1-2]
  \arrow["f", from=2-2, to=2-3]
  \arrow["{f \circ g}", from=1-3, to=2-3]
  \arrow[from=1-2, to=1-3]
  \arrow[from=1-2, to=2-2]
  \arrow["g"', from=1-1, to=2-2]
\end{tikzcd}\]
Since the graph $\Gamma$ is a base change of $\Delta_f$, this implies the statement for morphisms in $I$ as then $\Delta_f$ is an isomorphism, and for morphisms in $P$ as then $\Delta_f$ is a closed immersion.
\end{proof}

Now we can consider 6-functor formalisms on the construction setup $\mathfrak G$. The central example for us is étale cohomology with $\mb F_p$-module coefficients. More generally, one can allow modules over arbitrary torsion rings. For a torsion ring $\Lambda$ and $X \in \mc C(\mathfrak G)$ we consider the $\infty$-derived category $\mc D(X_\et, \Lambda)$ of étale sheaves of $\Lambda$-modules on $X$\footnote{That is, the category obtained from the homotopy category $\mc K(X_\et, \Lambda)$ of the abelian category of étale sheaves of $\Lambda$-modules by declaring the quasi-isomorphisms to be weak equivalences. This agrees with the category of hypersheaves $\mathrm{HypShv}(X_\et,D(\Lambda))$, cf. \cite[Corollary 2.1.2.3]{SAG}.}, which is a stable, presentable $\infty$-category with the standard $t$-structure. By general topos theory, a morphism $f: Y \to X$ induces a pullback functor $f^*: \mc D(X_\et, \Lambda) \to \mc D(Y_\et,\Lambda)$. We need to construct the corresponding $\mc D_{\Lambda,0}: \mc C(\mathfrak G)^\mathrm{op} \to \mathrm{Cat}_\infty$ and check the conditions in \Cref{prop:construct 6 functors}. For the construction of $\mc D_{\Lambda,0}: \mc C(\mathfrak G)^\mathrm{op} \to \mathrm{Cat}_\infty$ we refer to \cite[Section 2]{LiuZheng}. The symmetric monoidal structure $\tensor$ on $\mc D(X_\et, \Lambda)$ comes from deriving the usual tensor product $\tensor$ on $\mathrm{Shv}(X_\et,\Lambda)$, the category of étale sheaves of $\Lambda$-modules on $X$. Accordingly, its right adjoint $\iHom$ is obtained by deriving the $\iHom$ in $\mathrm{Shv}(X_\et,\Lambda)$. Similarly, for a morphism $f$ in $\mc C(\mathfrak G)$, deriving $f_*$ yields a right adjoint to $f^*$, which we will also denote by $f_*$. \par

\begin{theorem}\label{thm:et 6 functors}
$\mc D_{\Lambda,0}$ extends to a 6-functor formalism $\mc D_\Lambda$ on $\mathfrak G$.
\end{theorem}
\begin{proof}
    We refer to \cite[Theorem 7.16]{Scholze1}. In loc.\ cit.\ the statement is proven for the assignment $X \mapsto \mathrm{HypShv}(X_\et, \hat D_{pf}(\mb Z))$ (where $\hat D_{pf}(\mb Z)$ is the full subcategory of torsion complexes in $D(\mb Z)$), but the same argument applies mutatis mutandis to $\mc D_{\Lambda,0} = \mathrm{HypShv}((\blank)_\et, D(\Lambda))$.\par
\end{proof}

The étale site (and thus étale cohomology) is invariant under universal homeomorphisms (\cite[\texttt{04DZ}, \texttt{03SI}]{Stacks}). This makes the following observation particularly relevant:

\begin{lemma}\label{lmm:awn-univ-hom}
  A morphism $f$ in $\Sch$ is a universal homeomorphism if and only if $f^\mathrm{awn}$ is an isomorphism.
\end{lemma}
\begin{proof}
  Clearly, if $f: Y \to X$ is a universal homeomorphism then so is $f^\mathrm{awn}$. But $Y^\mathrm{awn}, X^\mathrm{awn}$ are reduced and a.w.n., so $f^\mathrm{awn}$ must be an isomorphism by definition. The converse is clear since the natural morphisms $X^\mathrm{awn} \to X$ are universal homeomorphisms.
\end{proof}

\begin{proposition}\label{prop:awn colocalisation}
    The functor $(\blank)^\mathrm{awn}$ from \Cref{lmm:awn-adjoint} exhibits the full subcategory $\mathrm{Awn}$ of a.w.n.\ schemes in $\Sch$ as the colocalisation of $\Sch$ at the universal homeomorphisms.
\end{proposition}
\begin{proof}
    By \Cref{lmm:awn-univ-hom}, a morphism $f$ in $\Sch$ is a universal homeomorphism if and only if $f^\mathrm{awn}$ is an isomorphism. By \Cref{lmm:awn-adjoint} the inclusion $\mathrm{Awn} \to \Sch$ is left-adjoint to the absolute weak normalisation and so by \cite[Example 1.3.4.3]{HA} $\mathrm{Awn}^\mathrm{op} \simeq \Sch^\mathrm{op}[(W^\mathrm{op})^{-1}]$ where $W$ is the class of universal homeomorphisms in $\Sch$.
\end{proof}

\begin{proposition}\label{prop:awn 6 functors}
    Let $\mc D_\Lambda$ be the 6-functor formalism on $\mathfrak G = (\Sch, E, I, P)$ from above. Let $\mathrm{Awn}$ denote the full subcategory of a.w.n.\ schemes in $\Sch$ and $E^\mathrm{awn}$ the classes of morphisms in $\mathrm{Awn}$ which lie in $E$. Then $\mc D_\Lambda$ factors as $\mc D_\Lambda^\mathrm{awn} \circ \mc P$ where $\mc D_\Lambda^\mathrm{awn}$ is the restriction of $\mc D_\Lambda$ to the geometric setup $\mathfrak G^\mathrm{awn} =(\mathrm{Awn}, E^\mathrm{awn})$ and $\mc P: \mathrm{Corr}(\Sch)_{E, all} \to \mathrm{Corr}(\mathrm{Awn})_{E^\mathrm{awn}, all}$ is the lax symmetric monoidal functor induced by $(\blank)^\mathrm{awn}: \Sch_k \to \mathrm{Awn}$.
\end{proposition}
\begin{proof}
    First note that $(\blank)^\mathrm{awn}$ induces a lax symmetric monoidal functor $(\Sch^\mathrm{op})^\amalg \to (\mathrm{Awn}^\mathrm{op})^\amalg$ (as a right adjoint, the absolute weak normalisation commutes with all products). Let $\mc D_{\Lambda,0}$ denote the restriction of $\mc D_\Lambda$ to $\Sch^\mathrm{op}$ and $\mc D_{\Lambda,0}^\mathrm{awn}$ the restriction of $\mc D_\Lambda^\mathrm{awn}$ to $\mathrm{Awn}^\mathrm{op}$. As pullback along universal homeomorphisms $Y \to X$ induces equivalences of categories $\mc D_\Lambda(X) \simeq \mc D_\Lambda(Y)$, $\mc D_{\Lambda,0}$ factors as $\mc D_{\Lambda,0}^\mathrm{awn} \circ (\blank)^\mathrm{awn}$ by \Cref{prop:awn colocalisation}. Now by definition $\mc D_\Lambda$ is obtained from $\mc D_{\Lambda,0}$ via the construction result \Cref{prop:construct 6 functors}. The restricted formalism $\mc D_\Lambda^\mathrm{awn}$ can be obtained from $\mc D_{\Lambda,0}^\mathrm{awn}$ via \Cref{prop:construct 6 functors} as well, using that $(\mathrm{Awn}, E^\mathrm{awn}, I^\mathrm{awn},P^\mathrm{awn})$ is still a construction setup (where $I^\mathrm{awn}$, $P^\mathrm{awn}$ are the classes of morphisms in $\mathrm{Awn}$ which lie in $I$ or $P$ respectively; in other words, they are the image of the classes $I$ and $P$ under $(\blank)^\mathrm{awn}$). We will sketch how this implies the claim by following the proof of the construction result in \cite[Proposition A.5.10]{Mann}: It is shown there that the construction of $\mc D_\Lambda$ resp. $\mc D_\Lambda^\mathrm{awn}$ is equivalent to the construction of a functor\footnote{We use the notation from \cite[Proposition A.5.10]{Mann}.}
    $$\delta^*_{3,\set{3}}(\mc C_\amalg)_{I^-,P^-,all} \to \mathrm{Cat}_\infty$$
    where $\mc C$ is $\Sch$ or $\mathrm{Awn}$, respectively (and $I^-$,$P^-$ are the respective classes). To construct this functor, one starts with the composed functor
    $$\delta^*_{3,\set{1,2,3}}(\mc C_\amalg)_{I^-,P^-,all} \to (\mc C^\mathrm{op})^\amalg \to \mathrm{Cat}_\infty$$
    where the first functor is the diagonal functor and the second one is $\mc D_{\Lambda,0}$ resp. $\mc D_{\Lambda,0}^\mathrm{awn}$. The trisimplical set $\delta^*_{3,\set{1,2,3}}(\mc C_\amalg)_{I^-,P^-,all}$ can be roughly described as follows: Its objects are the objects of $\mc C_\amalg := (\mc C^\mathrm{op})^{\amalg,\mathrm{op}}$ and its morphisms are cubes made out of cartesian squares, such that all the morphisms in vertical resp. horizontal direction lie in $I^-$ resp. $P^-$, with the morphisms in all directions being inverted. The above functor then sends such a cube to the composition of pullback functors along all three dimension. By our observations above we have a factorisation
    \begin{align*}
        \delta^*_{3,\set{1,2,3}}((\Sch)_\amalg)_{I^-,P^-,all} \to (\Sch^\mathrm{op})^\amalg \to (\mathrm{Awn}^\mathrm{op})^\amalg \to \mathrm{Cat}_\infty,
    \end{align*}
    writing $\mc D_{\Lambda,0} = \mc D_{\Lambda,0}^\mathrm{awn} \circ (\blank)^\mathrm{awn}$. Note that $(\blank)^\mathrm{awn}$ preserves morphisms in $I$ and $P$ and so the functor $\delta^*_{3,\set{1,2,3}}((\Sch)_\amalg)_{I^-,P^-,all} \to (\Sch^\mathrm{op})^\amalg \to (\mathrm{Awn}^\mathrm{op})^\amalg$ agrees with the functor
    $$\delta^*_{3,\set{1,2,3}}((\Sch)_\amalg)_{I^-,P^-,all} \to \delta^*_{3,\set{1,2,3}}((\mathrm{Awn})_\amalg)_{I^-,P^-,all} \to \mathrm{Awn}^\mathrm{op}$$
    where the first functor sends an object to its absolute weak normalisation and a cube to the corresponding cube obtained by applying $(\blank)^\mathrm{awn}$ to all edges; the second functor is again the diagonal. Now the desired functor
    $$\delta^*_{3,\set{3}}((\mc C)_\amalg)_{I^-,P^-,all} \to \mathrm{Cat}_\infty$$
    is obtained by inverting the morphisms in the vertical and horizontal direction, which is done by passing to the respective adjoints in $\mathrm{Cat}_\infty$. So we get an induced morphism
    $$\delta^*_{3,\set{3}}((\Sch)_\amalg)_{I^-,P^-,all} \to \delta^*_{3,\set{3}}((\mathrm{Awn})_\amalg)_{I^-,P^-,all} \to \mathrm{Awn}^\mathrm{op} \to \mathrm{Cat}_\infty$$
    which gives, by going backwards through the equivalences of simplicial sets, a factorisation
    $$\mathrm{Corr}(\Sch)_{E, all} \to \mathrm{Corr}(\mathrm{Awn})_{E^\mathrm{awn}, all} \to \mathrm{Cat}_\infty$$
    where the first functor $\mc P$ sends objects to their absolute weak normalisation and a diagram (corresponding to an $n$-simplex) to the diagram where all morphisms are replaced by their absolute weak normalisation, and the second functor is $\mc D^\mathrm{awn}_\Lambda$. 
\end{proof}

To conclude this subsection, let us quickly settle the question what the dualisable objects in $\mc D_\Lambda(X)$ for $X \in \mc C(\mathfrak G)$ look like. This will be important later.

\begin{remark}\label{rem:dualizable-perfect}
    \begin{enumerate}[label=\roman*)]
        \item In $\mc D_\Lambda(X)$, the notions of dualisable objects, perfect objects (in the sense of \cite[\texttt{08G5}]{Stacks}) and objects which are étale locally constant with perfect values agree. The first two are equivalent by \cite[\texttt{0FPP}]{Stacks}. It is also clear that objects which are locally constant with perfect values are perfect in the sense of \cite{Stacks}, i.e. étale locally strictly perfect (in the sense of \cite[\texttt{085L}]{Stacks}). Finally, the fact that dualisable objects are locally constant with perfect values is shown in \cite[Remark 6.3.27]{dualizable-perfect}. 
        \item It follows from $i)$ that the support of a dualisable object $A \in \mc D_\Lambda(X)$ is both open and closed. Indeed, as $A$ is locally constant, it is étale locally supported either everywhere or nowhere and thus, as étale morphisms are open, the corresponding étale cover yields a decomposition $X= U \amalg V$ into two open subsets such that $U$ is the support of $A$.
    \end{enumerate}
\end{remark}

\subsection{Cohomologically smooth morphisms}

Let us now finally recall the notion of "Poincaré duality" in a 6-functor formalism.

\begin{definition}[{\cite[Definition 5.1]{Scholze1}}]
    Let $\mathfrak G = (\mc C, E)$ be a geometric setup and $\mc D$ a 6-functor formalism on $\mathfrak G$. Then a morphism $f:Y \to X$ in $E$ is called ($\mc D$-)\textit{cohomologically smooth} if it satisfies the following conditions:
    \begin{enumerate}[label=\roman*)]
        \item The natural map $f^* \blank \tensor f^! 1_X \to f^! \blank$ adjoint to $f_!(f^* \blank \tensor f^! 1_X) = \blank \tensor f_! f^! 1_X \to \blank$ is an isomorphism.
        \item The object $f^! 1_X \in \mc D(Y)$ (called the \textit{dualising complex}) is invertible.
        \item If $g: X' \to X$ is any morphism in $\mc C$ and $f'$ the pullback of $f$ along $g$, then $f'$ satisfies $i)$ and $ii)$.
        \item In the situation of $ii)$, the natural map $(g')^* f^! 1_X \to (f')^! 1_{X'}$ is an isomorphism where $g'$ is the pullback of $g$ along $f$.
    \end{enumerate}
\end{definition}
\begin{remark}
    \begin{enumerate}[label=\roman*)]\label{rem:coh smooth}
        \item Without condition $ii)$, this is a special case of the notion of $f$-smooth objects which we will discuss below. Actually, we will see there that condition $i)$ and $iii)$ are already implied by $iv)$; and it is even enough to check $iv)$ for $g=f$, cf. \Cref{cor:sm obj sm mor} below.
        \item If $f:Y \to X$ is cohomologically smooth, then there is a left adjoint for $f^*$ (given by $f_!(\blank \tensor f^! 1_X)$) and a right adjoint for $f^!$ (given by $f_*(\blank \tensor (f^! 1_X)^{-1})$). The existence of the former again follows from $iv)$ alone.
    \end{enumerate}
\end{remark}
\begin{example}\label{xmpl:coh smooth example}
    \begin{enumerate}[label=\roman*)]
        \item If $\mc D$ is a 6-functor formalism on a construction setup $\mathfrak G$, then all morphisms $f\in I(\mathfrak G)$ are cohomologically smooth (as $f^* = f^!$).
        \item Any universal homeomorphism $f\in E(\mathfrak G)$ is $\mc D_\Lambda$-cohomologically smooth (where $\Lambda$ is any torsion ring). This follows from the invariance of the étale site under universal homeomorphisms (\cite[\texttt{03SI}]{Stacks}). Indeed, note that $f \in P(\mathfrak G)$ (a universal homeomorphism is integral), so $f_! = f_*$ and since $(f^*,f_*)$ (as underived functors) defines a equivalence of topoi the right adjoint of $f_*$ is $f^!=f^*$. In fact, this shows that universal homeomorphisms are even cohomologically étale, see \Cref{def:coh et prop} below.
        \item  If $\ell \not= p = \mathrm{char} \ k$, then any smooth morphism $f: Y \to X \in E(\mathfrak G)$ is $\mc D_{\mb F_\ell}$-cohomologically smooth, using \Cref{thm:sm coh sm crit} below and e.g. \cite[Remark 6.1.9]{Zav}. If $f$ is additionally of pure dimension $d$ and $Y$ and $X$ are of finite type over $\Spec k$, then $f^! \mb F_\ell$ identifies with the Tate twist $\mu_\ell(d)[2d]$, i.e. $\mu_\ell^{\tensor d}$ for the étale sheaf $\mu_\ell$ of $\ell$-th roots of unity. In particular if $k= \bar k$ and $f: Y \to \Spec k$ is smooth of dimension $d$, then for any sheaf $A$ on $Y$ we have
        \begin{align*}
            \Hom(H_c^i(Y,A),\mb F_\ell) &\cong \Hom(f_! A, \mb F_\ell[-i]) \cong \Hom(A, f^! \mb F_\ell[-i])  \\ 
            & \cong \Hom(\mb F_\ell, \iHom(A, f^* \mb F_\ell \tensor \mu_\ell(d)[2d-i]) \\ 
            &\cong H^{2d-i}(Y, \iHom(A, \mb F_\ell) \tensor \mu_\ell(d))
        \end{align*}
        which is usually written as $H^i_c(Y,A)^\lor \cong H^{2d-i}(Y,A^\lor(d))$; this is Poincaré duality.
    \end{enumerate}
\end{example}

The properties of cohomologically smooth morphisms are most naturally discussed in the context of the more general notion of \textit{smooth objects} for a morphism $f$.

\subsection{Smooth and proper objects}\label{section:sm obj}

Intuitively, smooth and proper objects are the answer to the following question: Let $\mc D$ be a 6-functor formalism on a geometric setup $\mathfrak G$. Given a morphism $f: Y \to X$ in $E(\mathfrak G)$, for which objects $A,B \in \mc D(Y)$ is the functor $f_!(A \tensor \blank)$ "universally" left resp. right adjoint to the functor $f^* \blank \tensor B$? It turns out that these adjunctions can be modelled in another 2-category, called the \textit{Lu-Zheng category} $LZ_{\mc D}$, the \textit{category of cohomological correspondences} or simply the \textit{magical 2-category}. For a precise construction we refer to \cite[Definition 2.2.3]{Zav}. We assume that $E(\mathfrak G)$ contains all morphisms in $\mc C(\mathfrak G)$. Then $LZ_{\mc D}$ can be described as follows:
\begin{itemize}
    \item The objects of $LZ_{\mc D}$ are the objects of $\mc C(\mathfrak G)$.
    \item Given objects $X,Y$ of $LZ_{\mc D}$ the category of morphisms $Y \to X$ is $\mc D(Y \times X)$.
    \item Given morphisms $\mc D(Z \times Y) \ni A: Z \to Y$, $\mc D(Y \times X) \ni B: Y \to X$ the composition $\mc D(Z \times X) \ni B \circ A: Z \to X$ is defined as the convolution
    $$B\star A:= p_{ZX,!}(p_{YX}^* B \tensor p_{ZY}^* A)$$
    where the $p_{??}$ denote the respective projections from $Z \times Y \times X$. The identity morphism $id_X \in \mc D(X \times X)$ is given by $\Delta_{X,!} 1_X$.
\end{itemize}

We quickly recall the notions of adjunctions in 2-categories and adjoint functors of $\infty$-categories.

\begin{definition}\cite[\href{https://kerodon.net/tag/02CG}{\texttt{02CG}}, \href{https://kerodon.net/tag/02EK}{\texttt{02EK}}, \href{https://kerodon.net/tag/02EN}{\texttt{02EN}}]{kerodon}\label{def:adj in 2cat}
    \begin{enumerate}[label=\roman*)]
            \item Let $\mc C$ be a 2-category. An adjunction between two morphisms $f \in \mc C(X,Y)$ and $g \in \mc C(Y,X)$ is a pair $(\eta, \epsilon)$ of a morphism $\eta: id_X \to g f$ in $\mc C(X,X)$ and a morphism $\epsilon: fg \to id_Y$ in $\mc C(Y,Y)$ such that the two compositions
            $$f \overset{f \eta}\to fgf \overset{\epsilon f} \to f \quad \mathrm{and} \quad g \overset{\eta g}\to gfg \overset{g \epsilon}\to g$$
            are the identities on $f \in \mc C(X,Y)$ and $g \in \mc C(Y,X)$ respectively (here we implicitly used the unit and associativity constraints of the 2-category $\mc C$).
        \item Let $\mc C$ and $\mc D$ be two $\infty$-categories. An adjunction of two functors $F: \mc C \to \mc D$ and $G: \mc D \to \mc C$ is an adjunction between $[F]$ and $[G]$ in the homotopy $2$-category $h_2(\mathrm{Cat}_\infty)$.
    \end{enumerate}
\end{definition}

Now assume $\mc C(\mathfrak G)$ has a final object $S$. Then we can consider the 2-functor $h_S = \Hom(S, \blank)$ from $LZ_\mc D$ to $\mathrm{Cat}$.\footnote{cf. \cite{2Cat} for the theory of representable functors in 2-categories.} By definition, this sends an object $X \in LZ_{\mc D}$ to $\mc D(X)$ and a morphism $\mc D(Y \times X) \ni A: Y \to X$ to the functor $FM_A: \mc D(Y) \to \mc D(X)$ mapping $M$ to $A \star M = p_{X,!}(A \tensor p_Y^* M)$, called the \textit{Fourier-Mukai functor with kernel $A$}. In particular, if $X=S$ and $f: Y \to X$, $FM_A$ is exactly the functor $f_!(A \tensor \blank)$; on the other hand, if we consider $A$ as morphism $X \to Y$, it is the functor $A \tensor f^* \blank$. As the 2-functor $h_S$ preserves adjunctions, an adjunction of morphisms $\mc D(X) \ni A: X \to S$ and $\mc D(X) \ni B: S\to X$ thus yields an adjunction of functors $f_!(A \tensor \blank)$ and $B \tensor f^* \blank$. The following definition is now sufficiently motivated:

\begin{definition}\label{def:sm obj}
    Let $\mc D$ be a 6-functor formalism on a geometric setup $\mathfrak G = (\mc C, E)$ and let $f:Y \to X \in E$. Let $\mc C_{E/X}$ be the full subcategory of the slice category $\mc C_{/X}$ consisting of objects with the structure morphism in $E$, and $\mc D_{E/X}$ the restriction of $\mc D$ to $\mc C_{E/X}$.\footnote{See \Cref{rem:6-functor slices}.} Then an object $A \in \mc D(Y)$ is called
    \begin{enumerate}[label=\roman*)]
        \item $f$-smooth (w.r.t. $\mc D$) if $A$ as a morphism $Y \to X$ in $LZ_{\mc D_{E/X}}$ admits a right adjoint $B$.
        \item $f$-proper (w.r.t. $\mc D$) if $A$ as a morphism $X \to Y$ in $LZ_{\mc D_{E/X}}$ admits a right adjoint $B$.
    \end{enumerate}
\end{definition}
\begin{remark}
    \begin{enumerate}[label=\roman*)]
    \item Both properties are stable under base change, as base change defines a 2-functor of the corresponding Lu-Zheng categories, cf. e.g. the proof of \cite[Proposition 7.7]{Mann2}.
    \item Observe that for a morphism $f:Y \to X$ in $E(\mathfrak G)$, the functor $f_!(A \tensor \blank)$ always admits the right adjoint $\iHom(A, f^! \blank)$; in particular, if $A$ is $f$-smooth, there is a natural isomorphism of functors $\iHom(A, f^! \blank) \cong B \tensor f^* \blank$. Applying this to $1_X$, we see that necessarily $B \cong \iHom(A,f^! 1_X)=:\mb D_f(A)$, the \textit{Verdier dual} of $A$. One can now wonder if it suffices conversely to supply a natural isomorphism $\iHom(A, f^! \blank) \cong \mb D_f(A) \tensor f^* \blank$ for $A$ to be $f$-smooth. The problem is that this condition is not necessarily stable under base change; however, being a bit more careful with this idea yields \Cref{cor:coh sm generalization} below.
    \end{enumerate}
\end{remark}

In the following we stay in the situation of \Cref{def:sm obj} and restrict the discussion to $f$-smooth objects as the discussion of $f$-proper objects is dual to this in a certain sense (cf. \cite[Remark 6.3]{Scholze1}).\par
We can think about an adjunction in 2-categories in different ways and the different viewpoints yield useful criteria for $f$-smoothness. The first way is really just the definition of adjunctions in 2-categories (\Cref{def:adj in 2cat}): to define an adjunction between morphisms $A: Y \to X$, $B: X \to Y$ in $LZ_{\mc D_{E/X}}$ is to provide a unit $\Delta_{Y,!} 1_Y = id_Y \to B \circ A = p_1^*B \tensor p_2^* A$ and a counit $f_!(A \tensor B) = A \circ B \to id_X = 1_X$ such that the triangle identities hold.

\begin{definition}
    An $A$-trace-cycle theory on $f$ is a triple $(B, \alpha, \beta)$ of
    \begin{enumerate}[label=\roman*)]
        \item an object $B \in \mc D(Y)$,
        \item an \textit{$A$-trace morphism}
                $$\alpha: f_! (A \tensor B) \to 1_Y,$$
        \item an \textit{$A$-cycle map}
                $$\beta: \Delta_{Y,!} 1_Y \to p_2^* B \tensor p_1^* A,$$
    \end{enumerate}
    such that the corresponding "triangle identity" diagrams commute.\footnote{For an explicit description of these triangles, see \cite[Lecture VI]{Scholze1}.} Here $p_i$ denote the projections $Y \times_X Y \to Y$ and $\Delta_Y$ the diagonal.
\end{definition}

\begin{lemma}\label{lmm:rel trace cycle}
An object $A$ is $f$-smooth if and only if $f$ admits an $A$-trace-cycle theory.
\end{lemma}
\begin{proof}
    This follows by unravelling the definitions.
\end{proof}
\begin{remark}\label{rem:smooth symmetry}
    Note the symmetry in this definition: $(B, \alpha, \beta)$ is an $A$-trace-cycle theory if and only if $(A, \alpha, \beta)$ is a $B$-trace-cycle theory on $f$. In particular, if $A$ is $f$-smooth the right adjoint of $A$ is $f$-smooth as well, with right adjoint $A$. In particular, if $A$ is $f$-smooth, then $\mb D_f(\mb D_f(A)) = A$.
\end{remark}

\begin{corollary}\label{cor:id sm is dualizable}
    $A \in \mc D(X)$ is $id_X$-smooth if and only if it is dualisable.
\end{corollary}
\begin{proof}
    This follows from \Cref{lmm:rel trace cycle} as an $A$-trace cycle theory $(B, \alpha, \beta)$ on $id_X$ is precisely the datum of a duality between $A$ and $B$.
\end{proof}

Another way to think about the adjunction (see \cite[\href{https://kerodon.net/tag/02D2}{\texttt{02D2}}]{kerodon}) is to ask for a 2-morphism $\epsilon: f_!(A \tensor B) = A \circ B \to id_X = 1_X$ such that for all $T: Z \to Y$, $S: Z \to X$ in $LZ$, the natural map $\Hom(T, B \circ S) \to \Hom(A \circ T, S)$ is an isomorphism.\footnote{The map is induced by taking the \textit{left adjunct} of a given map $\beta: T \to B \circ S$; this is the composition $A \circ T \overset{id_A \circ \beta}\Rightarrow A \circ (B \circ S) \Rightarrow (A \circ B) \circ S \overset{\epsilon \circ id_S} \Rightarrow S$.} As noted above, for the adjunction to exist, we must necessarily have $B \cong \mb D_f(A)$. Now we always have a map $\epsilon: f_!(A \tensor \mb D_f(A)) \to 1_X$, namely the one adjoint to the identity on $\mb D_f(A) = \iHom(A, f^! 1_X)$. Moreover, applying the definitions, we see that
$$\Hom(A \circ T, S) = \Hom(p_{ZX,!}(p_{YX}^* A \tensor p_{ZY}^* T), S) \cong \Hom(T, p_{ZY,*} \iHom(p_{YX}^*A, p_{ZX}^! S))$$
where we used the adjunctions of the six functors, and
$$\Hom(T,B \circ S) = \Hom(T, p_{ZY,!}(p_{XY}^* B \tensor p_{ZX}^* S))$$
Also note that $p_{ZY} = id$ as $X$ is final. All these observations culminate in the following result:
\begin{flushleft}
\begin{lemma}\label{lmm:sm obj crit}
In the situation of \Cref{def:sm obj}, the following are equivalent:
\begin{enumerate}[label=\roman*)]
    \item $A \in \mc D(Y)$ is $f$-smooth.
    \item For all morphisms $Z \to X$ and $S \in \mc D(Z)$, the natural map
    $$p_{Y}^* \mb D_f(A) \tensor p_{Z}^* S \to \iHom(p_{Y}^*A, p_{Z}^! S)$$
    induced by the identifications above under Yoneda is an isomorphism.\footnote{Here, $p_Y,p_Z$ denote the projections from $Z \times_X Y$. A calculation yields that this is the map adjoint to the natural map $p_{Z,!}(p_Y^* A \tensor p_Y^* \mb D_f(A) \tensor p_Z^* S) = p_{Z,!}p_Y^*( A \tensor \mb D_f(A)) \tensor S = f^* f_!(A \tensor \mb D_f(A)) \tensor S \to f^* f_! f^! 1_X \tensor S \to S$ where we used the projection formula and base change and the counits of the $\tensor$-$\iHom$- and $f_!$-$f^!$-adjunction.}
    \item \cite[Proposition 6.6]{Scholze1} The natural map
    $$p_2^* \mb D_f(A) \tensor p_1^* A \to \iHom(p_2^*A, p_1^! A)$$
    coming from the map in $ii)$ by choosing $Z = Y$ and $S = A$ is an isomorphism. In fact, it suffices to check this after applying $\Delta^!$ and taking global sections.
\end{enumerate}
\end{lemma}
\end{flushleft}
\begin{proof}
$i) \Rightarrow ii)$ follows from the discussion above by Yoneda; the only caveat is that one needs to ensure that the identification $B \cong \mb D_f(A)$ is compatible with the choice of counits. To see this, note that by the above discussion for $\mc D(Y) \ni T: X \to Y$ and $S=id=1_X$ the natural isomorphism $\Hom(T,B\circ S) \to \Hom(A\circ T, S)$ corresponds to an isomorphism $\Hom(T,B) \to \Hom(T,\mb D_f(A))$. By Yoneda, this induces an isomorphism $\phi: B \to \mb D_f(A)$ and a quick chase of the identifications reveals that it is given by the morphism adjoint to the counit $\epsilon: f_!(B \tensor A) \to 1_X$. But then from the commutative diagram
\[\begin{tikzcd}
    B \rar{\phi} \dar{\phi} & \mb D_f(A) \dar{id}\\
    \mb D_f(A) \rar{id} & \mb D_f(A)
\end{tikzcd}\]
we see that the map $\epsilon$ factors through the morphism $f_!(\mb D_f(A) \tensor A) \to 1$ adjoint to the identity on $\mb D_f(A)$ via $f_!(\phi \tensor A) = \phi \circ A$ by passing to adjoint morphisms in the horizontal direction.\par
$iii)$ is a special case of $ii)$, so it remains to show $iii) \Rightarrow i)$. By the proof of \cite[\href{https://kerodon.net/tag/02CY}\texttt{02CY}]{kerodon}, for $A$ to be left adjoint to $\mb D_f(A)$, it suffices to check that the natural map $\Hom(T, \mb D_f(A) \circ S) \to \Hom(A \circ T, S)$ is an isomorphism for the pairs $(T,S)=(\mb D_f(A), id_X)$ and $(T,S) = (id_Y = \Delta_! 1_Y, A)$. But for the first pair this is automatic from the definition of $\mb D_f(A)$ and $\epsilon$ and for the second pair this is by construction implied by condition $iii)$.
\end{proof}

\begin{remark}
After applying $\Delta^!$ to the map in $iii)$ in \Cref{lmm:sm obj crit} above, it becomes a morphism $$\Delta^! (p_2^* \mb D_f(A) \tensor p_1^* A) \to \iHom(A, A)$$ by the interior pullback formula (\Cref{PF-adj-isos}). The bifunctor $\Delta^!(p_2^* \blank \tensor p_1^* \blank)$ is known as the \textit{!-tensor product} $\tensor^!$ (and appears e.g. in the work of Gaitsgory-Rozenblyum, \cite{gaitsgory}). Using this notation, the criterion takes a more concise form: An object $A$ is $f$-smooth if and only if the natural map
$$\mb D_f(A) \tensor^! A \to \iHom(A,A)$$
is an isomorphism.
\end{remark}

\begin{corollary}\label{cor:sm distinguished}\label{cor:sm retracts}
    Assume that all $\mc D(X)$ are stable $\infty$-categories. Then:
    \begin{enumerate}[label=\roman*)]
        \item If two objects in a distinguished triangle are $f$-smooth, so is the third.
        \item $f$-smooth objects are stable under retracts.
    \end{enumerate}
\end{corollary}
\begin{proof}
    The transformation $p_{Y}^* \mb D_f(\blank) \tensor p_{Z}^* S \to \iHom(p_{Y}^*\blank, p_{Z}^! S)$ in condition $ii)$ of \Cref{lmm:sm obj crit} is a natural map of exact functors. This implies both statements.
\end{proof}

The following corollary mimics the definition of cohomologically smooth morphisms:
\begin{corollary}\label{cor:coh sm generalization}
    In the situation of \Cref{def:sm obj}, an object $A$ is $f$-smooth if and only if the following hold
    \begin{enumerate}[label=\roman*)]
        \item The natural map $\mb D_f(A) \tensor f^* \blank \to \iHom(A, f^! \blank)$ adjoint to $$f_!(A \tensor \mb D_f(A) \tensor f^* \blank) \overset{PF} \cong f_!(A \tensor \mb D_f(A)) \tensor \blank \to f_!f^! 1_X \tensor \blank \to \blank$$ is an isomorphism.
        \item If $g: X' \to X$ is any morphism in $\mc C$, $f'$ the pullback of $f$ along $g$, then $f'$ satisfies $i)$ with $A$ replaced by $g'^* A$ where $g': Y' \to Y$ is the pullback of $g$ along $f$.
        \item In the situation of $ii)$, the natural map $(g')^* \mb D_f(A) \to \mb D_{f'}((g')^*A)$ (the map from \Cref{lmm:sm obj crit} for $S=1_{Y'}$) is an isomorphism.
    \end{enumerate}
\end{corollary}
\begin{proof}
For the "only if" part, we only need to show that $i)$ and $ii)$ hold if $A$ is f-smooth by \Cref{lmm:sm obj crit}. By stability under base change, it even suffices to show $i)$. But this follows because
\[\begin{tikzcd}
  {\Hom(f_!(M \tensor A), N)} & {\Hom(M, \iHom(A,f^!N))} \\
  & {\Hom(M, \mb D_f(A) \tensor f^*N)}
  \arrow[from=2-2, to=1-2]
  \arrow["\cong"', from=2-2, to=1-1]
  \arrow["\cong", from=1-1, to=1-2]
\end{tikzcd}\]
commutes for all $M,N$ by construction.\footnote{More precisely the identity on $\mb D_f(A) \tensor f^*N$ maps to the map $f_!(\mb D_f(A) \tensor f^*N \tensor A) \overset{PF} \cong f_!(\mb D_f(A) \tensor A) \tensor N \to 1 \tensor N$ via $\epsilon \tensor id$ where $\epsilon$ is the counit of the adjunction in LZ, i.e. given by $f_!(\mb D_f(A) \tensor A) \to f_! f^! 1_X \to 1_X$. But this is precisely how we defined the natural map in $i)$.} For the "if" part, we use $ii)$ and $iii)$ with $g=f$ to see that $iii)$ in \Cref{lmm:sm obj crit} holds.
\end{proof}

As smooth and proper objects are actually morphisms in a (2-)category, it is a natural question if and how they can be composed. Indeed, suppose we are given a commutative triangle
\[\begin{tikzcd}
    Z \rar["g"] \drar[swap,"h"] & Y \dar["f"] \\
    & X
\end{tikzcd}\]
Then, given a smooth/proper object for $g$ and one for $f$ resp. $h$, one can wonder if they can be composed to a smooth/proper object for $h$ resp. $f$. By definition, being $g$-smooth/proper is a statement about adjunctions in $LZ_{\mc D_{E/Y}}$ as opposed to $LZ_{\mc D_{E/X}}$ for $h$ and $f$. So one needs an appropriate 2-functor $LZ_{\mc D_{E/Y}} \to LZ_{\mc D_{E/X}}$ which allows to compose the given adjoints in $LZ_{\mc D_{E/X}}$. Such a functor exists and is given by the functor $\mc T$ which is the forgetful functor (along $f$) on objects and sends a morphism $A \in \mc D(S \times_Y T) = \Hom_{LZ_{\mc D_{E/Y}}}(S,T)$ to $\iota_! A \in \mc D(S \times_X T) = \Hom_{LZ_{\mc D_{E/X}}}(S,T)$ where $\iota$ is the obvious map $S \times_Y T \to S \times_X T$ (cf. the proof of \cite[Proposition 7.11]{Mann2}). This allows us to prove the following:

\begin{lemma}\label{lmm:sm obj composition}
Let $\mc D$ be a 6-functor formalism on a geometric setup $\mathfrak G$. Let $g: Z \to Y$ and $f: Y \to X$ be two morphisms in $E(\mathfrak G)$ and $h$ their composition. Let $A \in \mc D(Y)$ and $B, C \in \mc D(Z)$. Then:
\begin{enumerate}[label=\roman*)]
    \item If $B$ is $g$-smooth and $A$ is $f$-smooth, $g^* A \tensor B$ is $h$-smooth.
    \item If $B$ is $g$-proper and $C$ is $h$-smooth, then $g_!(C \tensor B)$ is $f$-smooth.
    \item If $A$ is $f$-proper and $B$ is $g$-proper, then $B \tensor g^* A$ is $h$-proper.
    \item If $C$ is $h$-proper and $B$ is $g$-smooth, then $g_!(B \tensor C)$ is $f$-proper.
\end{enumerate}
More precisely, in the respective cases we have adjunctions
\begin{enumerate}[label=\roman*)]
    \item $g^* A \tensor B \dashv \mb D_g(B) \tensor g^* \mb D_f(A)$
    \item $g_!(C \tensor B) \dashv g_!(\pmb D_g(B) \tensor \mb D_h(C))$
    \item $B \tensor g^* A \dashv g^* \pmb D_f(A) \tensor \pmb D_g(B)$
    \item $g_!(B \tensor C) \dashv g_!(\pmb D_h(C) \tensor \mb D_g(B))$
\end{enumerate}
where $\pmb D_s(T)$ denotes the right adjoint of an $s$-proper object $T$.\footnote{Similar as in the case of smooth objects, one can show that necessarily $\pmb D_s(T) \cong p_1^*\iHom(p_2^*T, \Delta_! 1)$, cf. \cite[Proposition 6.8]{Scholze1}.}
\end{lemma}
\begin{proof}
We only prove $i)$ and $ii)$ as the arguments for $iii)$ and $iv)$ are dual to this. In $i)$, $B$ is considered as a morphism $Z \to Y$ and so we can consider the composition $A \circ \iota_! B$ where $\iota$ is the natural morphism $Z \to Z \times_X Y$, i.e. the graph of $g$. As noted above $\mc T$ defines a $2$-functor, so $\iota_! B$ will be left adjoint to $\iota_! \mb D_g(B)$\footnote{We are a bit sloppy here and do not distinguish between $Z \times Y$ and $Y \times Z$; but of course we have to be careful in which direction the morphisms given by the objects go.} and thus $A \circ \iota_! B$ will be left adjoint to the composition of right adjoints $\iota_! \mb D_g(B) \circ \mb D_f(A)$. Now we just calculate
$$A \circ \iota_! B \cong p_{Z,!}(p_Y^* A \tensor \iota_! B) \overset{PF}\cong p_{Z,!} \iota_!(\iota^* p_Y^* A \tensor B) \cong g^* A \tensor B$$
as claimed. Unravelling the composition of right adjoints in a similar way, we get $\iota_! \mb D_g(B) \circ \mb D_f(A) \cong \mb D_g(B) \tensor g^* \mb D_f(A)$.\par
In $ii)$ $B$ is considered as a morphism $Y \to Z$ and so we consider the composition
$$C \circ \iota_! B = p_{Y_!}(p_Z^* C \tensor \iota_! B) \overset{PF}\cong g_!(C \tensor B)$$
and the similar calculation for the right adjoints gives $\iota_! \mb D_g(B) \circ \mb D_h(C) \cong g_!(\pmb D_g(B) \tensor \mb D_h(C))$.
\end{proof}

\begin{corollary}\label{cor:dualizable+smooth}
In the situation of \Cref{def:sm obj}, assume that $A$ is $f$-smooth. Then for every dualisable object $L \in \mc D(Y)$, $A \tensor L$ is $f$-smooth. If $C$ is $f$-proper and $B$ is $f$-smooth, then $f_!(B \tensor C)$ is dualisable.
\end{corollary}
\begin{proof}
    This is a consequence of \Cref{cor:id sm is dualizable} and \Cref{lmm:sm obj composition} by writing $f = id \circ f$ and $f = f \circ id$.
\end{proof}

An interesting property of smooth objects is that they are compact under (relatively) mild assumptions.

\begin{proposition}\label{prop:sm obj compact}
    In the situation of \Cref{def:sm obj}, assume that $\Delta_{f,!} 1_Y$ is compact. Then any $f$-smooth object $A \in \mc D(Y)$ is compact.
\end{proposition}
\begin{proof}
    Assume $A$ is $f$-smooth. Then by \Cref{lmm:sm obj crit} we have for all $S \in \mc D(Y)$ that $p_2^* \mb D_f(A) \tensor p_1^* S \cong \iHom(p_2^* A,p_1^! S)$. Applying $\Delta_f^!$ on both sides yields $\Delta_f^!(p_2^* \mb D_f(A) \tensor p_1^* S) \cong \iHom(A,S)$ by interior pullback (\Cref{PF-adj-isos}). Thus
    $$\Hom(A,S) \cong \Hom(1_Y, \iHom(A,S)) \cong \Hom(1_Y, \Delta_f^!(p_2^* \mb D_f(A) \tensor p_1^* S)) \cong \Hom(\Delta_{f,!} 1_Y, p_2^* \mb D_f(A) \tensor p_1^* S)$$
    and the last term commutes with colimits in $S$ as $\Delta_{f,!} 1_Y$ is compact by assumption.
\end{proof}

\begin{example}\label{rem:sm obj compact assumption}
    We will see in \Cref{cor:et 6 sm obj compact} that if $k$ is an algebraically closed field, the assumption of \Cref{prop:sm obj compact} holds for all morphisms $f \in E(\mathfrak G)$ between schemes of finite type over $k$ and 6-functor formalisms $\mc D_\Lambda$ on $\mathfrak G$.
\end{example}

Let us also note the following descent result for smooth objects.

\begin{proposition}\label{prop:sm obj descent}
    Let $\mc D$ be a 6-functor formalism on a geometric setup $\mathfrak G$. Let $g: X' \to X$ be a map of universal $\mc D^*$-descent.\footnote{That is, $g$ induces \[\mc D(X) \overset{\sim} \longrightarrow \left(\begin{tikzcd}[ampersand replacement=\&,column sep=1.5em]
        {\mc D(X')} \& {\mc D(X'\times_X X')} \& {...}
        \arrow[shift right, from=1-1, to=1-2]
        \arrow[shift left, from=1-1, to=1-2]
        \arrow[from=1-2, to=1-3]
        \arrow[shift left=2, from=1-2, to=1-3]
        \arrow[shift right=2, from=1-2, to=1-3]
    \end{tikzcd}\right)\] where the arrows are given by the respective pullbacks, and the same holds for any base change of $g$} If $f: Y \to X \in E(\mathfrak G)$ and $g': Y' \to X'$ is the pullback of $g$ along $f$, then $A \in \mc D(Y)$ is $f$-smooth if and only if $g'^*A \in \mc D(Y')$ is $f'$-smooth, where $f'$ denotes the pullback of $f$ along $g$.
\end{proposition}
\begin{proof}
    This follows from the argument given in \cite[Corollary 7.8]{Mann2}. %
\end{proof}

Finally, we have the following very useful result on stability of smooth objects under retracts which generalises \Cref{cor:sm retracts} $ii)$:

\begin{proposition}\label{prop:sm retracts}
    Let $\mc D$ be a 6-functor formalism on a geometric setup $\mathfrak G$.
    Consider a diagram
    \[\begin{tikzcd}
        Z \rar["i"] \drar[swap,"h"] & Y \dar["f"] \\
        & X
    \end{tikzcd}\]
    with $f, h \in E(\mathfrak G)$ such that there is some $r: Y \to Z$ over $X$ with $r\circ i = id$. Let $A \in \mc D(Y)$ be $f$-smooth and $C \in \mc D(Z)$. If there are maps $i^* A \to C$ and $r^* C \to A$ such that $C = i^* r^* C \to i^* A \to C$ is the identity, then $C$ is $h$-smooth (and its right adjoint is a retract of $r_! \mb D_f(A)$).
\end{proposition}
\begin{proof}
See \cite[Proposition 6.17]{Scholze1}.
\end{proof}

\subsection{Properties of cohomologically smooth morphisms}

We can now return to the special case of cohomologically smooth morphisms. We fix a 6-functor formalism $\mc D$ on a geometric setup $\mathfrak G$.

\begin{definition}[{\cite[Definition 3.2.4]{Zav}}]
    Let $(f: Y \to X) \in E(\mathfrak G)$. A trace-cycle theory on $f$ is a $1_Y$-trace-cycle theory $(\omega_f, tr_f, cl_f)$ on $f$ such that $\omega_f$ is invertible.
\end{definition}

\begin{corollary}\label{cor:sm obj sm mor}
Let $(f: Y \to X) \in E(\mathfrak G)$. Then the following are equivalent:
\begin{enumerate}[label=\roman*)]
    \item $f$ is cohomologically smooth.
    \item $1_Y$ is $f$-smooth and $\mb D_f(1) = \iHom(1_Y, f^! 1_X) = f^! 1_X$ is invertible.
    \item $f^! 1_X$ is invertible and the natural map $p_2^* f^! 1_X \to p_1^! 1_Y$ is an isomorphism, where $p_1,p_2$ denote the projections $Y \times_X Y \to Y$.
    \item $f$ admits a trace-cycle theory.
\end{enumerate}
\end{corollary}
\begin{proof}
    This is just \Cref{lmm:rel trace cycle}, \Cref{lmm:sm obj crit} and \Cref{cor:coh sm generalization} for $A=1_Y$ and with the additional assumption that $\mb D_f(A) = f^! 1_X$ is invertible.
\end{proof}

\begin{corollary}\label{cor:coh sm composition}
    Cohomologically smooth morphisms are stable under composition.
\end{corollary}
\begin{proof}
    This follows immediately from \Cref{cor:sm obj sm mor} and \Cref{lmm:sm obj composition} $i)$ as pullback preserves invertible objects.
\end{proof}

\begin{corollary}\label{cor:Verdier trick}
    Assume $(f: Y \to X) \in E(\mathfrak G)$ is cohomologically smooth. Then $f^! 1_X \cong \Delta^* p_1^! 1_Y$.
\end{corollary}
\begin{proof}
Apply $\Delta^*$ to the isomorphism in $iii)$ of \Cref{cor:sm obj sm mor}.
\end{proof}
\begin{remark}
    This fact is called "Verdier's diagonal trick" and reduces the computation of the dualising object to cases where the morphism has a section. This is used in \cite{Zav} to deduce a formula for the dualising complex in sufficiently nice settings.
\end{remark}

\begin{lemma}\label{lmm:upper shriek for invertibles}
    Let $(f: Y \to X) \in E(\mathfrak G)$ and $L \in \mc D(X)$ dualisable. Then for all $A \in \mc D(X)$, $f^! (L \tensor A) \cong f^* L \tensor f^! A$.
\end{lemma}
\begin{proof}
    Let $K$ be the dual of $L$. As pullback preserves duality, this follows by Yoneda since
    \begin{align*}
        \Hom(\blank,f^!(L \tensor A)) & \cong \Hom(f_!(\blank) \tensor K,A) \overset{\text{PF}}{\cong} \Hom(f_!(\blank \tensor f^* K),A) \cong \Hom(\blank, f^* L \tensor f^! A)
    \end{align*}
\end{proof}

If $\mathfrak G$ is a good geometric setup, then for a cohomologically smooth morphism $f$ the inverse of the dualising complex can be described in terms of the diagonal $\Delta_f$.

\begin{corollary}\label{cor:dual to diagonal}
    Assume $\mathfrak G$ is a good geometric setup and $(f: Y \to X) \in E(\mathfrak G)$ is cohomologically smooth. Then $(f^! 1_X)^{-1} \cong \Delta^! 1_{Y \times_X Y}$. In particular, $f^* = f^! \blank \tensor \Delta^! 1_{Y \times_X Y}$.
\end{corollary}
\begin{proof}
    First note that if $\mathfrak G$ is good, $\Delta^!$ indeed exists because $\Delta$ as a section of a morphism in $E$ is contained in $E$. Then note that by \Cref{cor:Verdier trick} $f^! 1_X \tensor \Delta^! 1_{Y \times_X Y} \cong \Delta^* p_1^! 1_Y \tensor \Delta^! 1_{Y \times_X Y} \cong \Delta^! p_1^! 1_Y \cong 1_Y$ where we used \Cref{lmm:upper shriek for invertibles} and invertibility of $p_1^! 1_Y$.
\end{proof}

\begin{lemma}\label{lmm:coh sm interior pull}
    For a cohomologically smooth morphism $(f:Y \to X) \in E(\mathfrak G)$ and $A, B \in \mc D(X)$ we have $f^* \iHom(A,B) \cong \iHom(f^* A, f^* B)$.
\end{lemma}
\begin{proof}
    This follows from the interior pullback formula (\Cref{PF-adj-isos}):
    \begin{align*}
        f^* \iHom(A,B) \tensor f^! 1_X \cong f^! \iHom(A,B) \cong \iHom(f^* A,f^! B) &\cong \iHom(f^* A, f^* B \tensor f^! 1_X) \\
        &\cong \iHom(f^* A, f^* B) \tensor f^! 1_X
    \end{align*}
    where we used that $f^! 1_X$ is invertible and can thus be pulled out of the internal $\Hom$ functor.
\end{proof}

Cohomologically smooth morphisms can be used to check the smoothness of objects:

\begin{theorem}\label{thm:sm-obj-is-loc}
    Consider a collection of commutative diagram
    \[\begin{tikzcd}
            Z_i \arrow["g_i"]{r} \arrow[swap,"h_i"]{rd} & Y \arrow["f"]{d} \\
            & X
        \end{tikzcd}\]
    for $i \in I$ of morphisms in $E(\mathfrak G)$ with all $g_i$ cohomologically smooth and $A \in \mc D(Y)$, where $I$ is some index set. If $(g_i^*)_{I}: \mc D(Y) \to \prod_I \mc D(Z_i)$ is conservative, then the following are equivalent:
    \begin{enumerate}[label=\roman*)]
        \item  $A$ is $f$-smooth.
        \item  $g_i^*A$ is $h_i$-smooth for all $i$.
    \end{enumerate}
    Moreover, if $i)$ or $ii)$ holds then the Verdier-dual $\mb D_{h_i}(g_i^*A)$ of $g_i^*A$ is given by $g_i^! \mb D_f(A)$ and $\mb D_f(A)$ is invertible if and only if $\mb D_{h_i}(g_i^* A)$ is invertible for all $i$.
\end{theorem}
\begin{proof}
    If $A$ is $f$-smooth, then $g_i^*A$ is $h_i$-smooth for all $i$ by \Cref{lmm:sm obj composition}. To show the converse, we want to use the criterion from \Cref{lmm:sm obj crit} iii). Assume $g_i^* A$ is $h_i$-smooth for all $i$ and consider the natural map
    $$\Delta_f^!(p_{Y,2}^* \mb D_f(A) \tensor p_{Y,1}^* A) \to \Delta_f^! \iHom(p_{Y,2}^* A, p_{Y,1}^! A) = \iHom(A,A)$$
    where the identification on the right follows from the interior pullback formula (\Cref{PF-adj-isos}). We need to show that this map is an isomorphism, and since $(g_i^*)_{I}$ is conservative, we may check this after pulling back along each of the $g_i^*$. Thus, wlog.\ assume $I=\set{0}$ and write $g:=g_0$, $h.=h_0$, $Z:=Z_0$. Then the right hand side identifies with $\iHom(g^* A, g^* A)$ by \Cref{lmm:coh sm interior pull}. We claim that the pullback of the left hand side agrees with $\Delta_h^!(p_{Z,2}^* \mb D(g^* A) \tensor p_{Z,1}^* g^* A)$. Indeed, note that $\mb D_h(g^* A) \cong \iHom(g^* A, h^! 1_X) \cong g^! \iHom(A, f^! 1_X) \cong g^! \mb D_f(A)$ by interior pullback. Then, using \Cref{lmm:upper shriek for invertibles},
    \begin{align*}
        \Delta_h^!(p_{Z,2}^* g^! \mb D_f(A) \tensor p_{Z,1}^* g^* A) &\cong \Delta_h^!(p_{Z,2}^* g^* \mb D_f(A) \tensor p_{Z,1}^* g^* A \tensor p_{Z,2}^* g^! 1_Y)\\
        &\cong\Delta_h^!(p_{Z,2}^* g^* \mb D_f(A) \tensor p_{Z,1}^* g^* A) \tensor g^! 1_Y \\
        &\cong \Delta_h^!(g'^* p_{Y,2}^* \mb D_f(A) \tensor g'^* p_{Z,1}^* A) \tensor g^! 1_Y
    \end{align*}
    where $g': Z \times_X Z \to Y \times_X Y$ is the obvious map, given by the composition $Z \times_X Z \to Z \times_X Y \to Y \times_X Y$ of two base changes of $g$; in particular it is cohomologically smooth and one calculates $g'^! 1_{Y \times_X Y} = (p_{Z,1}^* g^! 1_Y) \tensor (p_{Z,2}^* g^! 1_Y)$. So
    \begin{align*}
        \Delta_h^!(g'^*(p_{Y,2}^* \mb D_f(A) \tensor p_{Z,1}^* A)) \tensor g^! 1_Y &= \Delta_h^! g'^! (p_{Y,2}^* \mb D_f(A) \tensor p_{Z,1}^* A) \tensor (\Delta_h^* g'^! 1_Y)^{-1} \tensor g^! 1_Y \\
        &\cong g^! \Delta_f^! (p_{Y,2}^* \mb D_f(A) \tensor p_{Z,1}^* A) \tensor (g^! 1_Y)^{-1} \\
        &\cong g^* \Delta_f^! (p_{Y,2}^* \mb D_f(A) \tensor p_{Z,1}^* A)
    \end{align*}
    as claimed. We have shown that pullback along $g$ yields a map
    $$\Delta_h^!(p_{Z,2}^* \mb D_h(g^* A) \tensor p_{Z,1}^* g^* A) \to \Delta_h^! \iHom(p_{Z,2}^* A, p_{Z,1}^! A)$$
    It remains to show that it agrees (up to homotopy coherence) with the natural map from \Cref{lmm:sm obj crit} which is an isomorphism by assumption. This can be done by an explicit calculation.\par
    The first statement about the Verdier duals follows from interior pullback (\Cref{PF-adj-isos}). The "only if" in the second statement is then obvious as all $g_i^! 1_Y$ are invertible by assumption. For the "if" direction we need to check that the natural map $\mb D_f(A) \tensor \iHom(\mb D_f(A), 1_Y) \to 1_Y$ is an isomorphism which is true as by \Cref{lmm:coh sm interior pull} its pullback along each $g_i$ is an isomorphism (as $g_i^* \mb D_f(A) \cong g_i^! \mb D_f(A) \tensor (g_i^! 1_Y)^{-1}$ is invertible by the first statement).
\end{proof}

In particular, cohomological smoothness is cohomologically smooth-local on the source in the following sense:

\begin{corollary}\label{sm-is-loc}
    Consider a collection of commutative diagrams
    \[\begin{tikzcd}
            Z_i \arrow["g_i"]{r} \arrow[swap,"h_i"]{rd} & Y \arrow["f"]{d} \\
            & X
        \end{tikzcd}\]
    for $i \in I$ of morphisms in $E(\mathfrak G)$ with all $g_i$ cohomologically smooth and $A \in \mc D(Y)$, where $I$ is some index set. If $(g_i^*)_{I}: \mc D(Y) \to \prod_I \mc D(Z_i)$ is conservative, then the following are equivalent:
    \begin{enumerate}[label=\roman*)]
        \item  All $h_i$ are cohomologically smooth.
        \item  $f$ is cohomologically smooth.
    \end{enumerate}
\end{corollary}
\begin{proof}
    This follows immediately from \Cref{cor:sm obj sm mor} and \Cref{thm:sm-obj-is-loc}.
\end{proof}

\subsection{Cohomologically étale and cohomologically proper morphisms}

A cohomologically étale morphism is, roughly speaking, a morphism $f$ for which $f^!$ and $f^*$ agree "universally" (with the additional property that this condition is stable under taking the diagonal). For example, consider $\mc D_{\mb F_\ell}$ on $\mathfrak G$ from \Cref{thm:et 6 functors}. We certainly expect that $f^! = f^*$ for any étale morphism; and indeed, we will see below that any étale morphism is cohomologically étale in this setting (whence the name). %

\begin{definition}\label{def:coh et prop}\cite[Definitions 6.10 and 6.12]{Scholze1}
    Let $\mc D$ be a 6-functor formalism on a construction setup $\mathfrak G$ and let $(f:Y \to X) \in E(\mathfrak G)$. Then
    \begin{enumerate}[label=\roman*)]
        \item $f$ is called ($\mc D$-)\textit{cohomologically étale} if it is an isomorphism or it is $n$-truncated for some $n$ with $\Delta_f$ cohomologically étale and $1_Y$ $f$-smooth.
        \item $f$ is called ($\mc D$-)\textit{cohomologically proper} if it is an isomorphism or it is $n$-truncated for some $n$ with $\Delta_f$ cohomologically proper and $1_Y$ $f$-proper.
    \end{enumerate}
\end{definition}
\begin{remark}
    This is an inductive definition which works by \cite[Lemma 5.5.6.15]{HTT}: A morphism is $n$-truncated if and only if its diagonal is $n-1$-truncated. But a $-2$-truncated morphism is an isomorphism. Note that $\Delta_f \in E$ if $f \in E$ as the construction setup $\mathfrak G$ is in particular a good geometric setup, cf. \Cref{rem:construction setup}.
\end{remark}
\begin{example}\label{lmm:proper is coh proper}
    If $\mc D$ is a 6-functor formalism on a construction setup $\mathfrak G=(\mc C, E, I, P)$ then all $n$-truncated (for some $n \geq -2$) morphisms in $I$ are cohomologically étale and all $n$-truncated morphisms in $P$ are cohomologically proper: As $\Delta_f \in I$ (resp. $\Delta_f \in P$) if $f \in I$ (resp. $f \in P$) this follows from the fact that $1$ is $f$-smooth (resp. $f-$proper) if $f \in I$ (resp. $f \in P$).
\end{example}
We will now again restrict the discussion to the case of cohomologically étale morphisms.\\
To make the terminology complete, we introduce a non-standard notion.
\begin{definition}\label{def:coh unram}
    In the situation of \Cref{def:coh et prop}, we call $f$ \textit{cohomologically unramified} if $\Delta_f$ is cohomologically étale.
\end{definition}

\begin{lemma}\cite[Proposition 6.13]{Scholze1}\label{lmm:coh et coh sm unram}
    In the situation of \Cref{def:coh et prop}, $f$ is cohomologically étale if and only if it is cohomologically smooth and cohomologically unramified. More precisely, if $f$ is cohomologically unramified, then there is a natural transformation $f^! \to f^*$ and the following are equivalent:
    \begin{enumerate}[label=\roman*)]
        \item $f^! 1_X \to 1_Y$ is an isomorphism on global sections.
        \item $f^! \to f^*$ is an isomorphism.
        \item $f$ is cohomologically smooth.
        \item $f$ is cohomologically étale.
    \end{enumerate}
\end{lemma}
\begin{proof}
    By definition, a cohomologically étale morphism is cohomologically unramified, so the first statement really follows from the second.\par
    Assume $f$ is cohomologically unramified. Then it is $n$-truncated for some $n$. We argue by induction on $n$. If $n=-2$ then $f$ is an isomorphism and all statements hold (as $I(\mathfrak G)$ contains all isomorphisms). Now assume by induction that $ii)$ holds for $\Delta_f$. We have a natural transformation
    $$p_2^* f^!  \to p_1^! f^* $$
    adjoint to $p_{1,!} p_2^* f^! = f^* f_! f^! \to f^*$, where we used the projection formula. In fact, applied to $1_X$ this is the map from \Cref{cor:sm obj sm mor} resp. \Cref{lmm:sm obj crit}. Applying $\Delta_f^* = \Delta_f^!$ on both sides yields the promised map $f^! \to f^*$. Thus $i) \Leftrightarrow iv)$ by \Cref{lmm:sm obj crit} and the defintion. The same result also proves $i) \Rightarrow iii)$ as it implies that already $f^! 1_X = 1_Y$. But then also $i) \Rightarrow ii)$ by definition of cohomological smoothness and thus $i) \Leftrightarrow ii)$. It remains to show $iii) \Rightarrow iv)$, but this is obvious from the definition by \Cref{cor:sm obj sm mor}.
\end{proof}

\begin{proposition}\label{prop:coh unram composition bc}
Cohomologically unramified morphisms are stable under base change and composition.
\end{proposition}
\begin{proof}
    Let $f: Y \to X$ be a cohomologically unramified morphism, in particular $n$-truncated for some $n$. We induct on $n \geq -2$: If $f$ is an isomorphism, then any base change is an isomorphism and thus cohomologically unramified; moreover, compositions $g \circ f$ are cohomologically unramified for all cohomologically unramified $g$. Now assume stability under base change holds for all $n-1$-truncated $f$. If $f$ is $n$-truncated and $g:X' \to X$ is any morphism and $f'$ is the pullback of $f$ along $g$, then $\Delta_{f'}$ is the base change of $\Delta_f$ (which is cohomologically étale by assumption and $n-1$-truncated) along $Y' \times_{X'} Y' \to Y \times_X Y$, so it is cohomologically étale by the induction hypothesis (and \Cref{lmm:coh et coh sm unram}). This proves the first assertion. Now assume that the second assertion holds for compositions $g\circ f$ with $f$ $n-1$-truncated. If $f$ is $n$-truncated and $g:X \to S$ is cohomologically unramified, then $\Delta_{g \circ f}$ can be factored as $\phi \circ \Delta_f: Y \to Y \times_X Y \to Y \times_S Y$, where $\phi$ is the base change of $\Delta_g$ along $(f,f)$; by the first assertion $\phi$ is thus cohomologically unramified, but it is also cohomologically smooth since $\Delta_g$ is and so it is cohomologically étale by \Cref{lmm:coh et coh sm unram}. By induction the composition $\phi \circ \Delta_f$ is cohomologically unramified (since $\Delta_f$ is $n-1$-truncated and cohomologically étale) and cohomologically smooth by \Cref{cor:coh sm composition}. But then it is cohomologically étale.
\end{proof}
\begin{corollary}\label{cor:coh et composition bc}
    Cohomologically étale morphisms are stable under base change and composition.
\end{corollary}
\begin{proof}
    This is immediate from \Cref{cor:coh sm composition}, \Cref{prop:coh unram composition bc} and \Cref{lmm:coh et coh sm unram}.
\end{proof}

Let us quickly make some additional observations about cohomologically unramified morphisms:

\begin{remark}\label{rem:unram}
    In the situation of \Cref{def:coh et prop}, let $f: Y \to X \in E(\mathfrak G)$ be a cohomologically unramified morphism.
    \begin{enumerate}[label=\roman*)]
        \item Cohomological smoothness can be cancelled along $f$: A morphism $g: Z \to Y$ is cohomologically smooth (resp. étale) if and only if $f\circ g$ is. This follows from stability of cohomological smooth morphisms under base change and composition by the usual argument (cf. the proof of \Cref{prop:G_k suitable decomp}).
        \item Let $A, S \in \mc D(Z)$. Then the natural map from \Cref{lmm:sm obj crit} becomes $\mb D_f(A) \tensor S \to \iHom(A,S)$, using that $\Delta_f^! = \Delta_f^*$, and so $A$ is $f$-smooth if and only if this map is an isomorphism for all $S$ (where $S=A$ suffices). In particular, if $A$ is $f$-smooth, its Verdier dual agrees with the naive dual; but then the the case $S = A$ says that $A$ is dualisable.\footnote{One checks that the natural map is indeed just adjoint to evaluation tensored with $A$; it is an isomorphism if and only if $A$ is dualisable, cf. \cite[Lemma 6.2]{Mann2} or combine \Cref{cor:id sm is dualizable} and \Cref{lmm:sm obj crit}.} Thus the $f$-smooth objects are exactly the dualisable $A$ for which $\mb D_f(A) \to \iHom(A,1)$ (induced by the natural map $f^! 1 \to 1$ from \Cref{lmm:coh et coh sm unram}) is an isomorphism. In particular, if $f$ is cohomologically étale, then $f$-smooth objects are exactly the dualisable ones.
    \end{enumerate}
\end{remark}

The following result is helpful to find non-trivial examples of cohomologically étale morphisms.

\begin{lemma}\label{lmm:descent coh et}
    In the situation of \Cref{def:coh et prop}, consider a cartesian square
    \[\begin{tikzcd}
        Y' \rar["f'"] \dar["g'"] & X' \dar["g"] \\
        Y \rar["f"] & X
    \end{tikzcd}\]
    with $g$ of universal $\mc D^*$-descent. Then $f'$ is cohomologically étale if and only if $f$ is.
\end{lemma}
\begin{proof}
    This follows from the definition and \Cref{prop:sm obj descent}.
\end{proof}

The following is the most important application for us:

\begin{proposition}[{\cite[Proposition 7.18]{Scholze1}}]\label{prop:et is coh et}
    Let $\Lambda$ be any torsion ring and $\mc D_\Lambda$ the 6-functor formalism on $\mathfrak G$ from \Cref{thm:et 6 functors}. Then all étale morphisms $f:Y \to X$ in $E(\mathfrak G)$ are $\mc D_\Lambda$-cohomologically étale.
\end{proposition}
\begin{proof}
    As the diagonal of étale morphisms is an open immersion, we know that all étale morphisms are cohomologically unramified (recall that $I(\mathfrak G)$ is the set of open immersions). As étale morphisms are open, we can reduce to the case that $f$ is surjective. Then we make use of \Cref{lmm:descent coh et}: As étale morphisms satisfy universal $\mc D^*$-descent for $\mc D_\Lambda$, $f$ is cohomologically étale if and only if this holds for the base change $p: Y \times_X Y \to Y$ of $f$ along itself. Now it suffices to check that $p^! 1_X \to 1_Y$ is an isomorphism by \Cref{lmm:coh et coh sm unram}, which may be checked on the cover given by $\Delta_f$ and its complement $U$. Both are open subschemes, and on $\Delta_f$ this clearly holds. By induction on the maximum cardinality of a geometric fiber of $f$ we reduce to the case $U = \emptyset$ and win.
\end{proof}
\begin{remark}
    We will use the idea of this proof later again to prove our \Cref{thm:qf open unram awn etale}.
\end{remark}

As a corollary, we can now prove the following criterion due to B. Zavyalov:

\begin{theorem}[{\cite[Theorem 3.3.3]{Zav}}]\label{thm:sm coh sm crit}
    Let $S$ be a qcqs scheme and consider the restricted 6-functor formalism $\mc D_\Lambda$ on $\mathfrak G_{/S}$ (cf. \Cref{rem:6-functor slices}). In the situation of \Cref{prop:et is coh et}, smooth morphisms in $\mc C(\mathfrak G_{/S})$ are $\mc D_\Lambda$-cohomologically smooth if and only if $\mb P^1_S \to S$ is $\mc D_\Lambda$-cohomologically smooth.
\end{theorem}
\begin{proof}
    The "only if" is clear, we need to show the "if"-direction. So assume that $\mb P^1_S \to S$ is cohomologically smooth. Then $\mb A^1_S \to S$ is as well since cohomological smoothness is stable under composition (\Cref{cor:coh sm composition}). Together with stability under base change this implies that $\mb A^n_S \to S$ is cohomologically smooth for all $n$ and thus $\mb A^n_X \to X$ is cohomologically smooth for all $(X \to S) \in \mc C(\mathfrak G_{/S})$. Now let $Y \to X$ be a smooth morphism in $\mc C(\mathfrak G_{/S})$. Then it is of finite presentation, so it is contained in $E(\mathfrak G_{/S})$. As cohomological smoothness can be checked Zariski-locally on the source by \Cref{sm-is-loc}, we may assume that $Y \to X$ factors via an étale morphism over some $\mb A^n_X$; and conclude by \Cref{prop:et is coh et} and \Cref{cor:coh sm composition}.
\end{proof}

Finally let us note that the property of being cohomologically étale is cohomologically étale-local on the source in the following sense:

\begin{corollary}\label{et-is-loc}
    Consider a collection of commutative diagrams
    \[\begin{tikzcd}
            Z_i \arrow["g_i"]{r} \arrow[swap,"h_i"]{rd} & Y \arrow["f"]{d} \\
            & X
        \end{tikzcd}\]
    for $i \in I$ of morphisms in $E(\mathfrak G)$ with all $g_i$ cohomologically étale and $A \in \mc D(Y)$, where $I$ is some index set. Assume $f$ is $n$-truncated for some $n$. If $(g_i^*)_{I}: \mc D(Y) \to \prod_I \mc D(Z_i)$ is conservative, then the following are equivalent:
    \begin{enumerate}[label=\roman*)]
        \item  All $h_i$ are cohomologically étale.
        \item  $f$ is cohomologically étale.
    \end{enumerate}
\end{corollary}
\begin{proof}
    $ii) \Rightarrow i)$ is \Cref{cor:coh et composition bc}.\par
    Let us prove $i) \Rightarrow ii)$. We induct on $n \geq -2$. If $f$ is an isomorphism, then it is cohomologically étale in any case (by definition). Now assume the statement holds for $n-1$-truncated $f$. We claim that there are commutative diagrams
    \[\begin{tikzcd}
        Z_i \arrow["g_i"]{r} \arrow[swap,"h'_i"]{rd} & Y \arrow["\Delta_f"]{d} \\
        & Y \times_X Y
    \end{tikzcd}\]
    with all $h'_i$ cohomologically étale. Indeed, by assumption all $\Delta_{h_i}: Z_i \to Z_i \times_X Z_i$ are cohomologically étale; but then $h_i': Z_i \to Z_i \times_X Z_i \to Z_i \times_X Y \to Y \times_X Y$ does the job by stability of cohomologically étale morphisms under base change and compositions (\Cref{cor:coh et composition bc}). We conclude by the induction hypothesis that $\Delta_f$ is cohomologically étale, i.e. $f$ cohomologically unramified and the statement follows from \Cref{sm-is-loc} and \Cref{lmm:coh et coh sm unram}.
\end{proof}

For the arguments in this paper we will only need the following instances of \Cref{sm-is-loc} and \Cref{et-is-loc}:

\begin{lemma}\label{lmm:et-is-loc-et}\label{lmm:sm-is-loc-et}
    Let $f: Y \to X$ be a morphism in $E(\mathfrak G)$ and $g: \coprod_I Y_i \to Y$ an étale cover of $Y$. Then $f$ is cohomologically smooth (resp. étale) if and only if each $h_i: Y_i \to X$ is cohomologically smooth (resp. étale).
\end{lemma}
\begin{proof}
        By \Cref{prop:et is coh et}, étale morphisms are cohomologically étale which implies the statement by \Cref{sm-is-loc} resp. \Cref{et-is-loc} as $g^*$ is indeed conservative (the étale topos has enough points).
\end{proof}

Moreover, in this paper, the only thing we need to know about the class of $\mc D_\Lambda$-cohomologically proper morphisms is that it contains the universally closed morphisms in $E(\mathfrak G)$ which follows from the construction, cf. \Cref{lmm:proper is coh proper}. In particular, we have the following result:

\begin{corollary}\label{cor:proper shriek sm objects}
    If $f$ is a universally closed morphism in $E(\mathfrak G)$ and $A$ is $f$-smooth, then $f_! A$ is dualisable.
\end{corollary}
\begin{proof}
    Combine \Cref{lmm:proper is coh proper} and \Cref{cor:dualizable+smooth}.
\end{proof}

\newpage

\section{The geometry of cohomologically étale morphisms}\label{section:char of coh et}
\begin{notation}\label{not:basic setup}
    In this section, we will always work with the 6-functor formalism $\mc D_\Lambda$ on the geometric setup $\mathfrak G$ from \Cref{thm:et 6 functors} for some non-zero finite torsion ring $\Lambda$. The six functors and the terms cohomologically smooth, cohomologically étale etc. will always refer to this formalism.
\end{notation}
\subsection{Cohomologically étale morphisms and the absolute weak normalisation}
Our goal in this subsection is to prove that a separated morphism of finite presentation in $\mc C(\mathfrak G)$ is cohomologically étale if and only if its absolute weak normalisation (see \Cref{lmm:awn-adjoint}) is étale. For this, we first need to take another look at absolutely weakly normal schemes.\par
We start with some basic properties which are preserved under absolute weak normalisation:

\begin{proposition}\label{prop:prop of awn}
    Let $f: Y \to X$ be a morphism in $\Sch$. Then the following properties hold for $f^\mathrm{awn}$ if they hold for $f$:
    \begin{enumerate}[label=\roman*)]
        \item a closed immersion
        \item an open immersion
        \item an immersion
        \item étale
        \item any property of universal homeomorphisms which is stable under base change and composition
    \end{enumerate}
\end{proposition}
\begin{proof}
    Note that absolute weak normalisation does not change the underlying topological spaces as $X^\mathrm{awn} \to X$ is a universal homeomorphism (cf. \Cref{lmm:awn-adjoint}). In fact, we observed there that $X^\mathrm{awn} \cong \ilim_{X' \to X \in \mc U'_X} X'$ where $\mc U'_X$ is the category of universal homeomorphisms $X' \to X$ of finite type (note that it is directed). Equivalently, it is the scheme with underlying topological space $|X|$ and structure sheaf $\colim_{X' \to X \in \mc U'_X} \mc O_{X'}$. This implies $i),ii)$ and $iii)$.\par
    For $iv)$, assume that $f$ is étale. We claim that $Y \times_{X} X^\mathrm{awn}$ is already absolutely weakly normal which is of course sufficient. Indeed, this claim reduces to the affine case which is proven in \cite[Proposition B.6]{rydh}. %
    \par
    Finally, assume that $f$ satisfies a property of universal homeomorphisms which is stable under base change and composition. Then in particular the base change $Y \times_X X^\mathrm{awn} \to X^\mathrm{awn}$ satisfies this property; but the morphism $Y^\mathrm{awn} \to X^\mathrm{awn}$ is just the composition $Y^\mathrm{awn} \to Y \times_X X^\mathrm{awn} \to X^\mathrm{awn}$ where the first morphisms is a universal homeomorphism. %
\end{proof}

\begin{corollary}\label{cor:awn-etale}
    If $f: Y \to X$ is étale and $X$ is a.w.n., then $Y$ is a.w.n..
\end{corollary}
\begin{proof}
    This follows from the proof of $iv)$ in \Cref{prop:prop of awn}.
\end{proof}

Recall from \Cref{prop:awn 6 functors} that the formalism $\mc D_\Lambda$ factors over the full subcategory $\mathrm{Awn}$ of absolutely weakly normal schemes via the absolute weak normalisation. We can turn $\mathrm{Awn}$ into a site, using the $v$-topology:

\begin{definition}[{\cite[Definition 1.1]{arc}, \cite[\texttt{0ETS}]{Stacks}}]
  A map $f: Y \to X$ of qcqs schemes is a \textit{$v$-cover} if for any valuation ring $V$ and map $\Spec(V) \to X$ there is an extension $V \to W$ of valuation rings and a map $\Spec(W) \to Y$ lifting the composition $\Spec(W) \to \Spec(V) \to X$. An \textit{$h$-cover} is a $v$-cover of finite presentation.\par
  The \textit{v-topology} on the category $\mathrm{Awn}$ of qcqs a.w.n. schemes is the Grothendieck topology where covering families $\set{f_i: Y_i \to X}_{i \in I}$ are families with the following property: for any affine open $V \subseteq X$ there is a map $t:J \to I$ of sets with $J$ finite and affine opens $U_j \subseteq f_{t(k)}^{-1}(V)$ such that $\coprod_J U_j \to V$ is a $v$-cover. We call $\mathrm{Awn}$ the \textit{absolutely weakly normal site}.
\end{definition}
The following result will be important for us:
\begin{lemma}\label{awn-subcanonical}
  The site $\mathrm{Awn}$ is subcanonical.
\end{lemma}
\begin{proof}
  Let $S$ be a qcqs a.w.n. scheme. We need to show that for a $v$-cover $Y \to X$ with $Y$ and $X$ a.w.n.
  \[\begin{tikzcd}[column sep=1.5em]
    {\Hom(X,S)} & {\Hom(Y,S)} & \Hom(Y \times_X Y, S)
    \arrow[from=1-1, to=1-2]
    \arrow[shift right, from=1-2, to=1-3]
    \arrow[shift left, from=1-2, to=1-3]
\end{tikzcd}\]
is exact. By a standard argument (see the paragraph following \cite[Lemma 2.61]{fga-explained}) this reduces to the case that $Y,X$ and $S$ are affine, say $Y = \Spec B$, $X=\Spec A$, $S= \Spec R$ so that it suffices to show exactness of
\[\begin{tikzcd}[column sep=1.5em]
  A & B & B \tensor_A B
  \arrow[from=1-1, to=1-2]
  \arrow[shift right, from=1-2, to=1-3]
  \arrow[shift left, from=1-2, to=1-3]
\end{tikzcd}\]
We can write the $v$-cover $Y \to X$ as a directed limit of $h$-covers $Y_i = \Spec B_i \to X$, cf. \cite[\texttt{0EVP}]{Stacks}. By \cite[\texttt{0EVT}]{Stacks}, the presheaf $\mc O^{\mathrm{awn}}$ on $\Sch$ which sends a scheme $X$ to $\Gamma(X^\mathrm{awn},\mc O_{X^\mathrm{awn}})$ satisfies the sheaf property for $h$-coverings. Thus every
\[\begin{tikzcd}[column sep=1.5em]
  A & B_i^\mathrm{awn} & B_i^\mathrm{awn} \tensor_A B_i^\mathrm{awn}
  \arrow[from=1-1, to=1-2]
  \arrow[shift right, from=1-2, to=1-3]
  \arrow[shift left, from=1-2, to=1-3]
\end{tikzcd}\]
is exact and taking a colimit of these diagrams yields the claim, using that $(\blank)^\mathrm{awn}$ (as a functor on rings) permutes with colimits.
\end{proof}

As the 6-functor formalism $\mc D_\Lambda$ factors over $\mathrm{Awn}$, we immediately observe the following:

\begin{proposition}\label{prop:awn etale implies coh etale}
    Let $f:Y\to X$ be a morphism in $E(\mathfrak G)$. If $f^\mathrm{awn}$ is étale, then $f$ is cohomologically étale.
\end{proposition}
\begin{proof}
    By \Cref{prop:awn 6 functors} $\mc D_\Lambda$ factors over the absolute weak normalisation which implies the statement as étale morphisms are cohomologically étale (\Cref{prop:et is coh et}).
\end{proof}

Let us now prove the converse for morphisms of finite presentation. As cohomologically étale morphisms are in particular cohomologically unramified (recall \Cref{def:coh unram}), we start with an analysis of the latter class of morphims.\par
 It is a well-known fact that unramified morphism are (locally) quasi-finite, cf. \cite[\texttt{02V5}]{Stacks}. We claim that cohomologically unramified morphisms of finite type are as well. In fact, we can even give a full geometric characterisation of these morphisms. We introduce a non-standard term:

\begin{definition}\label{def:top unram}
    We call a morphism $f$ topologically unramified if its diagonal is open (as a morphism of topological spaces).
\end{definition}

We claim that cohomologically unramified morphisms are exactly the topologically unramified ones. This follows from the next observation:

\begin{lemma}\label{lmm:coh unram open}
    Let $g: Y \to X$ be a universally closed morphism in $E(\mathfrak G)$ such that $\Lambda$ is $g$-smooth. Then $g(Y)$ is open and closed in $X$.
\end{lemma}
\begin{proof}
     By \Cref{cor:proper shriek sm objects}, $g_! \Lambda$ is dualisable. This implies that it has open and closed support by \Cref{rem:dualizable-perfect}. But it is exactly supported on $g(Y)$: By proper base change $(g_! \Lambda)_x = (g_* \Lambda)_x \cong R\Gamma(Y_x, \Lambda)$ for any (geometric) point $x$ in $X$ and we have $Y_x = \emptyset \Leftrightarrow R\Gamma(Y_x, \Lambda) = 0$ (since $\Lambda \not=0$).
\end{proof}

Then we only need the following easy fact:

\begin{lemma}\label{lmm:coh-et immersions}
    If $f: Z \to X$ is a closed immersion in $E(\mathfrak G)$, then $f(Z)$ is open in $X$ if and only if the induced morphism $f^\mathrm{red}$ of the reductions is an open immersion.
\end{lemma}
\begin{proof}
    If $f^\mathrm{red}$ is an open immersion, then $f$ is open as reduction does not change the underlying topological space. Conversely, assume $f$ is open, then we may assume that it is surjective as $f(Z)$ is open. Then we claim that $Z^\mathrm{red} \to X^\mathrm{red}$ is an isomorphism. Wlog.\ assume that both schemes are affine. Then surjectivity of $Z^\mathrm{red} \to X^\mathrm{red}$ implies that the corresponding surjective ring map is injective and thus an isomorphism.
\end{proof}

\begin{corollary}\label{cor:coh-unramified}
    A morphism $f \in E(\mathfrak G)$ is cohomologically unramified if and only if it is topologically unramified.
\end{corollary}
\begin{proof}
    If $f$ is topologically unramified then by \Cref{lmm:coh-et immersions} the absolute weak normalisation of its diagonal is cohomologically étale (as it is an open immersion, cf. \Cref{prop:prop of awn}) and thus the diagonal itself is cohomologically étale by \Cref{prop:awn etale implies coh etale}. Conversely, assume that $\Delta_f$ is cohomologically étale. As morphisms in $E(\mathfrak G)$ are separated by definition, \Cref{lmm:coh unram open} applies and we see that the image of $\Delta_f$ is open; but then it is already open as a map of topological spaces because it is an immersion.
\end{proof}

Let us now prove that topologically unramified morphisms of finite type are quasi-finite.

\begin{lemma}\label{lmm:coh unram-unram}
    Let $k$ be a field. Then a morphism $f: Y \to \Spec k$ of finite type is topologically unramified if and only if $f^\mathrm{red}: Y^\mathrm{red} \to \Spec k$ is étale.
\end{lemma}
\begin{proof}
    Assume $f$ is topologically unramified and let $\bar k$ be the algebraic closure of $k$. Then $f_{\bar k}:Y_{\bar k} \to \Spec \bar k$ is topologically unramified as well, and so by fpqc descent for étale morphisms wlog $f=f_{\bar k}$. As $f^\mathrm{red}$ is of finite type, we only need to check that its diagonal is an open immersion; then it is unramified and thus étale (it is clearly flat and of finite presentation). But the diagonal is open by assumption and $Y^\mathrm{red} \times_{\Spec k} Y^\mathrm{red}$ is still reduced as $k$ is algebraically closed, so that we conclude by \Cref{lmm:coh-et immersions}. The converse is obvious.
\end{proof}

\begin{lemma}\label{lmm:top unram bc}
    Topologically unramified morphisms are stable under base change.
\end{lemma}
\begin{proof}
    Let $f:Y \to X$ be a topologically unramified morphism and $f':Y' \to X'$ its base change along some $g: X' \to X$. Then $\Delta_{f'}$ is the base change of $\Delta_f$ along $Y' \times_{X'} Y' \to Y \times_X Y$. As $\Delta_f$ is an immersion, it is universally open if it is open, which proves the claim.
\end{proof}

\begin{corollary}\label{cor:coh-et implies quasi-finite}
    A topologically unramified morphism $f: Y \to X$ of finite type is quasi-finite. In particular, cohomologically unramified and cohomologically étale morphisms of finite type are quasi-finite.
\end{corollary}
\begin{proof}
    The second statement follows from the first by \Cref{cor:coh-unramified}. Assume $f$ is topologically unramified and of finite type. We claim that the fibers of $f$ are finite. Indeed, this is clear from \Cref{lmm:coh unram-unram} and the description of schemes étale over a field (\cite[\texttt{02GL}]{Stacks}) since topologically unramified morphisms are stable under base change by \Cref{lmm:top unram bc}.
\end{proof}

So we know that cohomologically étale morphisms are quasi-finite and cohomologically smooth.
By the next result, this has a very useful consequence.

\begin{proposition}\label{prop:quasi-finite coh sm open}
    Let $f \in E(\mathfrak G)$ be a quasi-finite cohomologically smooth morphism which is of finite presentation. Then $f$ is universally open.
\end{proposition}
\begin{proof}
    As all the assumptions on $f$ are stable under base change, we only need to show that $f$ is open and it suffices to show that $f$ is generalising by the finiteness condition, cf. \cite[\texttt{01U1}]{Stacks}. So take $y \in Y$ and pick a generalisation $x \rightsquigarrow f(y)$ which is realised by $\Spec R \to X$ for some discrete valuation ring $R$. Let $R^h$ be the henselisation of $R$ and $Z$ the fiber product $\Spec R^h \times_{X} Y$; let $z \in \Spec R^h$ be the point mapping to $f(y)$ and $\eta \in \Spec R^h$ the point mapping to $x$.
    \[\begin{tikzcd}
            Z_1 \rar \drar[swap]{g} & Z \rar \dar \ar[dot corner'=south east] & Y \dar \\
            & \Spec R^h \rar & X
        \end{tikzcd}\]

    $Z \to \Spec R^h$ is cohomologically smooth and quasi-finite as $Y \to X$ is. As $R^h$ is henselian, this implies $Z \cong Z_1 \coprod Z_2$ such that $g:Z_1 \to \Spec R^h$ is finite and $Z_2 \times_{\Spec R^h} z = \emptyset$ (cf. \cite[\texttt{04GG}]{Stacks}); in particular we have $(z,y) \in g^{-1}(z)$. Moreover, $Z_1$ is the (finite) disjoint union of the spectra of its stalks at closed points (as it is finite over a henselian ring); thus, replacing $Z_1$ by its stalk at the point $(z,y)$, $Z_1 \to Z$ is still an open immersion and so $g$ remains cohomologically smooth and finite. But then by \Cref{cor:proper shriek sm objects} $g_! \Lambda$ is dualisable and so by \Cref{rem:dualizable-perfect} its support is either $\Spec R^h$ or empty since $\Spec R^h$ is connected. Note that $g_! \Lambda$ is non-zero since $(g_! \Lambda)_z \cong R\Gamma(g^{-1}(z),\Lambda) \not= 0$ (as $g^{-1}(z) \not= \emptyset$) by proper base change. So we must have $0 \not= (Rg_! \Lambda)_\eta \cong R\Gamma(g^{-1}(\eta),\Lambda)$ and thus we find a lift of $\eta$ in $Z_1$, yielding a generalisation $y' \rightsquigarrow y$ in $Y$ with $f(y') = x$.
\end{proof}
\begin{remark}
    By arguing as in \cite[Proposition 23.11]{Scholze2} and using the results from \cite[Section 6.4]{Scholze3}, the finite presentation assumption in \Cref{prop:quasi-finite coh sm open} can be dropped whenever the target $X$ satisfies the following property for each $n$: all affine $U \in (\mb A^n_X)_\et$ have $\Lambda$-cohomological dimension $\leq d$ for some fixed $d \in \mb N$ (here we also used \cite[\texttt{0F31}]{Stacks}).
\end{remark}

\begin{corollary}\label{cor:et-is-univ-open}
    Cohomologically étale morphisms of finite presentation are universally open.
\end{corollary}
\begin{proof}
    This is immediate from \Cref{cor:coh-et implies quasi-finite} and \Cref{prop:quasi-finite coh sm open}.
\end{proof}

To finish the proof, we claim that if $f:Y \to X$ is a separated topologically unramified and universally open morphism of finite type with $Y$ and $X$ qcqs, then $f^\mathrm{awn}$ is étale. The idea is to use that étale morphisms satisfy $v$-descent in $\mathrm{Awn}$ to do induction on the maximum cardinality of a geometric fiber of the morphism. The induction is made possible by the following lemma:

\begin{lemma}
    \label{lmm:count-geometric-fibres}
    Let $f : Y \to X$ be a quasi-finite morphism where $Y$ is quasi-compact. Then there is some minimal $n_f \in \mb N$ such that all fibers of $f$ have at most $n_f$ points.
\end{lemma}
\begin{proof}
This is \cite[\texttt{03JA}]{Stacks}. Using that $Y$ is quasi-compact one reduces to the case that $Y$ and $X$ are affine, then one reduces via Zariski's main theorem to the case that $f$ is finite where it becomes a straightforward calculation.
\end{proof}

Next, let us state the descent property that we want to use.

\begin{lemma}\label{lmm:arc-descent for etale morphisms}
Let $f: Y \to X$ be a separated morphism of qcqs a.w.n. schemes and $g: X' \to X$ a cover in the absolutely weakly normal site $\mathrm{Awn}$. Let $f': Y' \to X'$ be the pullback of $f$ along $g$. Then $f'$ is étale if and only if $f$ is étale.
\end{lemma}
\begin{proof}
  By \cite[Theorem 5.6]{arc} the functor $\mc G: \mathrm{Sch}^\mathrm{op}_{\mathrm{qcqs},\mb F_p} \to \mathrm{Cat}_1$ that sends a qcqs scheme $X$ to the category of separated etale schemes over $X$ is a $v$-sheaf. By \Cref{cor:awn-etale}, the restriction of $\mc G$ to $\mathrm{Awn}$ sends a qcqs scheme $X$ to the category of separated etale a.w.n. schemes over $X$. As $\mathrm{Awn}$ is subcanonical by \Cref{awn-subcanonical}, the statement follows by applying the formal \Cref{lmm:subcanonical-descent} below to $\mathrm{Awn}$ and the class of separated étale morphisms.
\end{proof}

\begin{lemma}\label{lmm:subcanonical-descent}
Let $\mc C$ be a subcanonical site, $P$ a homotopy class of morphisms in $\mc C$ which is stable under base change, and $\mc F: \mc C^\mathrm{op} \to \mathrm{Cat}$ the functor sending $X$ to the category of $(Y \to X) \in P$. Assume $\mc F$ is a sheaf. Then in any cartesian diagram
\[\begin{tikzcd}
Y' \rar{f'} \dar{g'} & X' \dar{g} \\
Y \rar{f} \rar & X
\end{tikzcd}\]
where $g$ is a cover in $\mc C$, $f \in P$ if and only if $f' \in P$.
\end{lemma}
\begin{proof}
$\Rightarrow$ holds by assumption. So assume that $f' \in P$, i.e. $f' \in \mc F(X')$. Let $p_1, p_2$ denote the projections $X'' := X' \times_X X' \to X'$. By assumption, $p_1^* f', p_2^* f' \in \mc F(X'')$. Moreover, there is a canonical isomorphism
$$p_1^* Y' \cong Y' \times_X X' \cong Y \times_X X' \times_X X' \cong X' \times_X Y \times_X X' \cong X' \times_X Y' \cong p_2^* Y'$$
over $X''$. As $\mc F$ is a sheaf, this implies that there is some $(h: Z \to X) \in \mc F(X)$ (in particular $h \in P$) and an isomorphism $\phi: g'^*Z \cong Y'$ over $X'$ such that
\[\begin{tikzcd}
	{Z \times_X X' \times_X X'} & {X' \times_X Z \times_X X'} \\
	{Y \times_X X' \times_X X'} & {X' \times_X Y \times_X X'}
	\arrow["{p_1^* \phi}", from=1-1, to=2-1]
	\arrow["\cong",from=1-1, to=1-2]
	\arrow["{p_2^* \phi}"', from=1-2, to=2-2]
	\arrow["\cong", from=2-1, to=2-2]
\end{tikzcd}\]
commutes. In other words, there is an isomorphism $Y' \times_Z Y' \cong Y' \times_Y Y'$ compatible with the projections.\\
Now for any $T \in \mc C$, we have that
$$h_T(Y) \to h_T(Y') \rightrightarrows h_T(Y' \times_Y Y')$$
is exact and the same holds with $Y$ replaced by $Z$; by the observations above, we thus get an induced isomorphism $h_T(Y) \cong h_T(Z)$ natural in $T$. So $Y \cong Z$ over $X$ and thus $f \in P$.
\end{proof}

Now we are ready to prove our claim.

\begin{theorem}\label{thm:qf open unram awn etale}
Let $f: Y \to X$ be a separated, finite type morphism of qcqs schemes. If $f$ is universally open and topologically unramified, then $f^{\mathrm{awn}}$ is étale.
\end{theorem}
\begin{proof}
    First, we can reduce to the case that $f$ is surjective: $f$ is open and so we can replace $X$ by the image of $f$ in $X$ (using \Cref{prop:prop of awn}). Note that $f$ is quasi-finite by \Cref{cor:coh-et implies quasi-finite}, so we can do induction on $n_f$ from \Cref{lmm:count-geometric-fibres}: If $n_f = 0$, then $Y = \emptyset$ so the statement holds. Now assume the statement is proven for all feasible $g$ with $n_g < n_f$. By \Cref{lmm:arc-descent for etale morphisms} it suffices to show that the base change $p$ of $f^\mathrm{awn}$ along itself is étale, as $f^\mathrm{awn}$ is universally open (by \Cref{prop:prop of awn}) and surjective, and in particular a $v$-cover by \cite[Remark 2.5]{rydh}. Since $f$ is topologically unramified, the diagonal of its absolute weak normalisation is an open immersion by \Cref{lmm:coh-et immersions} and we have that $Y^\mathrm{awn} \times_{X^\mathrm{awn}} Y^\mathrm{awn} \cong Y^\mathrm{awn} \coprod U_\Delta$ where $U_\Delta$ is the complement of $\Delta_{f^\mathrm{awn}}$ in $Y^\mathrm{awn} \times_{X^\mathrm{awn}} Y^\mathrm{awn}$. As $p \circ \Delta_{f^\mathrm{awn}} = id$, it thus suffices to show that the morphism $p': U_\Delta \to Y^\mathrm{awn} \times_{X^\mathrm{awn}} Y^\mathrm{awn} \to Y^\mathrm{awn}$ is étale. Note that $p'$ agrees with the absolute weak normalisation of the morphism $g: V \to Y \times_X Y \to Y$ where $V$ is the complement of the diagonal $\Delta_f$ in $Y\times_X Y$, in particular open; so $g$ is still of finite type, universally open and topologically unramified as all these properties are stable under base change and composition and satisfied by open immersions. Moreover, note that $U_\Delta$ is closed in $Y^\mathrm{awn} \times_{X^\mathrm{awn}} Y^\mathrm{awn}$, and so $U_\Delta$ and thus $V$ is again qcqs. But $n_g < n_f$, so we win by the induction hypothesis.
\end{proof}

All in all, we finally arrive at the desired statement.

\begin{theorem}\label{cor:coh et awn et}
    A separated morphism $f$ of finite presentation in $\mc C(\mathfrak G)$ is cohomologically étale if and only if $f^\mathrm{awn}$ is étale.
\end{theorem}
\begin{proof}
    If $f^\mathrm{awn}$ is étale, then $f$ is cohomologically étale by \Cref{prop:awn etale implies coh etale}. The converse follows now immediately from \Cref{thm:qf open unram awn etale} as cohomologically étale morphisms of finite presentation are universally open (\Cref{cor:et-is-univ-open}) and topologically unramified (\Cref{cor:coh-unramified}).
\end{proof}

\begin{example}
    We already mentioned in \Cref{xmpl:coh smooth example} that any universal homeomorphism is cohomologically étale, so as a sanity check note that universal homeomorphisms become isomorphisms after absolute weak normalisation (cf. \Cref{prop:awn colocalisation}).\par
    In fact, it follows from \Cref{cor:coh et awn et} that universal homeomorphisms are "building blocks" of cohomologically étale morphisms in the following sense: any finite cohomologically étale morphism $f:Y \to X$ of finite presentation is étale locally of the form $\coprod \tilde X_i \to X$ such that each $\tilde X_i \to X$ is a universal homeomorphism. Indeed, by \Cref{cor:coh et awn et} $f^\mathrm{awn}$ is finite étale and so by the étale-local structure of finite étale morphisms (cf. \cite[04HN]{Stacks}) for any $x \in |X| = |X^\mathrm{awn}|$ there is an étale neighbourhood $U \to X^\mathrm{awn}$ such that $Y^\mathrm{awn} \times_{X^\mathrm{awn}} U \cong \coprod_I \tilde U_i$ for some finite set $I$ and each $\tilde U_i \to \coprod_I \tilde U_i \to U$ is an isomorphism. Then this $U$ corresponds to an étale neighbourhood $V \to X$ of $x$ with $V^\mathrm{awn} \cong U$ as $X^\mathrm{awn} \to X$ is a universal homeomorphism. Then $(Y \times_X V)^\mathrm{awn} \cong Y^\mathrm{awn} \times_{X^\mathrm{awn}} U \cong \coprod_I U$ and thus $Y \times_X V \cong \coprod \tilde V_i$ such that each $V_i \to V$ becomes an isomorphism after absolute weak normalisation; but then each $V_i \to V$ is a universal homeomorphism, cf. \Cref{prop:awn colocalisation}.
\end{example}

Let us also note the following descent result for cohomologically étale morphisms as a corollary:

\begin{corollary}\label{cor:coh et change of base}
    Let $f: Y\to X$ be a morphism of qcqs schemes of finite presentation and let $g$ be a $v$-cover $X' \to X$. Then $g^* f$ is cohomologically étale if and only if $f$ is.
\end{corollary}
\begin{proof}
    As perfection commutes with fiber products, $(g^* f)^\mathrm{awn} = g^{\mathrm{awn},*} f^\mathrm{awn}$. Moreover, $g^\mathrm{awn}$ is still a $v$-cover by \Cref{prop:prop of awn}; indeed, universal homeomorphisms are $v$-covers and $v$-covers are stable under base change and composition (\cite[\texttt{0ETE, 0ETF}]{Stacks}). Thus, the statement follows from \Cref{lmm:arc-descent for etale morphisms} and \Cref{cor:coh et awn et}.
\end{proof}

\subsection{Cohomological smoothness for quasi-finite morphisms}

We have seen in the previous section that cohomologically étale morphisms are cohomologically smooth and quasi-finite (\Cref{cor:coh-et implies quasi-finite}). Now we want to prove that a quasi-finite cohomologically smooth morphism is already cohomologically étale. Note that this statement is true if the word "cohomologically" is removed. We need the following result:

\begin{lemma}\label{lmm:closed imm et}
    Let $X \in \mc C(\mathfrak G)$\footnote{Recall \Cref{not:basic setup}.} and consider a closed immersion $i: Z \to X$. Then $i$ is cohomologically étale if and only if $i^! \Lambda \cong \Lambda$.
\end{lemma}
\begin{proof}
    Clearly $i$ is cohomologically unramified. By \Cref{lmm:coh et coh sm unram} it thus remains to show the "if"-direction. For this, it suffices to prove that the natural map $i^! \Lambda \to \Lambda$ is an isomorphism on global sections by \Cref{lmm:coh et coh sm unram}. One checks that this is the map induced by restricting the counit $i_* i^! \Lambda \to \Lambda$ to $Z$. We claim that the induced map $H^0(Z, i^! \Lambda) \to H^0(Z, \Lambda)$ is always injective. Indeed, consider the dualised excision triangle
    $$\iHom(i_*i^* \Lambda, \Lambda) \to \iHom(\Lambda, \Lambda) \to \iHom(j_! \Lambda, \Lambda)$$
    where $j$ is the inclusion of the complement of $i$. The first map is dual to the unit $\Lambda \to i_* i^* \Lambda$ and thus after identifying $\iHom(i_* i^* \Lambda, \Lambda) \cong i_* \iHom(i^* \Lambda, i^! \Lambda) \cong \iHom(\Lambda, i_* i^! \Lambda) \cong i_* i^! \Lambda$ (using \Cref{PF-adj-isos}) given by the counit of the $i_*$-$i^!$-adjunction (by construction of $\mc D_\Lambda$). On the other hand, we have $\iHom(j_! \Lambda, \Lambda) \cong j_* \Lambda$ by the interior pullback formula (\Cref{PF-adj-isos}). Applying $i^*$ thus yields a triangle
    $$i^! \Lambda \to \Lambda \to i^* j_* \Lambda$$
    and the induced long exact sequence in cohomology gives $0 \to H^0(Z, i^! \Lambda) \to H^0(Z, \Lambda)$. Now if $i^! \Lambda \cong \Lambda$, then this is an injective map $\Lambda^{\pi_0(Z)} \cong H^0(Z, \Lambda) \to  H^0(Z, \Lambda) \cong \Lambda^{\pi_0(Z)}$ of $\Lambda^{\pi_0(Z)}$-modules, so it is given by multiplication with some $(x_\alpha)_{\alpha \in \pi_0(Z)}$ where $x_\alpha \in \Lambda$. Because the map is injective, multiplication with $x_\alpha$ must be an injective map $\Lambda \to \Lambda$ for all $\alpha$. But as $\Lambda$ is finite, this implies that each $x_\alpha$ is a unit and so $H^0(Z, i^! \Lambda) \to H^0(Z, \Lambda)$ is already a bijection.
\end{proof}

\begin{corollary}\label{prop:diagonal and unram}
    Let $f: Y \to X$ be a morphism in $E(\mathfrak G)$. Then $\Delta_f^! \Lambda \cong \Lambda$ if and only if $f$ is cohomologically unramified.
\end{corollary}
\begin{proof}
    This follows immediately from \Cref{lmm:closed imm et} as the morphisms in $E(\mathfrak G)$ are separated by definition.
\end{proof}

As a corollary, we can now deduce that in our setting we only need to identify the dualising complex of a cohomologically smooth morphism to check if it is cohomologically étale.

\begin{lemma}\label{lmm:coh etale vs coh smooth}
    Let $f:Y \to X$ be a morphism in $E(\mathfrak G)$. The following are equivalent:
    \begin{enumerate}[label=\roman*)]
    \item $f$ is cohomologically étale.
    \item $f$ is cohomologically smooth and $f^! \Lambda \cong \Lambda$.
    \item $f$ is cohomologically smooth and $f^! \Lambda$ is concentrated in degree $0$.
    \end{enumerate}
\end{lemma}
\begin{proof}
    $i)$ implies $ii)$ by \Cref{lmm:coh et coh sm unram}. $ii)$ is a special case of $iii)$. We can check cohomological étaleness étale locally by \Cref{lmm:et-is-loc-et}. Moreover, note that an invertible sheaf on $Y$ is étale locally isomorphic to $\Lambda$: Namely, by \Cref{rem:dualizable-perfect} this reduces to the statement that an invertible $\Lambda$-module is free which is true since $\Lambda$ is a finite ring. To finish the proof it thus suffices to show $ii) \Rightarrow i)$. But this is now immediate from \Cref{prop:diagonal and unram} (and \Cref{lmm:coh et coh sm unram}) because we have an isomorphism $\Delta_f^! \Lambda \cong (f^! \Lambda)^{-1}$ by \Cref{cor:dual to diagonal}. %
\end{proof}

Hence, we only need to prove that if $f$ is cohomologically smooth and quasi-finite, then $f^! \Lambda$ is concentrated in degree $0$. This is clear if the target is a point: Indeed, if $f: Y \to \Spec l$ is quasi-finite (thus finite, cf. \cite[\texttt{02NH}]{Stacks}) for a separably closed field $l$, then $Y^\mathrm{awn} \to (\Spec l)^\mathrm{awn} = \Spec \bar l$ is étale\footnote{Note that $Y^\mathrm{awn}\to (\Spec l)^\mathrm{awn}$ agrees with the absolute weak normalisation of the base change $Y_{\bar l} \to \Spec \bar l$ which is already étale after reduction, cf. \cite[\texttt{056V}]{Stacks}.} and so $f$ is cohomologically étale by \Cref{prop:awn etale implies coh etale}. But in fact, this observation already suffices to deduce the desired statement in general since formation of the dualising complex commutes with base change by definition of cohomologically smooth morphisms.
Together with the result of the previous section we arrive at the following characterisation of cohomologically étale morphisms:

\begin{theorem}\label{cor:coh sm quasifinite coh etale}
    Let $f:Y \to X$ be a separated morphism of finite presentation in $C(\mathfrak G)$. Then the following are equivalent:
    \begin{enumerate}[label=\roman*)]
        \item $f$ is cohomologically smooth and quasi-finite.
        \item $f$ is cohomologically étale.
        \item $f^{\mathrm{awn}}$ is étale.
    \end{enumerate}
\end{theorem}
\begin{proof}
    $ii) \Leftrightarrow iii)$ is \Cref{cor:coh et awn et} and $ii)$ implies $i)$ by \Cref{lmm:coh et coh sm unram} and \Cref{cor:coh-et implies quasi-finite}. To see that $i)$ implies $ii)$, we only need to show that $f^! \Lambda$ is concentrated in degree $0$ by \Cref{lmm:coh etale vs coh smooth}. To this end, let $x \in X$ and $\bar{x} \to X$ a corresponding geometric point. Then $f_x: Y_{\bar x} \to \bar x$ is quasi-finite and thus cohomologically étale, as observed above. Thus, as $f$ is cohomologically smooth, we have $f^! \Lambda|_{f^{-1}(x)} \cong f_x^! \Lambda \cong \Lambda$. As $x$ was arbitrary, this concludes the proof.
\end{proof}

In positive characteristic, the fact that quasi-finite morphisms to a point are cohomologically étale can also be seen as an instance of a more general observation: A separated, quasi-finite morphism $f:Y \to X$ with $X$ noetherian is cohomologically étale over a stratification of $X$.\par
This is a consequence of the following well-known property of morphisms of perfect schemes:

\begin{proposition}\label{prop:perf schemes form etale}
    If $f: Y \to X$ is a morphism of perfect schemes of characteristic $p$, then it is formally étale. In particular, it is étale if and only if it is of finite presentation.
\end{proposition}
\begin{proof}
    Formal étaleness may be checked locally, so assume $Y = \Spec B$ and $X = \Spec A$ for perfect rings $A$ and $B$ and look at a diagram
    \[\begin{tikzcd}
        A \rar["f"] \dar["h"] & B \dar["g"] \\
        R/I^2 \rar & R/I
    \end{tikzcd}\]
    for some ring $R$ and ideal $I \subseteq R$. We want to construct a lift $B \to R/I^2$ over $A$ and then prove its uniqueness. For the first part, note that for each $b \in B$ we can choose some lift $r_b$ of $g(b^{1/p})$ in $R/I^2$. In fact, we claim that $b \mapsto r_b^p$ is already well-defined and yields the desired map. Indeed, if $r_b'$ is another lift of $g(b^{1/p})$ then $(r_b)^p - (r_b')^p = (r_b - r_b')^p \in I^2$. The commutativity of the lower triangle for the lift is obvious by construction. The commutativity of the upper triangle follows from the well-definedness since $h(a^{1/p})$ is a lift for $g(f(a)^{1/p})$. To see uniqueness, simply note that any lift sends $b^{1/p}$ to some $r'$ lifting $g(b^{1/p})$ and thus sends $b$ to $(r')^p$.
\end{proof}
    
The following small commutative algebra lemma is then the key observation.

\begin{lemma}\label{lmm:perfection of affine finite}
    Let $A$ be a noetherian domain over $\mb F_p$ and $f: A \to B$ a finite morphism of $\mb F_p$-algebras. Then there is some $0 \not= a \in A$ such that the induced morphism $f': A_\mathrm{perf}[a^{-1}] \to B_\mathrm{perf}[a^{-1}]$ is of finite type.
\end{lemma}
    \begin{proof}
    First note that we can wlog.\ assume that $B$ is reduced as $B$ and $B_\mathrm{red}$ are homeomorphic and the morphism $A_\mathrm{perf} \to B_\mathrm{perf}$ only depends on $A_\mathrm{red}\to B_\mathrm{red}$.
    Pick a finite set of generators $s_1,...,s_N$ of $B$ as an $A$-module.\\
    Since $B$ is a finite module over the noetherian ring $A$, it is a noetherian module; in particular for every $s \in \set{s_1,...,s_N}$ there is some $m \in \mb N$ such that
    $$s^{p^m} + a_{m-1} s^{p^{m-1}} + ... + a_0 s = 0$$
    and since $B$ is reduced, there is some minimal $i \geq 0$ with $a_i \not= 0$. In $B_\text{perf}$ we thus have
    $$s + a_{m-1}^{p^{-m}} s^{p^{-1}} + ... + a_0^{p^{-m}} s^{p^{-m}} = 0$$
    and it follows that in $B_\mathrm{perf}[a_{i}^{-1}]$, $s^{p^{-i}}$ is a $A_\mathrm{perf}[a_{i}^{-1}]$-linear combination of $s^{p^{-n}}$ for $n < i$; which implies in particular that the subalgebra $A_\mathrm{perf}[a_{i}^{-1}][s^{p^{-i}}]$ of $B_\mathrm{perf}[a_{i}^{-1}]$ already contains all $s^{p^{-n}}$, $n \in \mb Z$.\\
    Now choose such an $i_j$ and $a_{i_j}$ for every $s_j$ and set $a := \prod_j a_{i_j} \not= 0$. Then the subalgebra
    $$A':= A_{\mathrm{perf}}[a^{-1}][s_j^{p^{-i_j}}, \ j \leq N]$$ of $B[a^{-1}]_{\mathrm{perf}}$
    contains all $s_j^{p^{-n}}$. But every element of $B[a^{-1}]_\mathrm{perf}$ is $f(s_1,...,s_N)^{p^{-i}}$ for some polynomial $f \in A[a^{-1}][X_1,...,X_N]$ and since the Frobenius is a ring homomorphism, $$f(s_1,...,s_N)^{p^{-i}} = f^{p^{-i}}(s_1^{p^{-i}},...,s_N^{p^{-i}})$$
    where $f^{p^{-i}} \in A[a^{-1}]_\mathrm{perf}[X_1,...,X_n]$ is $f$ with all coefficients replaced by their $p^i$-th root.
    So we have $A' = B[a^{-1}]_\mathrm{perf}$.
    \end{proof}

    To reduce from quasi-finite to finite, we use the following form of Zariski's main theorem:

    \begin{theorem}[Zariski's main theorem]\label{thm:Zariski}
        Let $f:X \to S$ be a morphism of schemes. Assume $f$ is quasi-finite and separated. Then there is a factorisation $f = \pi \circ j$ into a quasi-compact open immersion $j:X \to T$ and a finite morphism $\pi: T \to S$. If $f$ is moreover locally of finite presentation, then there is such a factorisation with $T$ finite and of finite presentation over $S$.
    \end{theorem}
    \begin{proof}
        See \cite[\texttt{05K0,0F2N}]{Stacks}.
    \end{proof}
    
    \begin{proposition}\label{prop:perfection of finite}
        Let $X$ be a noetherian integral scheme over $\mb F_p$ and $f:Y \to X$ a quasi-finite, separated morphism. Then there is some non-empty open $U\subseteq X$ such that the induced morphism $f_U^{\text{perf}}:Y_U^\text{perf} \to U^\text{perf}$ is étale.
    \end{proposition}
    \begin{proof}
        Note that by \Cref{prop:perf schemes form etale} it suffices to show that there is some $U \not= \emptyset$ such that $f_U^{\text{perf}}$ is of finite presentation.\par
        First, we use \Cref{thm:Zariski} to reduce to the case that $f$ is finite (note that $f$ is of finite presentation as it is of finite type and $X$ is noetherian). Indeed, an open immersion induces an open immersion on perfections by \Cref{prop:prop of awn}. Now wlog.\ we can assume $Y = \Spec B$, $X = \Spec A$ so that \Cref{lmm:perfection of affine finite} applies: we find some $U \subseteq X$ such that $Y_{U}^{\mathrm{perf}} \to U^{\mathrm{perf}}$ is of finite type. Then by \cite[\texttt{052B}]{Stacks} it is of finite presentation after shrinking $U$.
    \end{proof} 
    
    Now we can do noetherian induction on $X$. 
    
    \begin{corollary}\label{cor:perfection etale on strat}
        Let $X$ be a noetherian scheme over $\mb F_p$ and $f:Y \to X$ a quasi-finite, separated morphism. Then there is a (finite) decomposition $\coprod_i X_i$ of $X$ into locally closed subschemes such that $(Y_{X_i})^\mathrm{perf} \to X_i^\mathrm{perf}$ is étale.
    \end{corollary}
    \begin{proof}
    
        Let $\mc M$ be the set of closed subschemes $Z$ of $X$ such that the statement does not hold for the morphism $Y_Z \to Z$ (note that it is again quasi-finite and separated). If $\mc M$ is non-empty, then it contains some inclusionwise minimal element $Z$ as $X$ is noetherian. As we are aiming for a contradiction, we might as well assume $Z=X$, i.e. that the statement is true for all closed subschemes $Z \subsetneq X$, but not for $X$.\par
        Then $X$ must be irreducible: Indeed, if $X$ has more than one irreducible component, then there are decompositions for the finitely many irreducible components of $X$ by assumption, but using \cite[09Y4]{Stacks} this would yield a decomposition for $X$, contradicting $X \in \mc M$. Wlog.\ we can replace $X$ by its reduction; then $X$ is integral. Now by \Cref{prop:perfection of finite} we find some non-empty $U$ such that $Y_U^\mathrm{perf} \to U^\mathrm{perf}$ is étale. Then the statement holds for its complement $Z$; but this yields again a decomposition for $X$, contradicting $X \in \mc M$.\par
        Thus $\mc M = \emptyset$, proving the claim.
    \end{proof}

    \begin{corollary}\label{cor:coh et on stratum}
        Let $\Lambda$ be any torsion ring. Let $X$ be a noetherian scheme over $\mb F_p$ and $f:Y \to X$ a quasi-finite, separated morphism. Then there is a (finite) decomposition $\coprod_i X_i$ of $X$ into locally closed subschemes such that $Y_{X_i} \to X_i$ is $\mc D_\Lambda$-cohomologically étale.
    \end{corollary}
    \begin{proof}
        This is now immediate from \Cref{cor:perfection etale on strat} and \Cref{prop:awn etale implies coh etale}.
    \end{proof}
\newpage

\section[Cohomologically smooth morphisms for \texorpdfstring{$\mb F_p$}{F\_p}-sheaves in characteristic p]{Cohomologically smooth morphisms for \texorpdfstring{${\mb F_p}$}{F\_p}-sheaves in characteristic p}

For the rest of this paper we consider the slice geometric setup $\mathfrak G_k:= \mathfrak G_{/k}$ of qcqs schemes over a field $k$ of characteristic $p$ and the restriction $\mc D_p$ of the 6-functor formalism $\mc D_{\mb F_p}$ from \Cref{thm:et 6 functors} on it (see \Cref{rem:6-functor slices}). The goal of this section is to prove that $\mc D_{p}$-cohomologically smooth morphisms of finite presentation are quasi-finite, turning \Cref{cor:coh sm quasifinite coh etale} into a characterisation of cohomologically smooth morphisms in this setting.\par
 We claim that a finite type morphism $X \to \Spec k$ is cohomologically smooth if and only if it is cohomologically étale. As $\Spec \bar k \to \Spec k$ is certainly a $v$-cover (as it is fpqc, cf. \cite[\texttt{0ETC}]{Stacks}), we may assume that $k$ is algebraically closed by \Cref{cor:coh et change of base} which we do in the following.  \par
Recall from \Cref{thm:sm coh sm crit} that smooth morphisms in $\mc C(\mathfrak G_k)$ are $\mc D_\Lambda$-cohomologically smooth if and only if $\mb P^1_k \to \Spec k$ is $\mc D_\Lambda$-cohomologically smooth. We mentioned in \Cref{xmpl:coh smooth example} that this holds for $\Lambda = \mb F_\ell$ if $\ell \not= p$. Let us prove that it is not the case if $\ell = p$.\par
For this, it is helpful to identify the invertible objects in $\mc D_p(\mb P^1_k)$.

\begin{lemma}
    \label{inv-obj}
    Let $X$ be a scheme. Then $\mc L \in \mc D_p(X)$ is invertible if and only if there exists a decomposition $X = \coprod_{i \in I} X_i$ into open and closed subsets and $\mb F_p$-local systems $\mb L_i$ of rank 1 on $X_i$ such that $\mc L_{|X_i} \cong \mb L_i[n_i]$ for some $n_i \in \mb Z$.
\end{lemma}
\begin{proof}
    This is \cite[\texttt{0FPY}]{Stacks} (by \cite[\texttt{0B8Q}]{Stacks} invertible $\mb F_q$-modules are $\mb F_q$-local systems).
\end{proof}

\begin{lemma}\label{lmm:p1 simply connected}
    The categories of finite étale covers of $\Spec k$ and of $\mb P^1_k$ are equivalent. In particular, $\mb P^1_k$ is simply connected, i.e. $\pi_1(\mb P^1_k) = 1$ (where $\pi_1$ denotes the étale fundamental group).
\end{lemma}
\begin{proof}
    The following argument is due to Fargues-Fontaine (cf. \cite[Théorème 8.6.1]{Fargues}). In the following all functors are underived. Let $X \to \mb P_k^1$ be a finite, étale morphism. Then the finite étale algebra $f_* \mc O_X$ is a vector bundle on $\mb P_k^1$ and $X \cong \rSpec_{\mc O_{\mb P^1_k}}(f_* \mc O_X)$. Thus, as a module, $f_* \mc O_X$ is a direct sum $\bigoplus_{d \in I} \mc O(d)^{\oplus n_d}$ for some finite set $I \subseteq \mb Z$ and some $n_d \in \mb N$. Let $d_m := \max_{d \in I} d$. Assume $d > 0$. Then the multiplication on $f_*\mc O_X$ restricts to a morphism $\mc O(d_m) \tensor \mc O(d_m) \to f_* \mc O_X$; but
    $$\Hom(\mc O(d_m) \tensor \mc O(d_m) , f_* \mc O_X) \cong
        \Hom(\mc O(0) , f_* \mc O_X \tensor \mc O(-2d_m)) = H^0(\mb P_k^1, f_* \mc O_X \tensor \mc O(-2d_m)) = 0$$
    by definition of $d_m$. In particular, all sections of $\mc O(d_m)$ are nilpotent; but this is absurd as $f_* \mc O_X$ is an étale algebra, in particular reduced. Thus, $d_m \leq 0$.\par
    But note that the trace pairing induces an isomorphism $f_* \mc O_X \cong \Hom(f_* \mc O_X, \mc O_{\mb P_k^1})$ because $f$ is étale; so we must have $d_m = 0$, i.e. $f_* \mc O_X$ is a trivial vector bundle.\par
    Hence, sending $f \mapsto f_* \mc O_X$ defines an equivalence between the category of finite, etale covers of $\mb P_k^1$ and the category of trivial vector bundles on $\mb P_k^1$. On the other hand, $\epsilon \mapsto H^0(\mb P_k^1, \epsilon)$ defines an equivalence between the category of trivial vector bundles on $\mb P_k^1$ and the category of finite dimensional vector spaces over $k$, since $H^0(\mb P_k^1, \mc O_{\mb P_k^1}) = k$. But the latter is equivalent to the category of finite étale covers of $\Spec k$ since $k$ is separably closed.
\end{proof}

\begin{proposition} \label{no-trace-cyc-for-p1} \label{p1n-not-smooth}
    Let $n \geq 0$ and $f: (\mb P^1_k)^n \to \Spec k$. Then
    \begin{enumerate}[label=\roman*)]
        \item $(\mb P^1_k)^n$ is simply connected.
        \item $f_! \mb F_p[0] \cong \mb F_p[0]$.
        \item $f$ is cohomologically smooth if and only if $n=0$.
    \end{enumerate}
\end{proposition}
\begin{proof}
    $\mb P^1_k$ is simply connected by \Cref{lmm:p1 simply connected}. This implies $i)$ by an application of the homotopy short exact sequence \cite[\texttt{0C0J}]{Stacks}: By induction, assume we know that $(\mb P^1_k)^n$ is simply connected, i.e. $\pi_1((\mb P^1_k)^n)=1$. Now the morphism $\pi: (\mb P^1_k)^{n+1} \to (\mb P^1_k)^n$ is flat and proper with connected and reduced geometric fibers and $(\mb P^1_k)^n$ is connected. Choose any closed point $s \in (\mb P^1_k)^n$, then we have the homotopy short exact sequence
    $$\underbrace{\pi_1(\mb P^1_k)}_{\cong 1} \to \pi_1((\mb P^1_k)^{n+1}) \to \pi_1((\mb P^1_k)^n) \to 1$$
    and so by the induction assumption $\pi_1((\mb P^1_k)^{n+1}) = 1$.\par
    For $ii)$ let us first compute the $\mb F_p$-cohomology of the projective line. For this, we can use the Artin-Schreier sequence (cf. \cite[\texttt{0A3J}]{Stacks})
    $$0 \to \mb F_p \to \mc O_{\mb P^1_k} \overset{F-1} \to \mc O_{\mb P^1_k} \to 0$$
    of étale sheaves, where $F-1$ is the map $x \mapsto x^p - x$. Since $H^0(\mb P^1_k,\mc O_{\mb P^1_k})\cong k$ and $H^i(\mb P^1_k, \mc O_{\mb P^1_k}) = 0$ (using \cite[\texttt{03P2}]{Stacks}) for all $i>0$ and moreover $F-1: k \to k$ is surjective since $k$ is algebraically closed, it follows that $R \Gamma(\mb P^1_k, \mb F_p) = \mb F_p[0]$. Then we can apply the Künneth formula (\Cref{cor:Kuenneth}) inductively to the diagram
    \[\begin{tikzcd}
            (\mb P_k^1)^{n+1} \rar{\pi} \dar{p} \drar{f} & (\mb P_k^1)^{n} \dar{f'} \\
            \mb P^1_k \rar{g} & \Spec k
    \end{tikzcd}\]
    to get $ii)$.\par
    Now assume that $f$ is cohomologically smooth. By \Cref{inv-obj} $f^! \mb F_p = A[i]$ for some invertible module $A$ on $((\mb P^1_k)^n)_\et$ and some $i \in \mb Z$. Then $A$ is represented by some finite étale $U \to (\mb P^1_k)^n$ but then by $i)$ it is constant and so we must have $A \cong \mb F_p$. Then by $ii)$, we have
    $$\mb F_p = \iHom(f_!\mb F_p,\mb F_p) = f_* \iHom(\mb F_p, f^! \mb F_p) = f_* \mb F_p [i] = \mb F_p [i]$$
    where we used the Verdier duality formula (\Cref{PF-adj-isos}), so $i=0$ and we conclude by \Cref{lmm:coh etale vs coh smooth} that $f$ is cohomologically étale. This is clearly a contradiction to \Cref{cor:coh-et implies quasi-finite} (quasi-finiteness of cohomologically étale morphisms) if $n > 0$. However for the convenience of the reader we give a self-contained argument: For $n > 0$, consider the open immersion $j: \mb A^1_k \times (\mb P^1_k)^{n-1} \to (\mb P^1_k)^n$ and let $i$ be the inclusion of the complement $(\mb P^1_k)^{n-1}$. Then the distinguished triangle
    $$f_* j_! \mb F_p \to f_* \mb F_p \to f_* i_* \mb F_p$$
    yields $f_* j_! \mb F_p = 0$ since $f_* \mb F_p \cong \mb F_p \cong f_* i_* \mb F_p$. Thus, if $f$ were cohomologically étale, then
    $$0 = \Hom(f_* j_! \mb F_p, \mb F_p) \cong \Hom(\mb F_p, j^* f^* \mb F_p) \cong H^0(\mb A^1_k \times (\mb P^1_k)^{n-1}, \mb F_p)$$
    which is absurd.
   
\end{proof}

\begin{corollary}\label{A^n-not-sm}
    $\mb A^n_k \to \Spec k$ is cohomologically smooth if and only if $n=0$.
\end{corollary}
\begin{proof}
    This follows from \Cref{p1n-not-smooth} as $(\mb P^1_k)^n$ is locally isomorphic to $\mb A^n_k$ and cohomological smoothness is a cohomologically smooth-local property (on the source) by \Cref{lmm:sm-is-loc-et}.
\end{proof}

We can use this to understand the smooth cohomologically smooth morphisms $f: X \to \Spec k$ because locally they factor over an $\mb A^n_k$ via an étale morphism. 

\begin{corollary}\label{cor:smooth + coh-smooth over k}
    If $f: X \to \Spec k$ is smooth and cohomologically smooth then $f$ is etale, i.e. $X \cong \coprod_{i \in I} \Spec k$ for some finite set $I$.
\end{corollary}
\begin{proof}
    The statement is local on the source so assume $f$ factors as
    \[\begin{tikzcd}
            X \arrow["g"]{r} \arrow[swap,"f"]{dr} & \mb A_k^n \arrow["\pi"]{d}\\
            & \Spec k
        \end{tikzcd} \]
    with $g$ etale. Then in particular $\im g$ is an open subscheme and $X \to \im g$ being cohomologically smooth and surjective implies by \Cref{lmm:sm-is-loc-et} that $\pi|_{\im g}$ is cohomologically smooth. But by another application of \Cref{lmm:sm-is-loc-et} this implies that $\pi$ is cohomologically smooth because we can cover $\mb A_k^n$ by translating $\im g$ around and thus we have $n=0$ by \Cref{A^n-not-sm}, i.e. $f=g$.
\end{proof}

\begin{corollary}\label{cor:smooth over k}
    A morphism $X \to \Spec k$ of finite type is cohomologically smooth if and only if $X_{\text{red}} \to \Spec k$ is etale, i.e. $X_{\text{red}} \cong \coprod_{i \in I} \Spec k$ for some finite set $I$. In particular, it is cohomologically smooth if and only if it is cohomologically étale.
\end{corollary}
\begin{proof}
    As a reduced, finite type scheme over an algebraically closed field, $X_{\text{red}}$ is generically smooth (cf. \cite[\texttt{056V}]{Stacks}) and thus the first statement follows from \Cref{cor:smooth + coh-smooth over k}. The second statement is then clear from \Cref{prop:awn etale implies coh etale}.
\end{proof}
In the following, $k$ can be an arbitrary field.
\begin{corollary}\label{cor:smooth over k gen}
    A morphism $X \to \Spec k$ of finite type is cohomologically smooth if and only if it is cohomologically étale.
\end{corollary}
\begin{proof}
    As we already noted at the beginning of this section, we may assume that $k$ is algebraically closed by \Cref{cor:coh et change of base}. Then this is \Cref{cor:smooth over k}.
\end{proof}

\begin{proposition}\label{cor:coh-smooth implies quasi-finite}
    If $f: Y \to X$ is of finite type and cohomologically smooth, then it is quasi-finite.
\end{proposition}
\begin{proof}
By \Cref{cor:smooth over k gen} and \Cref{cor:coh-et implies quasi-finite} the fibers of $f$ are finite discrete sets. But a finite type morphism with discrete fibers is quasi-finite (cf. \cite[\texttt{06RT}]{Stacks}).
\end{proof}

Putting all of our results together we can now finally answer our main question.

\begin{theorem}\label{thm:char}
    Let $k$ be an arbitrary field over $\mb F_p$ and $f:Y \to X$ a separated morphism of finite presentation in $\mc C(\mathfrak G_k)$. Then the following are equivalent:
    \begin{enumerate}[label=\roman*)]
        \item $f$ is $\mc D_p$-cohomologically smooth.
        \item $f$ is $\mc D_p$-cohomologically étale.
        \item $f^{\mathrm{perf}}$ is étale.
    \end{enumerate}
\end{theorem}

\begin{proof}
    \Cref{cor:coh et awn et} shows $ii) \Leftrightarrow iii)$ (recall from \Cref{prop:perf and awn} that perfection and absolute weak normalisation agree in positive characteristic). $ii) \Rightarrow i)$ is clear (\Cref{lmm:coh et coh sm unram}). Finally, $i) \Rightarrow ii)$ follows from \Cref{cor:coh-smooth implies quasi-finite} and \Cref{cor:coh sm quasifinite coh etale}.
\end{proof}

\newpage

\section{Smooth objects for \texorpdfstring{$\mb F_p$}{F\_p}-sheaves in characteristic \texorpdfstring{$p$}{p}}
We stay in the setting of the previous section and consider the 6-functor formalism $\mc D_p$ of étale $\mb F_p$-sheaves on the slice geometric setup $\mathfrak G_k:=\mathfrak G_{/k}$, where $k$ is a field of characteristic $p$. After characterising the cohomologically smooth morphisms for $\mb F_p$-sheaves in characteristic $p$, we now want to answer the more general question when an object is $\mc D_p$-smooth for a given separated morphism $f$ of finite presentation, see \Cref{section:sm obj}. We start with the case of morphisms $f: X \to \Spec k$.

\subsection{The absolute case}

In the following we assume that $k$ is algebraically closed, we can get rid of this assumption later (\Cref{prop:sm object char absolute gen}). As a first observation, we claim that $\mb F_p$ is not $\mb A^1_k \to \Spec k$ smooth, i.e. that non-invertibility of the Verdier dual was not the only obstruction to the cohomological smoothness of $\mb A^1_k \to \Spec k$ in \Cref{A^n-not-sm}. The proof requires a bit more work than the argument we gave in \Cref{p1n-not-smooth}. An important ingredient is the fact that $f$-smooth objects are \textit{perfect-constructible} if $f$ is a morphism of schemes of finite type over $k$.

\begin{definition}[{\cite[Definition 6.3.1]{Scholze3}}]
    Let $\Lambda$ be any torsion ring. A complex $K \in \mc D_\Lambda(X)$ is called \textit{perfect-constructible} if there exists a finite stratification $\set{X_i \to X}$ by constructible locally closed $X_i \subseteq X$ such that $K|_{X_i}$ is locally constant with perfect values on $X_{i,\et}$.
\end{definition}

\begin{lemma}[{\cite[Lemma 6.4.5]{Scholze3}}]\label{lmm:1 is compact}
    Let $\Lambda$ be any torsion ring and assume that $X$ is of finite type over a separably closed field $k$. Then for each étale map $j:U \to X$
    \begin{enumerate}[label=\roman*)]
        \item $j_! \Lambda$ is compact.
        \item $j_*: \mc D_\Lambda(U) \to \mc D_\Lambda(X)$ commutes with all direct sums.
    \end{enumerate}
\end{lemma}
\begin{proof}
    For $i)$, note that $j_!$ preserves compact objects (as it is left adjoint to the colimit preserving functor $j^*$) so wlog.\ assume $j=id_X$. Then we need to show that $\Hom(\Lambda,\blank)$ commutes with arbitrary direct sums, and in particular it suffices if $R\Gamma(X,\blank)$ does. By the assumptions on $X$, $\mc D_\Lambda(X)$ is left-complete (cf. \cite[Proposition 3.3.7]{Scholze3}) and the cohomological dimension of $X$ is bounded (\cite[\texttt{0F10}]{Stacks}). Thus for any $K \in \mc D_\Lambda(X)$ it follows that $H^i(X,K) \cong H^i(X, \tau^{\geq -n} K)$ for $n > d-i$. This reduces us to the case that $K$ is bounded from below which is \cite[\texttt{0GIS}]{Stacks}. 
    Now the observation that $R\Gamma(W,\blank)$ commutes with all direct sums for $W$ finite type over $k$ also implies $ii)$ as for any étale $V \to X$ and $K \in \mc D_\Lambda(U)$ $R\Gamma(V, j_* K|_V) = R\Gamma(V \times_X U, K|_{V \times_X U})$.
\end{proof}

\begin{corollary}[{\cite[Lemma 6.4.6]{Scholze3}}]\label{cor:et 6 sm obj compact}
    In the situation of \Cref{lmm:1 is compact}, let $i: Z\to X$ be a closed immersion. Then $i_* K$ is compact for every perfect complex $K$. In particular, if $f$ is a morphism in $E(\mathfrak G_k)$ of schemes of finite type over $k$, then $f$-smooth objects (w.r.t. $\mc D_\Lambda$) are compact.
\end{corollary}
\begin{proof}
    For the first statement, we first reduce to the case that $K$ is constant with perfect value by étale localisation (we proved in \Cref{lmm:1 is compact} that $\Lambda \in \mc D_\Lambda(X)$ is compact, so it suffices to show that $\iHom(i_* K,\blank)$ commutes with direct sums; then use that $j^* i_* K \cong i'_* j'^* K$ by proper base change if $j$ is an étale map and $i',j'$ are the respective pullbacks). Then the statement follows from $i)$ in \Cref{lmm:1 is compact} by considering the distinguished triangle $j_! K \to K \to i_* K$. The second statement follows from the first by \Cref{prop:sm obj compact}, choosing $f=\Delta$ and $K=\Lambda$.
\end{proof}

\begin{proposition}[{\cite[Proposition 6.4.8]{Scholze3}}]\label{prop:sm obj perfect-constructible}
    In the situation of \Cref{lmm:1 is compact}, the compact objects in $\mc D_\Lambda(X)$ are precisely the perfect-constructible objects. In particular, if $f$ is a morphism in $E(\mathfrak G_k)$ of schemes of finite type over $k$, then $f$-smooth objects (w.r.t. $\mc D_\Lambda$) are perfect-constructible.
\end{proposition}
\begin{proof}
    By \Cref{cor:et 6 sm obj compact} the second statement follows from the first. For the first assertion, let $\mc D_c(X,\Lambda)$ denote the stable subcategory of compact objects in $\mc D_\Lambda(X)$ and $\mc D_\mathrm{cons}(X,\Lambda)$ the stable subcategory of perfect-constructible objects. One checks formally that $\mc D_c(X,\Lambda)$ is the smallest idempotent complete subcategory that contains the objects $j_! \Lambda$ for $j: U \to X$ étale, which implies $\mc D_c(X,\Lambda) \subseteq \mc D_\mathrm{cons}(X,\Lambda)$. For the converse, we use that $\mc D_\mathrm{cons}(X,\Lambda)$ is the smallest idempotent complete subcategory that contains all $i_! K$ for $i:Z \to X$ the inclusion of a locally closed subset and $K$ a perfect complex (cf. \cite[Lemma 6.3.9]{Scholze3}). But all of these objects are compact, using \Cref{cor:et 6 sm obj compact} and that lower-$!$ preserves compact objects for open immersions. So $\mc D_c(X,\Lambda) = \mc D_\mathrm{cons}(X,\Lambda)$.
\end{proof}

Now we can return to our objective of showing that $\mb F_p$ is not $\mb A_k^1 \to \Spec k$-smooth. The following observation is crucial:

\begin{lemma}\label{lmm:sm concentrated}
Let $f: \mb P^1_k \to \Spec k$. Then $f^! \mb F_p = A[1]$ for some $\mb F_p$-sheaf $A$ on $\mb P^1_k$.
\end{lemma}
\begin{proof}
    We need to check that $f^! \mb F_p[-1] \in \mc D_p^\heartsuit(\mb P^1_k)$ and we may check this on stalks. By the Verdier duality formula (\Cref{PF-adj-isos}), we have
    \begin{align*}
        R\Gamma(U, f^! \mb F_p[-1]) &= f_* j_{U,*} \iHom(\mb F_p, j_U^* f^! \mb F_p[-1]) = \iHom(f_! j_{U,!} \mb F_p, \mb F_p[-1]) = \iHom(R\Gamma_c(U,\mb F_p)[1],\mb F_p)
    \end{align*}
     for any étale $j_U: U \to \mb P^1_k$. Thus it suffices to show the following: Assume $U$ is a normal curve and not proper over $\Spec k$ (these $U$ are certainly cofinal in the étale neighbourhoods of any point\footnote{We can of course assume that $U$ is connected; but then it is already integral by \cite[\texttt{0357}]{Stacks} as it is a normal scheme (since $\mb P^1_k$ is, cf. \cite[Théorème 9.5]{SGA1}) and thus a normal curve. By replacing it by some open subscheme we can ensure that $U$ is not proper.}), then $R\Gamma_c(U,\mb F_p)$ is concentrated in degree $1$.\par
    To prove this, let us first compute $R\Gamma(\bar{U},\mb F_p)$ where $\bar{U}$ is the natural compactification of $U$, cf. \cite[\texttt{0BXW}]{Stacks}. Recall the Artin-Schreier sequence
    $$0 \to \mb F_p \to \mc O_{\bar{U}} \overset{F-1} \to \mc O_{\bar{U}} \to 0$$
    Taking cohomology, this yields $H^0(\bar{U}, \mb F_p) \cong \mb F_p$ since $H^0(\bar U,\mc O_{\bar U}) \cong k$ is algebraically closed. Then we use the following fact:
    \begin{claim}[{\cite[\texttt{0A3L}]{Stacks}}]
        Let $k$ be an algebraically closed field of characteristic $p > 0$.
        Let $V$ be a finite dimensional $k$-vector space. Let $F : V \to V$
        be a Frobenius linear map, i.e., an additive map such that
        $F(\lambda v) = \lambda^p F(v)$ for all $\lambda \in k$ and $v \in V$.
        Then $F - 1 : V \to V$ is surjective with kernel a finite dimensional
        $\mb F_p$-vector space of dimension $\leq \dim_k(V)$.
    \end{claim}
    By finiteness of coherent cohomology on proper schemes over $k$ (\cite[\texttt{02O6}]{Stacks}), this shows that $H^1(\bar{U},\mb F_p)$ is some finite-dimensional $\mb F_p$-vector space and $H^i(\bar{U}, \mb F_p)=0$ for $i \geq 2$.\par
    Note that $U$ is an open subscheme of its compactification $\bar{U}$, so now we can use the excision sequence: We have a distinguished triangle
    $$R\Gamma_c(U, \mb F_p) \to R\Gamma(\bar{U},\mb F_p) \to R\Gamma(Z,\mb F_p)$$
    where $Z$ is the complement of $U$ in $\bar{U}$ and the second map is induced by restriction. The long exact sequence in cohomology thus yields $H_c^0(U,\mb F_p) = \ker(H^0(\bar{U}, \mb F_p) \to H^0(Z, \mb F_p)) = \ker(\mb F_p \to H^0(Z, \mb F_p)) = 0$ as $Z$ is non-empty by assumption, and $H_c^i(U, \mb F_p) = 0$ for all $i \geq 2$, proving the statement. Let us remark that $H^1_c(U,\mb F_p)$ has finite dimension since $H^0(Z,\mb F_p)$ does (the complement of $U$ in its natural compactification consists of finitely many points) and $H^1(\bar U, \mb F_p)$ does as observed above. This will be used in the proof of \Cref{prop:1 not P1 sm} below.
\end{proof}

With this lemma at hand we get the desired negative result.

\begin{proposition}\label{prop:1 not P1 sm}
    Let $f: \mb A^1_k \to \Spec k$. Then $f^! \mb F_p$ is not perfect-constructible, in particular $\mb F_p$ is not $f$-smooth.
\end{proposition}
\begin{proof}
The second assertion follows from the first because if $\mb F_p$ were $f$-smooth, then $f^! \mb F_p = \mb D_f(\mb F_p)$ would be as well (cf. \Cref{rem:smooth symmetry}) and $f$-smooth objects are perfect-constructible by \Cref{prop:sm obj perfect-constructible}.\par
So now assume that $f^! \mb F_p$ is perfect-constructible. Then in particular we find some open $U \subseteq \mb A^1_k$ such that $f^! \mb F_p |_U$ is constant with perfect value. After restricting further to some étale $U \to \mb A^1_k$ we can assume that $f^! \mb F_p |_U \cong \mb F_p^i[1]$ for some $i \geq 0$ by \Cref{lmm:sm concentrated} and by the arguments in the proof of \Cref{lmm:sm concentrated} we can choose $U$ as a curve such that $R\Gamma_c(U, \mb F_p)$ is concentrated in degree $1$ and has finite dimension. Let $f_U: U \to \Spec k$. Then $f^! \mb F_p |_U \cong f_U^! \mb F_p$ and thus by the adjunction
\begin{equation}\label{eq:p}
    \begin{aligned}
    H^0(U, \mb F_p)^{\oplus i} \cong \Hom(\mb F_p, \mb F_p^i) \cong \Hom(\mb F_p, f_U^! \mb F_p[-1]) &\cong  \Hom(f_{U,!} \mb F_p, \mb F_p[-1]) \\
    &\cong \Hom(H^1_c(U, \mb F_p), \mb F_p)
    \end{aligned}
\end{equation}
 But then we must have $i=1$: Let $x \in U$ be a closed point and $V:= U \backslash \set{x}$. Then the excision sequence $R\Gamma_c(V,\mb F_p) \to R\Gamma_c(U, \mb F_p) \to R\Gamma(\set{x}, \mb F_p) = \mb F_p[0]$ yields $H^1_c(V, \mb F_p) \cong H^1_c(U, \mb F_p) \oplus \mb F_p$. But (\ref{eq:p}) also holds for $V$ instead of $U$ and replacing $U$ by $V$ increases the dimension of the right hand side (which is finite) by $1$; so $i \cdot \dim H^0(U,\mb F_p) + 1 = i \cdot \dim H^0(V, \mb F_p)$, which forces $i=1$. But then $\dim H^0(V, \mb F_p) = 2$ as $U$ is connected. This means that the open subscheme $V$ of $U$ has two connected components which is absurd as $U$ is irreducible. So $f^! \mb F_p$ cannot be perfect-constructible.
\end{proof}

\begin{remark}
    \begin{enumerate}[label=\roman*)]
        \item As the proof of \Cref{prop:1 not P1 sm} is independent of the previous section, together with \Cref{cor:1 not An sm} below this yields another proof of \Cref{A^n-not-sm}.
        \item We had difficulties to find an explicit description of $f^! \mb F_p$. Note that techniques for the computation of the dualising complex as in \cite{Zav} are not available here as they only work under the assumption that smooth morphisms are cohomologically smooth. \par
        Using the adjunction $f_! \dashv f^!$ one can try to compute the stalks of $f^!\mb F_p$ directly; here, one faces two problems: the first is the computation of $H^1_c(U, \mb F_p)$ for a curve $U$ étale over $\mb A^1_k$; this reduces to the problem of understanding the kernel of the map $H^1(C, \mc O_C) \to H^1(C,\mc O_C)$ induced by the map $F - 1$ on $\mc O_C$ for a smooth projective curve $C$, where $F$ is the Frobenius. The second problem is that one needs to understand the restriction maps $R\Gamma(f^! \mb F_p, U) \to R\Gamma(f^! \mb F_p, V)$ for $j: V \to U$ étale and $U \to \mb A^1_k$ étale. These maps are dual to the maps $R\Gamma_c(V,\mb F_p) \to R\Gamma_c(U, \mb F_p)$ induced by the counit $j_! j^* \mb F_p \to \mb F_p$.
    \end{enumerate}
\end{remark}

\begin{corollary}\label{cor:1 not An sm}
    $\mb F_p$ is $\mb A^n_k \to \Spec k$-smooth if and only if $n=0$.
\end{corollary}
\begin{proof}
    This follows from \Cref{prop:sm retracts}: $\mb A^1_k$ is a retract of $\mb A^n_k$ over $\Spec k$ for all $n\geq 1$. If $i$ and $r$ are the corresponding maps, then of course we can just choose the identities $i^* \mb F_p \to \mb F_p$ and $r^* \mb F_p \to \mb F_p$ and so $\mb F_p$ is $\mb A^1_k \to \Spec k$-smooth if $\mb F_p$ is $\mb A^n_k \to \Spec k$ smooth. Thus, we must have $n=0$ by \Cref{prop:1 not P1 sm}.
\end{proof}

\begin{corollary}\label{cor:sm obj quasi-finite}
    Let $f: Y \to X \in E(\mathfrak G_k)$ be of finite type. If $\mb F_p$ is $f$-smooth, then $f$ is quasi-finite.
\end{corollary}
\begin{proof}
    Using \Cref{cor:1 not An sm}, the strategy from the previous sections applies again: We need to see that the fibers of $f$ are finite discrete sets, so assume $X=\Spec k$. Generic smoothness reduces us to the case that $f$ is smooth. For smooth $f: Y \to \Spec k$ we can reduce to the case that $f$ factors via an étale morphism over some $\mb A^n_k$, using that $f$-smoothness may be checked étale locally on the source by \Cref{thm:sm-obj-is-loc}. Using \Cref{thm:sm-obj-is-loc} and \Cref{cor:1 not An sm}, translating the image of $f$ around in $\mb A^n_k$ shows that $\mb F_p$ is $\mb A^n_k \to \Spec k$-smooth. But then $n=0$ by \Cref{cor:1 not An sm} and thus $f$ is étale. Compare e.g. the arguments following \Cref{p1n-not-smooth}.
\end{proof}

\begin{corollary}\label{cor:dualizable-sm imp quasi-finite}
    Let $f: Y \to X \in E(\mathfrak G_k)$ be of finite type. If $Y$ is connected and $A$ is dualisable and $f$-smooth, then either $A=0$ or $f$ is quasi-finite.
\end{corollary}
\begin{proof}
    We claim that if $A\not=0$ is $f$-smooth, then $\mb F_p$ is $f$-smooth, so that we can conclude by \Cref{cor:sm retracts}. As $A$ is dualisable, its support is an open and closed subset of $Y$ (cf. $ii)$ in \Cref{rem:dualizable-perfect}) and as $A \not=0$, it is supported everywhere. As we can check $f$-smoothness of $\mb F_p$ étale locally by \Cref{thm:sm-obj-is-loc}, using $i)$ in \Cref{rem:dualizable-perfect} and that $Y$ is the support of $A$, we can reduce to the case that $A \not= 0$ just arises from a complex of vector spaces and so we may as well assume that all differentials are $0$. Thus, $A$ is a finite direct sum of complexes $\mb F_p[i]$ for some $i \in \mb Z$. But $f$-smooth objects are stable under retracts by \Cref{cor:sm retracts} and so if $A \not= 0$, some $\mb F_p[i]$ is $f$-smooth; but then so is $\mb F_p$ by \Cref{cor:dualizable+smooth} as $\mb F_p[i]$ is invertible.
\end{proof}

To finish the characterisation in the absolute case, we need the following result\footnote{We thank Tim Kuppel for pointing this out to us.}:

\begin{proposition}\label{prop:sm with closed support}
    Consider a diagram
    \[\begin{tikzcd}
        Z \rar["i"] \drar[swap,"h"] & Y \dar["f"] \\
        & X
    \end{tikzcd}\]
    where $i$ is a closed immersion. Let $A \in \mc D_p(Z)$ such that $i_! A$ is $f$-smooth. Then $A$ is $h$-smooth.
\end{proposition}
\begin{proof}
    By \Cref{lmm:sm obj crit} we need to check that the natural map
    $$\Delta_h^!(p_{Z,2}^* \mb D_h(A) \tensor p_{Z,1}^* A) =:\mb D_h(A) \tensor^! A \to \iHom(A,A)$$
    is an isomorphism.\par
    We claim that $p_{Y,2}^* i_! \mb D_h(A) \tensor p_{Y,1}^* i_! A \cong (i \times i)_* (p_{Z,2}^* \mb D_h(A) \tensor p_{Z,1}^* A)$ via the unit $id \to (i \times i)_*(i \times i)^*$. Indeed, $(i \times i)^* (p_{Y,2}^* i_! \mb D_h(A) \tensor p_{Y,1}^* i_! A) \cong p_{Z,2}^* \mb D_h(A) \tensor p_{Z,1}^* A$ by the base change formula as the diagram
    \[\begin{tikzcd}
        {Z \times_X Z} & {Z \times_X Z} \\
        Z & Y
        \arrow["{p_{Z,i}}"', from=1-1, to=2-1]
        \arrow["{p_{Y,i} \circ (i \times i)}", from=1-2, to=2-2]
        \arrow["i"', from=2-1, to=2-2]
        \arrow["id", from=1-1, to=1-2]
    \end{tikzcd}\]
    is cartesian. Replacing $(i \times i)$ by $(j \times i)$ resp. $(i \times j)$ where $j$ is the inclusion of the complement of $Z$, base change and the similar diagrams show that $ p_{Y,2}^* i_! \mb D_h(A) \tensor p_{Y,1}^* i_! A$ is $0$ on the complement of $Z \times_X Z$ in $Y \times_X Y$, and we conclude.\par
    Note that $\Delta_{f}^! (i \times i)_* \cong i_* \Delta_{h}^!$ by adjoint base change (\Cref{PF-adj-isos}). Thus, by the observation above,
    $$\mb D_h(A) \tensor^! A \cong i^* i_* (\mb D_h(A) \tensor^! A) \cong i^* (i_! \mb D_h(A) \tensor^! i_! A)$$
    and we have an isomorphism $i_! \mb D_h(A) \cong \mb D_f(i_! A)$ from \Cref{PF-adj-isos} since $i_! = i_*$. But using the isomorphism $$ \mb D_f(i_! A) \tensor^! i_! A \to \iHom(i_! A,i_! A)$$ coming from \Cref{lmm:sm obj crit} by assumption, this agrees with
    $$i^* \iHom(i_!A,i_! A)\cong i^* i_* \iHom(i^* i_! A, A) \cong i^* i_* \iHom(A,A) \cong \iHom(A,A)$$
    Now one checks that this composition of isomorphisms yields the natural map from \Cref{lmm:sm obj crit}.
\end{proof}

\begin{proposition}\label{prop:sm object char absolute}
    Let $f: X \to \Spec k$ be a separated morphism of finite type. Then the $f$-smooth objects are precisely the scyscraper sheaves $i_! B$ where $i: Z \to X$ is the inclusion of some finite set $Z = \coprod_i \Spec k$ of closed points and $B \in \mc D(Z)$ is dualisable.
\end{proposition}
\begin{proof}
It follows from \Cref{lmm:sm obj composition} that all $i_! B$ of this form are $f$-smooth. Indeed, $Z \to \Spec k$ is cohomologically étale and so all dualisable objects are $Z \to \Spec k$-smooth by \Cref{cor:dualizable+smooth}.\par
Now let $A$ be $f$-smooth. We do induction on the Krull dimension $n$ of $X$. If $n=0$, $X$ is a finite set of points, in particular $X \to \Spec k$ is cohomologically étale (e.g. by \Cref{cor:coh et awn et}) and smooth objects are precisely the dualisable ones by \Cref{rem:unram}. Now assume $n>0$. First, we prove the case $X= \mb A^n_k$. As $A$ is smooth, it is perfect-constructible by \Cref{prop:sm obj perfect-constructible}. Let $\coprod_j T_j \to X$ be the corresponding stratification. Taking the union of the closures of all $T_j$ which do not contain the generic point in $X$ yields a closed set $T \subseteq X$ whose complement $U$ is a dense open in $X$ (as it contains the generic point). Moreover, if $T_i$ is the unique stratum which contains the generic point, we must have $U \subseteq T_i$; thus, $A|_U$ is dualisable (cf. \Cref{rem:dualizable-perfect}) and $f|_U$-smooth. But $U$ is connected, and so $A|_U = 0$ by \Cref{cor:dualizable-sm imp quasi-finite} since $n>0$. Let $z$ be the inclusion $T \subseteq X$. Since $A|_U = 0$, we have $A \cong z_! z^* A$. By \Cref{prop:sm with closed support}, $z^*A$ is thus $f|_T$-smooth. But the Krull dimension of $T$ is smaller than the Krull dimension of $X$ and thus by induction $z^*A = i_! B$ for some dualisable $B$ and $i: \coprod_j \Spec k \to T$. But then of course $A \cong z_! i_! B$.\par
Now let $f: X \to \Spec k$ be an arbitrary finite type morphism. Note that the statement is local on $X$, so we may assume that $X$ is affine. Then $f$ factors over a finite morphism $s: X \to \mb A^n_k$ by Noether normalisation. The object $s_! A$ is $\mb A^n_k \to \Spec k$ smooth by \Cref{lmm:sm obj composition}, so $s_! A \cong i_! B$ for some dualisable $B$ and suitable $i: Z\to \mb A^n_k$ by the proven case. Letting $j$ be the inclusion of the complement of $Z$ and $j'$ its base change along $s$, we see that $s'_! j'^* A \cong j^* s_! A \cong j^* i_! B = 0$ where $s'$ is the base change of $s$ along $j$. This implies $j'^* A = 0$.\footnote{for a finite morphism $z$ and a sheaf $\mc F$, the counit $z^*z_* \mc F \to \mc F$ (on the abelian level) is surjective as can be seen by looking at stalks and using proper base change; as $z_*$ is exact for a finite morphism, this implies the statement.} Thus, $A$ is supported on some finite set of closed points and by applying \Cref{prop:sm with closed support} the statement follows from the base case of the induction (alternatively, use that $A$ is dualisable on a stratification).
\end{proof}

Now we can generalise this to arbitrary $k$ of characteristic $p$.

\begin{proposition}\label{prop:sm object char absolute gen}
    Let $k$ be a field of characteristic $p$ and let $f: X \to \Spec k$ be a separated morphism of finite type. Then the $f$-smooth objects with respect to $\mc D_p$ are precisely the scyscraper sheaves $i_! B$ where $i: Z \to X$ is the inclusion of some finite set $Z = \coprod_i \Spec k$ of closed points and $B \in \mc D(Z)$ is dualisable.
\end{proposition}
\begin{proof}
    Again, it follows from \Cref{lmm:sm obj composition} that all $i_! B$ of this form are $f$-smooth: $Z \to \Spec k$ is cohomologically étale and so all dualisable objects are $Z \to \Spec k$-smooth by \Cref{cor:dualizable+smooth}.\par
    For the converse, let again $\bar k$ denote the algebraic closure of $k$ and let $g: \Spec \bar k \to \Spec k$ be the corresponding morphism. Assume $A$ is $f$-smooth. Let $g'$ denote the base change of $g$ along $f$. Then $g'^*A$ is $g^*f$-smooth by stability of smooth objects under base change. Thus, its support $Z'$ is a finite number of closed points by \Cref{prop:sm object char absolute} and so $Z' \to \Spec \bar{k}$ is cohomologically étale. The map $g^* X \to X$ is topologically a quotient map (as it is faithfully flat) and $Z'$ is topologically the preimage of the support $Z$ of $A$. So $Z$ is closed and we can equip it with the reduced subscheme structure. We conclude that $Z \to \Spec k$ is cohomologically étale by \Cref{cor:coh et change of base}. It follows that $A|_Z$ is dualisable as it is smooth for the cohomologically étale morphism $Z \to \Spec k$ by \Cref{prop:sm with closed support}, cf. \Cref{rem:unram}.
\end{proof}

\subsection{The relative case}

Let us again assume that $k$ is algebraically closed. Moreover, we restrict ourselves to the category $\mc C$ of schemes which are separated and of finite type over $\Spec k$. Then for a morphism $f$ in $\mc C$ all $f$-smooth objects are perfect-constructible (\Cref{prop:sm obj perfect-constructible}) and in particular have locally closed support (cf. \Cref{rem:dualizable-perfect}).\par
A simple consequence of \Cref{prop:sm object char absolute} is the following:
\begin{theorem}\label{cor:supp of sm coh et}
    Let $f:Y \to X$ be a morphism in $\mc C$ and $A$ an $f$-smooth object. Let $Z$ be the support of $A$. Then $A|_Z$ is $f|_Z$-smooth and $Z \to X$ is quasi-finite.
\end{theorem}
\begin{proof}
    The inclusion $Z \subseteq Y$ factors as $Z \to U \to Y$ where $Z \subseteq U$ is a closed subscheme and $U \subseteq Y$ is an open subscheme. $A|_U$ is $f|_U$-smooth by \Cref{lmm:sm obj composition} and supported on $Z$ and thus $A|_Z$ is $f|_Z$-smooth by \Cref{prop:sm with closed support}. Then by stability of smooth objects under base change and \Cref{prop:sm object char absolute} it follows that $Z \to X$ is quasi-finite.
\end{proof}

We can now employ our observation in \Cref{cor:coh et on stratum} to prove:

\begin{proposition}\label{prop:some sm observation}
    Let $f: Y \to X$ be a morphism in $\mc C$, $A$ an $f$-smooth object and $Z$ its support. Then there is a (finite) decomposition $\coprod_I X_i \to X$ of $X$ into locally closed subsets such that
    \begin{enumerate}[label=\roman*)]
        \item $f|_{Z_i}: Z_i \to X_i$ is cohomologically étale and in particular
        \item $A|_{Z_i}$ is dualisable and $\mb D_f(A)|_{Z_i} \cong \iHom(A|_{Z_i},\mb F_p)$, 
    \end{enumerate}
    where $Z_i:=Z \cap f^{-1}(X_i)$.
\end{proposition}
\begin{proof}
    Let $Z$ be the support of $A$. Then by \Cref{cor:supp of sm coh et} $Z \to X$ is quasifinite and $A|_Z$ is $f|_Z$-smooth. From \Cref{cor:coh et on stratum} we get a decomposition $\coprod_I X_i \to X$ into locally closed subsets such that $f|_{Z_i}: Z_i \to X_i$ is cohomologically étale. We conclude by \Cref{rem:unram} since the formation of $\mb D_{f|_Z}(A|_Z)$ commutes with base change by \Cref{cor:coh sm generalization}.
\end{proof}

On the other hand, we know that the following class of objects are smooth for a morphism in $\mc C$ for formal reasons:

\begin{proposition}\label{prop:a class of sm}
    Let $f: Y \to X$ be a morphism in $\mc C$. Then all objects $A \in \mc D_p(Y)$ with closed support $Z$ such that $A|_Z$ is dualisable and $f|_Z$ is cohomologically étale are $f$-smooth.
\end{proposition}
\begin{proof}
    This follows from \Cref{lmm:sm obj composition} and \Cref{rem:unram}.
\end{proof}

However, stability of smooth objects under direct sums implies that there are smooth objects which are not of the form in \Cref{prop:a class of sm}.

\begin{example}\label{xmpl:sm-sums}
   Let $k$ be a field of characteristic $p \not=2$ and consider the projection $f: \mb A^2 \to \mb A^1$ to the first coordinate and the two sections $s_1: \mb A^1 \to \mb A^2, (a) \mapsto (a,a)$ and $s_2: \mb A^1 \to \mb A^2, (a) \mapsto (a,-a)$. As they are closed immersions, $s_{1,!} \mb F_p$ and $s_{2,!} \mb F_p$ are both $f$-smooth by \Cref{lmm:sm obj composition} and \Cref{cor:id sm is dualizable} and thus so is their direct sum by \Cref{cor:sm distinguished} which is supported on $X:=\Spec k[x,y]/(x^2-y^2)$. We claim that the projection $g: X \to \mb A^1$ is not cohomologically étale. By \Cref{cor:coh-unramified} it suffices to show that $g$ is not topologically unramified. As $\Delta_g$ is certainly not surjective ($g$ is not injective), it suffices to show that $X \times_{\mb A^1} X$ is connected. Indeed, we calculate $X \times_{\mb A^1} X \cong \Spec k[x,y,z]/(x^2-y^2,x^2-z^2)$ whose irreducible components are four lines which intersect at the origin.\par
   As a sanity check, note that picking $\mb A^1-\set{0}$ and $\set{0}$ as a stratification of $\mb A^1$ validates \Cref{prop:some sm observation} in this example.
\end{example}

We pose the following question:
\begin{pquestion}\label{question:rel smooth}
    Let $f: Y \to X$ be a morphism in $\mc C$. Are $f$-smooth objects generated by smooth objects of the form in \Cref{prop:a class of sm} under direct sums and cones?
\end{pquestion}

\newpage 
\begin{FlushLeft}
\bibliography{refs}
\end{FlushLeft}
\end{document}